\documentclass[10pt]{amsart}
\usepackage[a4paper,margin=22.7mm]{geometry}
\usepackage[T1]{fontenc}
\usepackage[utf8]{inputenc}
\usepackage{lmodern}
\usepackage{amsmath,amssymb,amsthm,mathtools}
\usepackage{enumitem}
\usepackage{xcolor}
\usepackage[colorlinks=true,linkcolor=blue,citecolor=blue,urlcolor=blue]{hyperref}
\usepackage{aliascnt}
\usepackage[nameinlink,capitalise,noabbrev]{cleveref}

\setlist[itemize]{leftmargin=2em}
\setlist[enumerate]{leftmargin=2em}

\newtheorem{theorem}{Theorem}[section]
\newaliascnt{proposition}{theorem}
\newtheorem{proposition}[proposition]{Proposition}
\aliascntresetthe{proposition}
\newaliascnt{lemma}{theorem}
\newtheorem{lemma}[lemma]{Lemma}
\aliascntresetthe{lemma}
\newaliascnt{corollary}{theorem}
\newtheorem{corollary}[corollary]{Corollary}
\aliascntresetthe{corollary}
\newaliascnt{claim}{theorem}

\aliascntresetthe{claim}
\newaliascnt{conjecture}{theorem}

\aliascntresetthe{conjecture}
\theoremstyle{definition}
\newaliascnt{definition}{theorem}
\newtheorem{definition}[definition]{Definition}
\aliascntresetthe{definition}
\newaliascnt{example}{theorem}
\newtheorem{example}[example]{{Example}}
\aliascntresetthe{example}
\newaliascnt{remark}{theorem}
\newtheorem{remark}[remark]{Remark}
\aliascntresetthe{remark}
\newtheorem*{unnumberedremark}{Remark}
\newtheorem*{unnumbereddefinition}{Definition}
\newaliascnt{assumption}{theorem}

\aliascntresetthe{assumption}

\crefname{theorem}{Theorem}{Theorems}
\crefname{proposition}{Proposition}{Propositions}
\crefname{lemma}{Lemma}{Lemmas}
\crefname{corollary}{Corollary}{Corollaries}
\crefname{claim}{Claim}{Claims}
\crefname{conjecture}{Conjecture}{Conjectures}
\crefname{definition}{Definition}{Definitions}
\crefname{example}{Example}{Examples}
\crefname{remark}{Remark}{Remarks}
\crefname{assumption}{Assumption}{Assumptions}
\Crefname{theorem}{Theorem}{Theorems}
\Crefname{proposition}{Proposition}{Propositions}
\Crefname{lemma}{Lemma}{Lemmas}
\Crefname{corollary}{Corollary}{Corollaries}
\Crefname{claim}{Claim}{Claims}
\Crefname{conjecture}{Conjecture}{Conjectures}
\Crefname{definition}{Definition}{Definitions}
\Crefname{remark}{Remark}{Remarks}
\Crefname{assumption}{Assumption}{Assumptions}

\newcommand{\Z}{\mathbb Z}

\newcommand{\id}{\mathrm{id}}
\newcommand{\Stab}{\mathrm{Stab}}
\newcommand{\ord}{\mathrm{ord}}

\newcommand{\supp}{\mathrm{supp}}

\newcommand{\latestchg}[1]{{#1}}
\newcommand{\journalchg}[1]{{#1}}
\newcommand{\PathVar}{{\mathcal P}}
\newcommand{\PathExtVar}{{\widehat{\mathcal P}}}
\newcommand{\Wnorm}{{\lVert\mathbf W\rVert_1}}
\newcommand{\SumC}{{\Sigma_C}}
\newcommand{\SumR}{{\Sigma_R}}
\newcommand{\BarSum}{{\overline\Sigma}}

\newcommand{\PoolBad}{{\mathcal E_{\rm pool}}}
\newcommand{\BlockDatum}{{\mathfrak B}}
\newcommand{\RepairPos}{{z}}

\title[Kneserized Anticoncentration and Reverse Absorption]{Kneserized Anticoncentration and Reverse Absorption
for Graham's Rearrangement Conjecture}

\author{Simone Costa}
\address{\latestchg{DICATAM, Universit\`a degli Studi di Brescia, Via Branze 43, 25123 Brescia, Italy}}

\author{Stefano Della Fiore}
\address{\latestchg{DII, Universit\`a degli Studi di Brescia, Via Branze 38, 25123 Brescia, Italy}}

\author{Tao Feng}
\address{\latestchg{School of Mathematics and Statistics, Beijing Jiaotong University, Beijing, 100044, PR China}}

\author{Hengrui Liu}
\address{\latestchg{School of Mathematics and Statistics, Beijing Jiaotong University, Beijing, 100044, PR China}}

\date{}

\begin{document}

{
\begin{abstract}
{
We establish a Kneser-based anticoncentration estimate for uniform
subset sums in composite cyclic groups.  The estimate contains a
periodic loss and is weaker than its prime-modulus counterpart.
Nevertheless, together with known small- and large-set results, it
proves that, for every fixed $t\geq2$ such that $\mathbb Z_t$ is
strongly sequenceable and every sufficiently large prime $p$, every
subset of $\mathbb Z_{tp}\setminus\{0\}$ has a valid ordering, thus
establishing the analogue of Graham's rearrangement conjecture for
this family of composite cyclic groups.

We then identify the structural source of this loss.  An inverse
theorem shows that failure of the stabilizer-free growth underlying
prime-type anticoncentration forces almost all of the set into a
proper subgroup or one of its cosets.  We exploit this structure by
\emph{reverse absorption}.  Iterating the resulting dichotomy between
non-periodic anticoncentration and structured concentration proves
that every subset of
\[
        \mathbb Z_k\setminus\{0\},
        \qquad
        k=\prod_{i=1}^{s}p_i^{e_i},
        \qquad
        \sum_{i=1}^{s}e_i\leq L,
        \qquad
        p_1<\cdots<p_s\leq\gamma p_1,
\]
admits a valid ordering whenever $L$ and $\gamma$ are fixed and the
primes $p_i$ are sufficiently large.
}
\end{abstract}}

\maketitle

\latestchg{anticoncentration}, Kneser's theorem, sequenceability,
local repair, reverse absorption.

\section{Introduction}

Let $G$ be a finite group. A sequence $s_1,\ldots,s_n$ of distinct elements of $G\setminus\{\id\}$ is called a \emph{valid ordering} of the set $S=\{s_1,\ldots,s_n\}$ if the partial products
\[
s_1,\quad s_1s_2,\quad \ldots,\quad s_1s_2\cdots s_n
\]
are pairwise distinct. In additive notation, this means that the partial sums
\[
s_1,\quad s_1+s_2,\quad \ldots,\quad s_1+\cdots+s_n
\]
are pairwise distinct.

Graham's rearrangement conjecture, posed by Graham \cite{Graham} and later reiterated by Erd\H{o}s and Graham \cite{ErdosGraham}, asserts that, for every prime $p$, every subset of $\mathbb Z_p\setminus\{0\}$ admits a valid ordering.

A broader and stronger conjecture concerns sequencings in arbitrary abelian groups. A valid ordering $s_1,\ldots,s_n$ is called a \emph{sequencing} if, after adjoining the initial value $0$ to the list of partial sums
\[
0,\quad s_1,\quad s_1+s_2,\quad\ldots,\quad s_1+\cdots+s_n,
\]
the only possible repetition is between the first and the last terms. An abelian group $G$ is \emph{strongly sequenceable} if every subset of $G\setminus\{0\}$ admits a sequencing. A long-standing conjecture of Alspach asserts that every abelian group is strongly sequenceable. A published formulation appears in \cite{AlspachLiversidge}; before that, the conjecture had circulated for many years in unpublished form and was discussed, for instance, in \cite{CostaMoriniPasottiPellegrini}.

{
This subset-sequencing problem is related to the classical sequencing
problems for whole groups. A finite group $G$ of order $n$ is called \emph{sequenceable} if
its elements can be ordered $g_0,g_1,\ldots,g_{n-1}$ so that the
successive partial products
\[
g_0,\quad
g_0g_1,\quad
\ldots,\quad
g_0g_1\cdots g_{n-1}
\]
are pairwise distinct. For finite { nontrivial} abelian groups, Gordon's theorem \cite{Gordon}
characterizes the sequenceable groups as precisely those having a
unique involution.

The related notion of an $R$-sequencing can be expressed in the same
partial-product language: the non-identity elements can be ordered
$a_1,\ldots,a_{n-1}$ so that
\[
\mathrm{id},\quad
a_1,\quad
a_1a_2,\quad
\ldots,\quad
a_1a_2\cdots a_{n-2}
\]
are pairwise distinct, while
\[
a_1a_2\cdots a_{n-1}=\mathrm{id}.
\]
The Friedlander--Gordon--Miller conjecture~\cite{FriedlanderGordonMiller} predicted that
every finite abelian group not covered by Gordon's theorem admits an
$R$-sequencing, and was proved by Alspach, Kreher, and Pastine~\cite{AlspachKreherPastine}.
A related complete mapping problem is the
Friedlander--Gordon--Tannenbaum conjecture~\cite{FriedlanderGordonTannenbaum}, recently solved,
for sufficiently large abelian groups, by Müyesser~\cite{MuyesserCycleType}.
We refer to Ollis's survey~\cite{OllisSurvey} for the classical group-sequencing
and $R$-sequencing literature.}

The present paper studies Graham's rearrangement problem for cyclic groups of composite order. More precisely, we ask whether every subset of $\mathbb Z_k\setminus\{0\}$ admits a valid ordering when $k$ is composite. This is the cyclic-group case of the broader problem suggested by Alspach's conjecture.

We first recall the relevant progress on these problems. For arbitrary finite abelian groups, with no restriction on the order of the group, Alspach's conjecture was previously proven to be true for every subset
\[
S\subseteq G\setminus\{0\}
\qquad\text{with}\qquad
|S|\leq9.
\]
Recent work improves this to $|S|\leq20$, and to $|S|\leq22$ when $S$ has sum $0$ \cite{CostaDellaFioreFontanaVena}.

For cyclic groups, substantially stronger bounds for $|S|$ are known when the least prime divisor of the group order is large. For a positive integer $m$, write
\[
P^-(m)=\min\{\ell:\ell\text{ is prime and }\ell\mid m\},
\]
with the convention $P^-(1)=\infty$; accordingly, $(P^-(1))^{-\delta}=0$ for every $\delta>0$.

For prime modulus $p$, Kravitz \cite{KravitzSmall} proved Graham's conjecture in the range
\[
|S|\leq \frac{\log p}{\log\log p},
\]
{where, here and throughout the paper, $\log$ denotes the natural logarithm unless a base is explicitly indicated.} Bedert and Kravitz \cite{BedertKravitz} subsequently extended it, for sufficiently large $p$, to
\[
|S|\leq \exp\!\bigl(c(\log p)^{1/4}\bigr)
\]
for some absolute constant $c>0$. Their rectification framework extends to general cyclic moduli, with $p$ replaced by the least prime divisor of the modulus. More recently, Costa and Della Fiore obtained directly for $\mathbb Z_k$ the stronger bound
\[
|S|\leq \exp\!\bigl(c(\log P^-(k))^{1/3}\bigr)
\]
for some absolute constant $c>0$ \cite{CostaDellaFioreOneShot}.

At the opposite end of the size spectrum, Bedert--Buci\'c--Kravitz--Montgomery--M\"uyesser proved a large-set theorem that holds for every finite group $G$, with no largeness assumption on $|G|$ \cite{BBKMM}.

\begin{theorem}[Large sets, Bedert--Buci\'c--Kravitz--Montgomery--M\"uyesser \cite{BBKMM}]
\label{thm:large-known}
There exists an absolute constant $c>0$ such that, for every finite group $G$, every subset
\[
S\subseteq G\setminus\{\id\}
\qquad\text{with}\qquad
|S|\geq |G|^{1-c}
\]
has a valid ordering.
\end{theorem}

For prime cyclic groups, Pham and Sauermann recently bridged the remaining intermediate range \cite{PhamSauermann}.

\begin{theorem}[Pham--Sauermann \cite{PhamSauermann}]
\label{thm:pham-sauermann-known}
For every $0<\alpha<1$, there exists a constant $C_\alpha>0$ such that the following holds. Let $p$ be a prime and let
\[
S\subseteq\mathbb Z_p\setminus\{0\}
\qquad\text{satisfy}\qquad
C_\alpha\leq |S|\leq p^{1-\alpha}.
\]
Then $S$ has a valid ordering.
\end{theorem}

Let $c>0$ be the constant in Theorem~\ref{thm:large-known}. Choosing $0<\alpha<c$, the Pham--Sauermann range overlaps with the large-set range. Moreover, for all sufficiently large $p$, the small-set results discussed above cover the bounded range $|S|<C_\alpha$. Consequently, Graham's rearrangement conjecture holds for all sufficiently large primes.

For a general cyclic group $\mathbb Z_k$, however, the analogous problem remains open. The known small- and large-set results leave a substantial intermediate gap, and the Pham--Sauermann argument does not transfer formally: whereas Cauchy--Davenport supplies uniform growth in $\mathbb Z_p$, Kneser's theorem permits a non-trivial stabilizer. The objective of this paper is to push the analogue of Graham's rearrangement conjecture as far as possible for composite cyclic moduli. Our principal result closes the gap when the modulus has a bounded number of sufficiently large comparable prime factors, counted with multiplicity.

\begin{theorem}[Comparable bounded prime-power moduli]
\label{thm:comparable-prime-power}
Fix $L\geq2$ and $\gamma\geq1$. There exists $p_0=p_0(L,\gamma)$ such that the following holds. Let
\[
k=\prod_{i=1}^s p_i^{e_i},
\qquad
e_1+\cdots+e_s\leq L,
\qquad
p_0\leq p_1<\cdots<p_s\leq\gamma p_1,
\]
where the $p_i$ are distinct primes and the exponents $e_i$ are positive integers. Then every subset $S\subseteq \mathbb Z_k\setminus\{0\}$ has a valid ordering.
\end{theorem}

\subsection{Intermediate results for the main theorem}

We next collect the intermediate results that lead to Theorem~\ref{thm:comparable-prime-power}. Stating these results separately allows us to isolate the main mechanisms of the proof before assembling the recursive argument for the full theorem. They fall naturally into two groups, corresponding to Kneserized anticoncentration and inverse structure. The first concerns a Kneserized version of the Pham--Sauermann anticoncentration and local-repair argument and yields a polynomial-range theorem for cyclic groups of composite order. The second describes the structural obstruction created by non-trivial stabilizers and leads to the reverse-absorption mechanism used to overcome it. \journalchg{The full theorem is obtained by iterating the resulting dichotomy through the bounded number of prime-power layers allowed in its statement.}

We begin with the anticoncentration component. The analytic starting point is an anticoncentration estimate in $\mathbb Z_{tq}$, where $t$ is bounded, $(t,q)=1$, and the least prime divisor of $q$ is large. We replace the Cauchy--Davenport growth step in the Boolean-slice argument of Pham and Sauermann by Kneser's theorem. A non-trivial stabilizer produces an additional periodic loss and hence a weaker polynomial range. This gives the following composite-modulus analogue of the Pham--Sauermann polynomial-range theorem.

\begin{theorem}[Cyclic polynomial range]
\label{thm:cyclic-small}
Fix $T\geq1$ and $0<\alpha<1$. There exists a constant
\[
C=C(\alpha,T)
\]
such that the following holds. Let
\[
k=tq,
\qquad
1\leq t\leq T,
\qquad
(t,q)=1,
\]
and put
\[
\lambda=P^-(q).
\]
If
\[
S\subseteq\mathbb Z_k\setminus\{0\},
\qquad
C\leq |S|\leq\lambda^{1-\alpha},
\]
then $S$ has a valid ordering.
\end{theorem}

In Theorem~\ref{thm:cyclic-small}, when $t=1$, it specializes to the Kneserized Boolean-slice argument. For bounded $t$, the only additional observation is that a periodic stabilizer whose order does not contain a prime factor of $q$ has order at most $t$.

When $q=p$ is prime and $t$ is fixed, Theorem~\ref{thm:cyclic-small} can be combined with Theorem~\ref{thm:large-known} and a direct-product form of Kravitz's logarithmic result. We record the latter in the form needed here.

\begin{theorem}[Logarithmic direct products, Kravitz {\cite[Remarks~3--5]{KravitzSmall}}]
\label{thm:kravitz-direct}
Let $H$ be a fixed strongly sequenceable abelian group. Then, as
$p\to\infty$ over the primes, every subset
\[
        S\subseteq(H\times\mathbb Z_p)\setminus\{0\}
\]
with
\[
        |S|
        \leq
        (1+o_{p\to\infty}(1))
        \frac{\log p}{\log\log p}
\]
has a valid ordering. 
{More precisely, writing
$\pi_{\mathbb Z_p}\colon H\times\mathbb Z_p\to\mathbb Z_p$
for the canonical projection onto the second factor, if
\[
\bigl|\pi_{\mathbb Z_p}(S)\bigr|
\leq
\frac{\log p}{\log\log p},
\]
then $S$ has a sequencing.}
\end{theorem}

The result of \cite{CostaDellaFioreFontanaVena} also gives an explicit family of strongly sequenceable groups. Indeed, if $2\leq t\leq21$, then every subset of $\mathbb Z_t\setminus\{0\}$ has size at most $20$ and hence admits a sequencing by \cite{CostaDellaFioreFontanaVena}. Thus $\mathbb Z_t$ is strongly sequenceable for every $2\leq t\leq21$. Combining this observation with the results above gives the following consequence.

\begin{corollary}[Almost-prime consequences]
\label{cor:prime-cases}
There exists $p_0$ such that, for every integer $2\leq t\leq21$ and every prime $p\geq p_0$ with $p\nmid t$, every subset of $\mathbb Z_{tp}\setminus\{0\}$ has a valid ordering. More generally, if $t\geq2$ is fixed and $\mathbb Z_t$ is strongly sequenceable, then the same conclusion holds for every sufficiently large prime $p\nmid t$.
\end{corollary}

Indeed, let $c$ be the constant in Theorem~\ref{thm:large-known}. For $0<\alpha<c$, the polynomial range and the large-set range overlap whenever
\begin{equation}
\label{eq:overlap}
(P^-(q))^{1-\alpha}\geq (tq)^{1-c}.
\end{equation}
When $q=p$ is prime and $t$ is fixed, this inequality holds for all sufficiently large $p$. The bounded cardinalities below the threshold in Theorem~\ref{thm:cyclic-small} are covered by Theorem~\ref{thm:kravitz-direct}, provided that $\mathbb Z_t$ is strongly sequenceable.

We now turn to the second group of intermediate results, which addresses the periodic obstruction responsible for the loss in the Kneserized anticoncentration argument. An inverse theorem shows that failure of the stabilizer-free growth needed for prime-type anticoncentration forces almost all of the set into a proper subgroup or one of its cosets. We treat this structured alternative by \emph{reverse absorption}: the exceptional elements are consumed first, while the regular bulk is preserved for the final completion. This reverse-absorption viewpoint is related to reservoir and greedy-completion methods in weak sequenceability \cite{CostaWeakMultiset,CD1} and, more broadly, to absorption-based constructions in group-theoretic combinatorics \cite{BBKMM,MuyesserCycleType,MuyesserPokrovskiy}. The inverse theorem is the conceptual form of the obstruction; the applications use the quantitative, conditioned chain version developed in Section~\ref{sec:kneser-anticoncentration}.

In the structured branches of the argument, we will also use the following bounded-cardinality input.

\begin{theorem}[Small sets by weak Freiman rectification, Costa--Della Fiore \cite{CostaDellaFioreOneShot}]
\label{thm:cdf-known}
Let $s\geq1$. If every prime factor of $m$ is larger than $s!/2$, then every $s$-element subset of $\mathbb Z_m\setminus\{0\}$ has a valid ordering.
\end{theorem}

\journalchg{This theorem covers the bounded-cardinality cases that arise in the structured argument of Section~\ref{sec:bounded-prime-power}.}

\journalchg{We now carry out the recursive argument for the general moduli covered by Theorem~\ref{thm:comparable-prime-power}, iterating the dichotomy between non-periodic anticoncentration and structured concentration through the prime-power layers of the modulus.}

\subsection{Organization of the paper}

\journalchg{Section~2 collects the additive preliminaries. Section~3 develops the one-slice and interval-chain estimates, first in their ordinary form and then in the quantitative non-periodic form, and isolates the corresponding inverse theorem. Section~4 abstracts the Pham--Sauermann local-repair argument and proves the polynomial-range theorem. Section~5 iterates the dichotomy between non-periodic anticoncentration and structured concentration for moduli with a bounded number of prime factors, counted with multiplicity, and proves Theorem~\ref{thm:comparable-prime-power}. Section~6 contains the proof of the layered local-repair lemma, and Section~7 concludes with a possible extension of the method beyond cyclic groups.}

\section{Preliminaries}

We record the notation for Boolean-slice sums. If $A$ is a finite subset of an abelian group, and $R\subseteq A$, write
\[
        \Sigma(R)=\sum_{x\in R}x.
\]
For a positive integer $m$, a \emph{uniform $m$-slice} of $A$ is a uniformly random subset $R\subseteq A$ of size $m$.

For every positive integer $M$, we use characters of $\Z_M$ in the
form
\[
        e_M(x)=\exp(2\pi i x/M).
\]
For $x\in\Z_M$ define
\[
        \|x\|_M=\min_{a\in\Z}|x/M-a|.
\]
The elementary inequalities { (cf. \cite[Section 2]{PhamSauermann}) }
\[
        \mathrm{Re}\, e_M(x)\le 1-2\|x\|_M^2,
        \qquad
        \|x_1+\cdots+x_r\|_M^2
        \le r\sum_{i=1}^r\|x_i\|_M^2
\]
are used throughout.

{
We shall also use the standard lower-tail Chernoff bound for
hypergeometric random variables. If $X$ counts the marked elements in
a uniformly random $k$-subset of a finite set and $\mu=\mathbb E [X]$,
then, for every $0<\delta<1$,
\[
        \Pr\bigl[X\leq(1-\delta)\mu\bigr]
        \leq
        \exp\left(-\frac{\delta^2\mu}{2}\right).
\]
In particular,
\[
        \Pr[X\leq\mu/2]\leq e^{-\mu/8}.
\]
}

{
If $\mathcal G\subseteq\mathcal F$ are $\sigma$-algebras and $Y$ is
integrable, then
\[
 \mathbb E\!\left[\mathbb E[Y\mid\mathcal F]\mid\mathcal G\right]
 =\mathbb E[Y\mid\mathcal G].
\]
We refer to this identity, and to its event-probability specialization,
as the \emph{tower property}.  We also write $X\asymp Y$ when
$cY\leq X\leq CY$ for positive constants $c,C$ independent of the
parameters tending to infinity; a subscript indicates the parameters on
which these constants are allowed to depend.
}

We shall use Kneser's theorem { (see \cite[Theorem 5.5]{TaoVu})} in the following standard form.

\begin{theorem}[Kneser]\label{thm:kneser}
Let $A_1,\ldots,A_r$ be nonempty subsets of a finite abelian group $G$, and let
\[
        H=\Stab(A_1+\cdots +A_r)
        =\{h\in G:A_1+\cdots +A_r+h=A_1+\cdots +A_r\}.
\]
Then
\[
        |A_1+\cdots +A_r|
        \ge
        |A_1+H|+\cdots+|A_r+H|-(r-1)|H|.
\]
\end{theorem}

{
For finite subsets $A,B$ of an abelian group and a relation
$\mathcal S\subseteq A\times B$, write
\[
        A\stackrel{\mathcal S}{+}B
        :=\{a+b:(a,b)\in\mathcal S\}.
\]
For a set $A$, put
\[
        \widehat 2A:=\{a+b:a,b\in A,\ a\neq b\}.
\]
The only consequence of Huicochea's Kneser-type theorem needed below
is the following specialization.
}

\begin{theorem}[Restricted double-sum estimate]
\label{thm:huicochea-restricted}
Let $A$ be an $m$-element subset of a finite abelian group of odd order,
where $m\geq2$.  If
\[
        \widehat 2A\neq A+A,
\]
then
\begin{equation}
\label{eq:huicochea-specialized}
        |\widehat 2A|
        \geq 2m-2\sqrt{2m}-3.
\end{equation}
\end{theorem}

\begin{proof}
{
Let
\[
        \mathcal S:=\{(a,b)\in A\times A:a\neq b\}.
\]
Then $A\stackrel{\mathcal S}{+}A=\widehat 2A$, and exactly one pair is
excluded from each row and column of $A\times A$.  Moreover, if
$c\in A+A$ has at least two ordered representations, at most one can
be diagonal, because multiplication by $2$ is injective in a group of
odd order.  Thus the popular-sum hypothesis in Huicochea's theorem is
satisfied.  Since $\widehat 2A\neq A+A$, its quantitative
non-full-sumset alternative~\cite{Huicochea} gives
\[
        |\widehat 2A|>
        \left(1-\sqrt{\frac2m}\right)2m-3
        =2m-2\sqrt{2m}-3,
\]
which implies \eqref{eq:huicochea-specialized}.
}
\end{proof}

In the prime case $G=\Z_p$, the only stabilizers are trivial or all of
$G$, and this gives the Cauchy--Davenport growth used in
\cite{PhamSauermann}. In composite cyclic groups, the stabilizer may be
nontrivial. {This is the source of the additional
$1/P^-(q)$ term in the anticoncentration estimates of
Section~\ref{sec:kneser-anticoncentration}.}

\section{Kneserized anticoncentration and interval chains}
\label{sec:kneser-anticoncentration}

{
In this section we work directly in the cyclic group $\Z_{tq}$. We
establish the Kneserized one-slice estimate, the ordinary
interval-chain estimates used in the cyclic application, and the
quantitative non-periodic variants needed in the recursive argument.
}

Throughout this section, we write
\[
        N=tq,
        \qquad
        1\leq t\leq T,
        \qquad
        (t,q)=1,
        \qquad
        \lambda=P^-(q).
\]
{Here $P^-(q)$ denotes the least prime divisor of $q$, with the
convention $P^-(1)=\infty$; accordingly, $\lambda^{-\delta}=0$
when $q=1$ and $\delta>0$.}
The constants are allowed to depend on the fixed bound $T$ and on the
fixed number of equations under consideration.

{\subsection{One-slice anticoncentration in composite cyclic groups}

For a character $\chi\in\widehat{\mathbb Z_N}$, we use the same symbol $\chi\in\mathbb Z_N$ for the corresponding frequency, so that $\chi$ acts on $x\in\mathbb Z_N$ as $x\mapsto e_N(\chi x)$. Let $A\subseteq\mathbb Z_N$ have size $n$, and let $1\leq m\leq n$. Define
\begin{equation}
	\label{eq:deterministic-energy}
	\Psi_{A,m}(\chi):=\frac{m}{n^2}\sum_{x,y\in A}\|\chi(x-y)\|_N^2
\end{equation}
and, for $s>0$,
\begin{equation}
	\label{eq:low-energy-set}
	B_s(A,m):=\{\chi\in\mathbb Z_N:\Psi_{A,m}(\chi)\leq s\}.
\end{equation}
When $A$ and $m$ are fixed, we abbreviate $B_s(A,m)$ to $B_s$.

\begin{lemma}[Fourier reduction for a Boolean slice]
	\label{lem:boolean-slice-fourier-reduction}
	There exist absolute constants $C_0,C_{\mathrm E},C_{\mathrm F}\geq1$ and $c_0,c_{\mathrm F}\in(0,1)$ such that the following holds. Let $A\subseteq\mathbb Z_N$ have size $n\geq2$, let
	\[
	C_0\log(2+n)\leq m\leq c_0\frac{n}{\log(2+n)},
	\]
	and let $R$ be a uniformly random $m$-subset of $A$. Then, for every $z\in\mathbb Z_N$,
	\begin{equation}
		\label{eq:boolean-slice-fourier-reduction}
		\Pr[\Sigma(R)=z]\leq\frac1N+\frac{C_{\mathrm F}}{N}\sum_{\ell=0}^{\left\lfloor\log_2(m/C_{\mathrm E})\right\rfloor}e^{-c_{\mathrm F}2^\ell}\bigl(|B_{C_{\mathrm E}2^\ell}|-1\bigr)+3n^{-9}.
	\end{equation}
\end{lemma}

{The point of the proof is to verify that the Fourier reduction
of~\cite{PhamSauermann} is independent of the primality of the
modulus.  The genuinely additive input enters only afterward, through
the Kneser bound for the low-energy set.}

\begin{proof}
Choose $c_0>0$ so that
\[
c_0\leq\frac1{480}.
\]
Fix $C_{\mathrm E}:=2000$ and set
\[
c_{\mathrm F}:=\frac12\min\left\{1,\frac1{C_{\mathrm E}}\right\}.
\]
Choose
\[
C_{\mathrm F}\geq\max\left\{1,(1+e^{-1})e^{c_{\mathrm F}},\frac1{1-e^{-1}}\right\},
\]
and then choose $C_0$ sufficiently large that
\[
C_0\log4\geq C_{\mathrm E}
\qquad\text{and}\qquad
c_{\mathrm F}C_0\geq10+\log_2 C_{\mathrm F}.
\]
The lower bound $m\geq C_0\log(2+n)$ then implies $m\geq C_{\mathrm E}$. Since $n\geq2$ and $\log(2+n)\geq\log4$, the upper bound $m\leq c_0n/\log(2+n)$ also gives
\[
m\leq\frac{c_0n}{\log4}<\frac n4.
\]
Moreover,
\[
\frac{n}{48m}\geq\frac{\log(2+n)}{48c_0},
\]
and hence
\begin{equation}
\label{eq:partition-concentration-range}
n\exp\left(-\frac{n}{48m}\right)
\leq n(2+n)^{-1/(48c_0)}
\leq n^{1-1/(48c_0)}
\leq n^{-9},
\end{equation}
where the last inequality follows from $1/(48c_0)\geq10$.

	Choose a uniformly random ordered partition $\mathcal S=(A_1,\ldots,A_m)$ of $A$ whose part sizes differ by at most one, and then, conditional on $\mathcal S$, choose independent uniformly random elements $X_i\in A_i$. By permutation-invariance, $\{X_1,\ldots,X_m\}$ is a uniformly random $m$-subset of $A$.

We write $\Pr_{\mathcal S}$ and $\mathbb E_{\mathcal S}$ for probability and expectation with respect to the random partition $\mathcal S$. For a fixed realization of $\mathcal S$, we write $\Pr_X$ and $\mathbb E_X$ for probability and expectation with respect to the independent choices $X_1,\ldots,X_m$, conditional on $\mathcal S$. Unsubscripted $\Pr$ and $\mathbb E$ refer to the full two-stage random experiment.

For a fixed partition $\mathcal S$, Fourier inversion gives
\[
{\Pr}_X[X_1+\cdots+X_m=z]\leq\frac1N\sum_{\chi\in\mathbb Z_N}\prod_{i=1}^m\left|\mathbb E_X\!\left[e_N(\chi X_i)\right]\right|.
\]
For each $i$,
\[
\left|\mathbb E_X\!\left[e_N(\chi X_i)\right]\right|^2
=\frac1{|A_i|^2}\sum_{x,y\in A_i}\operatorname{Re}\!\left(e_N(\chi(x-y))\right)
\leq1-\frac2{|A_i|^2}\sum_{x,y\in A_i}\|\chi(x-y)\|_N^2.
\]
	Since the left-hand side is nonnegative and $(1-u)^{1/2}\leq e^{-u/2}$ for $0\leq u\leq1$, it follows that
	\begin{equation}
		\label{eq:conditional-fourier-bound}
		{\Pr}_X[X_1+\cdots+X_m=z]\leq\frac1N\sum_{\chi\in\mathbb Z_N}e^{-\psi(\chi)},
	\end{equation}
	where
	\begin{equation}
		\label{eq:random-energy}
		\psi(\chi):=\sum_{i=1}^m\frac1{|A_i|^2}\sum_{x,y\in A_i}\|\chi(x-y)\|_N^2.
	\end{equation}
	
	We next compare the random energy $\psi$ with the deterministic energy $\Psi_{A,m}$. Put
	\[
	L_{\mathrm E}:=\left\lfloor\log_2\left(\frac{m}{C_{\mathrm E}}\right)\right\rfloor.
	\]
	Since $m\geq C_{\mathrm E}$, one has $L_{\mathrm E}\geq0$. We claim that, for every integer $\ell$ with $0\leq\ell\leq L_{\mathrm E}$,
	\begin{equation}
\label{eq:energy-comparison}
\mathbb E_{\mathcal S}\!\left[\left|\{\chi\neq0:\psi(\chi)<2^{\ell+1}\}\right|\right]
\leq Nn^{-9}+|B_{C_{\mathrm E}2^\ell}|-1.
\end{equation}
	The argument is the one used in \cite[Lemmas~3.1--3.3]{PhamSauermann}; we give the details needed to verify that it remains valid for an arbitrary modulus $N$.
	
	Fix $0\leq\ell\leq L_{\mathrm E}$ and put $u:=2^\ell$. Let $\mathcal C_\ell$ be the set of nonzero $\chi\in\mathbb Z_N$ for which there exists $y_\chi\in\mathbb Z_N$ such that
	\[
	\left|\left\{x\in A:\|\chi x-y_\chi\|_N\leq8\sqrt{\frac{u}{m}}\right\}\right|\geq\frac{3n}{4}.
	\]
	Because $m\leq n/4$, every part satisfies $|A_i|\leq\sqrt2\,n/m$, and hence
	\begin{equation}
		\label{eq:random-energy-lower-bound}
		\psi(\chi)\geq\frac{m^2}{2n^2}\sum_{i=1}^m\sum_{x,y\in A_i}\|\chi(x-y)\|_N^2.
	\end{equation}
	Suppose first that $\chi\notin\mathcal C_\ell$. Fix $x'\in A$ and condition on the index of the part containing $x'$. The set of other elements in that part is a uniformly random $k$-subset of $A\setminus\{x'\}$, where
\[
k\geq\left\lfloor\frac{n}{m}\right\rfloor-1\geq\frac{2n}{3m};
\]
the last inequality follows from $n/m\geq\log4/c_0$. At least $n/4$ elements $x\in A$ satisfy $\|\chi(x-x')\|_N>8\sqrt{u/m}$. The hypergeometric Chernoff bound therefore shows that, except with probability at most
\[
e^{-k/32}\leq e^{-n/(48m)},
\]
at least
\[
\frac{k}{8}\geq\frac{n}{12m}>\frac{n}{16m}
\]
such elements lie in the same part as $x'$. By a union bound over $x'\in A$ and \eqref{eq:partition-concentration-range}, with probability at least $1-n^{-9}$ this holds simultaneously for every $x'$. On that event, \eqref{eq:random-energy-lower-bound} gives
	\[
	\psi(\chi)\geq\frac{m^2}{2n^2}\cdot n\cdot\frac{n}{16m}\cdot\frac{64u}{m}=2u=2^{\ell+1}.
	\]
	Thus
	\begin{equation}
		\label{eq:nonconcentrated-energy-tail}
		{\Pr}_{\mathcal S}[\psi(\chi)<2^{\ell+1}]\leq n^{-9}\qquad(\chi\notin\mathcal C_\ell).
	\end{equation}
	
	Now suppose that $\chi\in\mathcal C_\ell\setminus B_{C_{\mathrm E}u}$, and fix a corresponding center $y_\chi$. Put
	\[
	J_{\chi,\ell}:=\left\{x\in A:\|\chi x-y_\chi\|_N\leq16\sqrt{\frac{u}{m}}\right\}.
	\]
	Writing $d_x:=\|\chi x-y_\chi\|_N$, the inequality $\|\chi(x-x')\|_N^2\leq2d_x^2+2d_{x'}^2$ and the assumption $\Psi_{A,m}(\chi)>C_{\mathrm E}u=2000u$ imply
	\[
	\frac{2000u}{m}<\frac1{n^2}\sum_{x,x'\in A}\|\chi(x-x')\|_N^2\leq\frac4n\sum_{x\in A}d_x^2.
	\]
	Hence $\sum_{x\in A}d_x^2>500un/m$. Since $d_x^2\leq256u/m$ for $x\in J_{\chi,\ell}$,
	\begin{equation}
		\label{eq:outside-arc-energy}
		\sum_{x\in A\setminus J_{\chi,\ell}}d_x^2>\frac{244un}{m}>\frac{200un}{m}.
	\end{equation}
	Fix $x\in A\setminus J_{\chi,\ell}$ and condition on the index of its part. At least $3n/4$ elements $x'\in A$ satisfy $d_{x'}\leq8\sqrt{u/m}\leq d_x/2$. The hypergeometric Chernoff bound shows that, except with probability at most $e^{-n/(48m)}$, at least $n/(4m)$ such elements lie in the same part as $x$. For each of them,
	\[
	\|\chi(x-x')\|_N\geq d_x-d_{x'}\geq\frac{d_x}{2}.
	\]
	A union bound over $x\in A\setminus J_{\chi,\ell}$ and \eqref{eq:partition-concentration-range} shows that, with probability at least $1-n^{-9}$, these conclusions hold simultaneously for all such $x$. On this event, \eqref{eq:random-energy-lower-bound} and \eqref{eq:outside-arc-energy} give
	\[
	\psi(\chi)\geq\frac{m^2}{2n^2}\sum_{x\in A\setminus J_{\chi,\ell}}\frac{n}{4m}\left(\frac{d_x}{2}\right)^2=\frac{m}{32n}\sum_{x\in A\setminus J_{\chi,\ell}}d_x^2>\frac{25}{4}u>2u=2^{\ell+1}.
	\]
	Therefore
	\begin{equation}
		\label{eq:concentrated-energy-tail}
		{\Pr}_{\mathcal S}[\psi(\chi)<2^{\ell+1}]\leq n^{-9}\qquad(\chi\in\mathcal C_\ell\setminus B_{C_{\mathrm E}u}).
	\end{equation}
	The proofs of \eqref{eq:nonconcentrated-energy-tail} and \eqref{eq:concentrated-energy-tail} use only sampling without replacement and the metric properties of $\|\cdot\|_N$; no invertibility of $\chi$ modulo $N$ is used. Summing these bounds over all nonzero characters outside $B_{C_{\mathrm E}u}$ and counting the remaining nonzero characters trivially proves \eqref{eq:energy-comparison}.
	
	{We now decompose the characters into dyadic energy levels according to the size of $\psi(\chi)$.  {For this dyadic decomposition, put}}
	\[
	{\ell_{\max}:=\lfloor\log_2m\rfloor},
	\qquad
	\mathcal A_{<1}:=\{\chi\in\mathbb Z_N:0\leq\psi(\chi)<1\},
	\qquad
	\mathcal A_{<1}^{\ast}:=\mathcal A_{<1}\setminus\{0\},
	\]
	and, for $0\leq\ell\leq {\ell_{\max}}$,
	\[
	\mathcal A_\ell:=\{\chi\in\mathbb Z_N:2^\ell\leq\psi(\chi)<2^{\ell+1}\}.
	\]
	Since $\psi(0)=0$, the trivial character contributes exactly $1/N$. Moreover, $\psi(\chi)\leq m$ for every $\chi$, so $\mathcal A_{<1},\mathcal A_0,\ldots,\mathcal A_{\ell_{\max}}$ cover all characters. For $\chi\in\mathcal A_\ell$, one has $e^{-\psi(\chi)}\leq e^{-2^\ell}$. Taking expectations in \eqref{eq:conditional-fourier-bound} therefore gives
\begin{equation}
\label{eq:dyadic-fourier-decomposition}
\Pr[\Sigma(R)=z]
\leq\frac1N
+\frac1N\mathbb E_{\mathcal S}\!\left[|\mathcal A_{<1}^{\ast}|\right]
+\frac1N\sum_{\ell=0}^{{\ell_{\max}}}e^{-2^\ell}
\mathbb E_{\mathcal S}\!\left[|\mathcal A_\ell|\right].
\end{equation}
	Since
	\[
	\mathcal A_{<1}^{\ast}\subseteq\{\chi\neq0:\psi(\chi)<2\},
	\]
	the energy comparison \eqref{eq:energy-comparison} with $\ell=0$ gives
\begin{equation}
\label{eq:small-energy-shell}
\mathbb E_{\mathcal S}\!\left[|\mathcal A_{<1}^{\ast}|\right]
\leq Nn^{-9}+|B_{C_{\mathrm E}}|-1.
\end{equation}
	For every $0\leq\ell\leq L_{\mathrm E}$,
	\[
	\mathcal A_\ell\subseteq\{\chi\neq0:\psi(\chi)<2^{\ell+1}\},
	\]
	and hence
\begin{equation}
\label{eq:dyadic-energy-shell}
\mathbb E_{\mathcal S}\!\left[|\mathcal A_\ell|\right]
\leq Nn^{-9}+|B_{C_{\mathrm E}2^\ell}|-1.
\end{equation}
	For $L_{\mathrm E}<\ell\leq {\ell_{\max}}$, we use the trivial estimate $|\mathcal A_\ell|\leq N$. Substituting \eqref{eq:small-energy-shell} and \eqref{eq:dyadic-energy-shell} into \eqref{eq:dyadic-fourier-decomposition}, and writing
	\[
	b_\ell:=|B_{C_{\mathrm E}2^\ell}|-1\geq0,
	\]
	we obtain
	\begin{align}
		\Pr[\Sigma(R)=z]
		&\leq\frac1N+n^{-9}+\frac{b_0}{N}+n^{-9}\sum_{\ell=0}^{L_{\mathrm E}}e^{-2^\ell}+\frac1N\sum_{\ell=0}^{L_{\mathrm E}}e^{-2^\ell}b_\ell+\sum_{\ell=L_{\mathrm E}+1}^{{\ell_{\max}}}e^{-2^\ell}\notag\\
		&=\frac1N+n^{-9}\left(1+\sum_{\ell=0}^{L_{\mathrm E}}e^{-2^\ell}\right)+\frac1N\left((1+e^{-1})b_0+\sum_{\ell=1}^{L_{\mathrm E}}e^{-2^\ell}b_\ell\right)+\sum_{\ell=L_{\mathrm E}+1}^{{\ell_{\max}}}e^{-2^\ell}.
		\label{eq:dyadic-expanded-bound}
	\end{align}
	We estimate the three nontrivial terms in \eqref{eq:dyadic-expanded-bound} separately.
	
	First, since $2^\ell\geq\ell+1$ for every $\ell\geq0$,
	\[
	\sum_{\ell=0}^{L_{\mathrm E}}e^{-2^\ell}\leq\sum_{\ell=0}^{\infty}e^{-2^\ell}\leq\sum_{\ell=0}^{\infty}e^{-(\ell+1)}=\frac1{e-1}<1.
	\]
	Therefore
	\[
	n^{-9}\left(1+\sum_{\ell=0}^{L_{\mathrm E}}e^{-2^\ell}\right)<2n^{-9}.
	\]
	
	Second, by the choice of $c_{\mathrm F}$ and $C_{\mathrm F}$ at the beginning of the proof,
	\[
	0<c_{\mathrm F}\leq1,\qquad c_{\mathrm F}\leq\frac1{C_{\mathrm E}},\qquad C_{\mathrm F}e^{-c_{\mathrm F}}\geq1+e^{-1},\qquad C_{\mathrm F}\geq1.
	\]
	For $\ell=0$, this gives
	\[
	(1+e^{-1})b_0\leq C_{\mathrm F}e^{-c_{\mathrm F}}b_0.
	\]
	For every $\ell\geq1$, the inequality $c_{\mathrm F}\leq1$ gives
	\[
	e^{-2^\ell}b_\ell\leq e^{-c_{\mathrm F}2^\ell}b_\ell\leq C_{\mathrm F}e^{-c_{\mathrm F}2^\ell}b_\ell.
	\]
	Consequently,
	\[
	(1+e^{-1})b_0+\sum_{\ell=1}^{L_{\mathrm E}}e^{-2^\ell}b_\ell\leq C_{\mathrm F}\sum_{\ell=0}^{L_{\mathrm E}}e^{-c_{\mathrm F}2^\ell}b_\ell.
	\]
	
	Third, let
	\[
	u_\ast:=2^{L_{\mathrm E}+1}.
	\]
	By the definition of $L_{\mathrm E}$,
	\[
	u_\ast>\frac{m}{C_{\mathrm E}},
	\]
	and, since $m\geq C_{\mathrm E}$, one has $u_\ast\geq1$. Extending the finite tail to an infinite one and writing the successive dyadic values as $2^ju_\ast$, we obtain
	\[
	\sum_{\ell=L_{\mathrm E}+1}^{{\ell_{\max}}}e^{-2^\ell}\leq\sum_{j=0}^{\infty}e^{-2^ju_\ast}.
	\]
	Since $2^j\geq j+1$,
	\[
	\sum_{j=0}^{\infty}e^{-2^ju_\ast}\leq\sum_{j=0}^{\infty}e^{-(j+1)u_\ast}=\frac{e^{-u_\ast}}{1-e^{-u_\ast}}\leq\frac{e^{-m/C_{\mathrm E}}}{1-e^{-1}}.
	\]
	Because $c_{\mathrm F}\leq1/C_{\mathrm E}$,
	\[
	e^{-m/C_{\mathrm E}}\leq e^{-c_{\mathrm F}m}.
	\]
	By the choices of $C_{\mathrm F}$ and $C_0$,
\[
\sum_{\ell=L_{\mathrm E}+1}^{{\ell_{\max}}}e^{-2^\ell}\leq C_{\mathrm F}e^{-c_{\mathrm F}m}\leq n^{-9}.
\]
Indeed, since $m\geq C_0\log(2+n)$,
\[
C_{\mathrm F}e^{-c_{\mathrm F}m}
\leq C_{\mathrm F}(2+n)^{-c_{\mathrm F}C_0}
\leq C_{\mathrm F}n^{-c_{\mathrm F}C_0}
\leq n^{-9},
\]
where the first inequality uses $2+n\geq n$, and the last inequality follows from $c_{\mathrm F}C_0\geq10+\log_2 C_{\mathrm F}$ and $n\geq2$.
Applying the preceding three estimates to \eqref{eq:dyadic-expanded-bound} gives
	\[
	\Pr[\Sigma(R)=z]\leq\frac1N+\frac{C_{\mathrm F}}{N}\sum_{\ell=0}^{L_{\mathrm E}}e^{-c_{\mathrm F}2^\ell}\bigl(|B_{C_{\mathrm E}2^\ell}|-1\bigr)+3n^{-9},
	\]
	which is \eqref{eq:boolean-slice-fourier-reduction}.
\end{proof}

{
A subgroup $H<\mathbb Z_N$ will be called \emph{$q$-large} if $|H|$ is
divisible by a prime factor of $q$. If $H$ is not $q$-large, then
$|H|\mid t$ and hence $|H|\leq T$.

\begin{lemma}[Kneser bound for low-energy characters]
	\label{lem:kneser-low-energy}
	Fix $T\geq1$. There exist an absolute constant $c_{\mathrm K}>0$ and a constant $C_{\mathrm K}=C_{\mathrm K}(T)>0$ such that the following holds. Let $N=tq$, where $1\leq t\leq T$ and $(t,q)=1$, let $A\subseteq\mathbb Z_N$ have size $n$, and let $1\leq m\leq n$. For every $s$ with $1\leq s\leq c_{\mathrm K}m$,
	\begin{equation}
		\label{eq:kneser-low-energy}
		|B_s(A,m)|-1\leq C_{\mathrm K}\left(\frac{N\sqrt{s}}{n\sqrt{m}}+\frac{N}{\lambda}\right),
	\end{equation}
	where $N/\lambda$ is interpreted as $0$ when $q=1$.
\end{lemma}

\begin{proof}
	Fix
	\[
	\kappa:=\frac1{100},\qquad c_{\mathrm K}:=\frac{\kappa^2}{16}.
	\]
	If $c_{\mathrm K}m<1$, the asserted range of $s$ is empty, so assume $1\leq s\leq c_{\mathrm K}m$. Put
	\[
	Q_{s,\delta}:=\left\{x\in\mathbb Z_N:\sum_{\chi\in B_s}\|\chi x\|_N^2<\delta|B_s|\right\}.
	\]
	Let $Y,Y'$ be independent uniformly random elements of $A$. By \eqref{eq:deterministic-energy},
	\[
	\mathbb E\sum_{\chi\in B_s}\|\chi(Y-Y')\|_N^2=\frac1m\sum_{\chi\in B_s}\Psi_{A,m}(\chi)\leq\frac{s}{m}|B_s|.
	\]
	Markov's inequality and averaging over $Y'$ give an element $a_0\in A$ such that
	\begin{equation}
		\label{eq:Q-large}
		|\{a\in A:a-a_0\in Q_{s,10s/m}\}|\geq\frac{9n}{10}.
	\end{equation}
	In particular,
	\begin{equation}
		\label{eq:Q-cardinality-lower}
		|Q_{s,10s/m}|\geq\frac{9n}{10}.
	\end{equation}
	
	The set $B_s$ is symmetric and contains $0$. Hence, for every $x\in Q_{s,1/200}$,
	\[
	\left|\sum_{\chi\in B_s}e_N(\chi x)\right|=\sum_{\chi\in B_s}\cos(2\pi\chi x/N)\geq\sum_{\chi\in B_s}(1-20\|\chi x\|_N^2)\geq\frac9{10}|B_s|.
	\]
	Character orthogonality therefore gives
	\begin{equation}
		\label{eq:Q-parseval-N}
		|Q_{s,1/200}|\leq\frac{5N}{4|B_s|}.
	\end{equation}
	Moreover, for every positive integer $r$,
	\begin{equation}
		\label{eq:Q-growth-N}
		rQ_{s,\delta}\subseteq Q_{s,r^2\delta},
	\end{equation}
	because $\|x_1+\cdots+x_r\|_N^2\leq r\sum_{i=1}^r\|x_i\|_N^2$.
	
	Choose
	\[
	r:=\left\lfloor\kappa\sqrt{\frac{m}{s}}\right\rfloor.
	\]
	Since $s\leq c_{\mathrm K}m$, one has $\kappa\sqrt{m/s}\geq4$, and hence
	\[
	r\geq2,\qquad r\geq\frac{\kappa}{2}\sqrt{\frac{m}{s}},\qquad r^2\frac{10s}{m}\leq10\kappa^2<\frac1{200}.
	\]
	Put $Q:=Q_{s,10s/m}$ and $K:=\operatorname{Stab}(rQ)$. By \eqref{eq:Q-growth-N}, $rQ\subseteq Q_{s,1/200}$. Kneser's theorem gives
	\begin{equation}
		\label{eq:kneser-rQ}
		|rQ|\geq r|Q+K|-(r-1)|K|=r(|Q+K|-|K|)+|K|.
	\end{equation}
	
	Suppose first that $Q+K$ contains at least two $K$-cosets. Then $|Q+K|-|K|\geq|Q+K|/2\geq|Q|/2$, and therefore, by \eqref{eq:Q-cardinality-lower},
	\[
	|rQ|\geq\frac{r|Q|}{2}\geq\frac{9rn}{20}.
	\]
	Combining this with \eqref{eq:Q-parseval-N} gives
	\[
	|B_s|\leq\frac{25N}{9rn}\leq\frac{50}{9\kappa}\frac{N\sqrt{s}}{n\sqrt{m}}.
	\]
	
	Suppose next that $Q+K=K$. Then $Q\subseteq K$. Since $K=\operatorname{Stab}(rQ)$, the nonempty set $rQ\subseteq K$ is a union of $K$-cosets and hence equals $K$. Thus $K\subseteq Q_{s,1/200}$, and \eqref{eq:Q-parseval-N} gives
	\begin{equation}
		\label{eq:periodic-Bs-bound}
		|B_s|\leq\frac{5N}{4|K|}.
	\end{equation}
	If $K=\mathbb Z_N$, then $|B_s|\leq5/4$, and therefore $B_s=\{0\}$. Assume that $K$ is proper. If $|K|$ is divisible by a prime divisor of $q$, then $|K|\geq\lambda$, and \eqref{eq:periodic-Bs-bound} is at most $5N/(4\lambda)$. Otherwise $\gcd(|K|,q)=1$; since $|K|\mid N=tq$ and $(t,q)=1$, it follows that $|K|\mid t$ and hence $|K|\leq T$. Equation \eqref{eq:Q-cardinality-lower} then gives $n\leq10T/9$. In this bounded case, using $m\leq n$ and $s\geq1$,
	\[
	|B_s|-1\leq N\leq\left(\frac{10T}{9}\right)^{3/2}\frac{N\sqrt{s}}{n\sqrt{m}}.
	\]
	Choosing $C_{\mathrm K}(T)$ larger than the constants in the preceding cases proves \eqref{eq:kneser-low-energy}.
\end{proof}
}

\begin{proposition}[Boolean-slice anticoncentration in $\mathbb Z_{tq}$]
	\label{prop:one-colour}
	For every $0<\varepsilon<1$ and every $T\geq1$, there exists $C_{\varepsilon,T}>0$ such that the following holds. Let $N=tq$, where $1\leq t\leq T$ and $(t,q)=1$, let $A\subseteq\mathbb Z_N$ have size $n\geq2$, and let $R$ be a uniformly random $m$-subset of $A$, where $1\leq m\leq(1-\varepsilon)n$. Then, for every $z\in\mathbb Z_N$,
	\begin{equation}
		\label{eq:one-slice-N}
		\Pr[\Sigma(R)=z]\leq\frac1N+C_{\varepsilon,T}\left(\frac{\sqrt{\log(2+n)}}{n\sqrt{m}}+\frac1\lambda\right),
	\end{equation}
	where $1/\lambda$ is interpreted as $0$ when $q=1$.
\end{proposition}

\begin{proof}
	We divide the proof into three ranges.
	
	{\emph{Intermediate values of $m$.}} Suppose that
	\begin{equation}
		\label{eq:medium-slice-range}
		C_0\log(2+n)\leq m\leq c_0\frac{n}{\log(2+n)},
	\end{equation}
	where the constants are those of Lemma~\ref{lem:boolean-slice-fourier-reduction}. Put
	\[
	c_2:=\frac{c_{\mathrm K}}{C_{\mathrm E}},
	\qquad
	L_{\mathrm K}:=\lfloor\log_2(c_2m)\rfloor,
	\qquad
	L_{\mathrm E}:=\left\lfloor\log_2\left(\frac{m}{C_{\mathrm E}}\right)\right\rfloor.
	\]
	After increasing $C_0$, we may assume $c_2m\geq1$, so $0\leq L_{\mathrm K}\leq L_{\mathrm E}$. Define
	\[
	C_{\mathrm M}:=C_{\mathrm K}(T)\max\{\sqrt{C_{\mathrm E}},1\}.
	\]
	For every $0\leq\ell\leq L_{\mathrm K}$, one has $C_{\mathrm E}2^\ell\leq c_{\mathrm K}m$, and {Lemma}~\ref{lem:kneser-low-energy}, applied with $C_{\mathrm E}2^\ell$ in place of $s$, gives
	\[
	|B_{C_{\mathrm E}2^\ell}|-1\leq C_{\mathrm M}\left(\frac{N2^{\ell/2}}{n\sqrt{m}}+\frac{N}{\lambda}\right).
	\]
	Consequently,
	\begin{align}
		&\frac{C_{\mathrm F}}{N}\sum_{\ell=0}^{L_{\mathrm K}}e^{-c_{\mathrm F}2^\ell}\bigl(|B_{C_{\mathrm E}2^\ell}|-1\bigr)\notag\\
		&\qquad\leq C_{\mathrm F}C_{\mathrm M}\left(\frac1{n\sqrt m}\sum_{\ell=0}^{L_{\mathrm K}}e^{-c_{\mathrm F}2^\ell}2^{\ell/2}+\frac1\lambda\sum_{\ell=0}^{L_{\mathrm K}}e^{-c_{\mathrm F}2^\ell}\right).
		\label{eq:medium-low-energy-sum}
	\end{align}
	Both infinite series
	\[
	S_1:=\sum_{\ell=0}^{\infty}e^{-c_{\mathrm F}2^\ell}2^{\ell/2},
	\qquad
	S_0:=\sum_{\ell=0}^{\infty}e^{-c_{\mathrm F}2^\ell}
	\]
	converge. Indeed, for all sufficiently large $\ell$, one has $2^{\ell/2}\leq e^{c_{\mathrm F}2^\ell/2}$, so the tail of $S_1$ is dominated by $\sum_{\ell}e^{-c_{\mathrm F}2^\ell/2}$; this and the convergence of $S_0$ follow from $2^\ell\geq\ell+1$. Hence the right-hand side of \eqref{eq:medium-low-energy-sum} is at most
	\[
	C_T\left(\frac1{n\sqrt m}+\frac1\lambda\right)
	\]
	for a constant $C_T$ depending only on $T$.
	
	For $L_{\mathrm K}<\ell\leq L_{\mathrm E}$, use the trivial estimate $|B_{C_{\mathrm E}2^\ell}|-1\leq N$. Let $v:=2^{L_{\mathrm K}+1}$. Then $v>c_2m$ and $v\geq1$, so
	\begin{align*}
		\frac{C_{\mathrm F}}{N}\sum_{\ell=L_{\mathrm K}+1}^{L_{\mathrm E}}e^{-c_{\mathrm F}2^\ell}\bigl(|B_{C_{\mathrm E}2^\ell}|-1\bigr)
		&\leq C_{\mathrm F}\sum_{j=0}^{\infty}e^{-c_{\mathrm F}2^jv}\\
		&\leq C_{\mathrm F}\sum_{j=0}^{\infty}e^{-c_{\mathrm F}(j+1)v}\\
		&=\frac{C_{\mathrm F}e^{-c_{\mathrm F}v}}{1-e^{-c_{\mathrm F}v}}\\
		&\leq\frac{C_{\mathrm F}}{1-e^{-c_{\mathrm F}}}e^{-c_{\mathrm F}c_2m}.
	\end{align*}
	Thus this remaining dyadic range contributes at most $C'e^{-c_3m}$, where $c_3:=c_{\mathrm F}c_2>0$ and $C'$ is absolute. Enlarge $C_0$, if necessary; this only narrows the range in Lemma~\ref{lem:boolean-slice-fourier-reduction}. We may then assume that
	\[
	C'e^{-c_3m}\leq n^{-9}
	\]
	whenever $m\geq C_0\log(2+n)$. Substituting the preceding estimates into Lemma~\ref{lem:boolean-slice-fourier-reduction} gives
	\[
	\Pr[\Sigma(R)=z]\leq\frac1N+C_T\left(\frac1{n\sqrt m}+\frac1\lambda\right)+4n^{-9}.
	\]
	Since $m\leq n$,
	\[
	4n^{-9}\leq\frac4{n\sqrt m}.
	\]
	Therefore, for a constant $C_{\mathrm{med}}=C_{\mathrm{med}}(T)$,
	\begin{equation}
		\label{eq:medium-one-slice-N}
		\Pr[\Sigma(R)=z]\leq\frac1N+C_{\mathrm{med}}\left(\frac1{n\sqrt m}+\frac1\lambda\right).
	\end{equation}
	
	{\emph{Small values of $m$.}} Suppose that $m<C_0\log(2+n)$. Expose $m-1$ elements of $R$. Conditional on this exposure, the last element is uniform on a set of size $n-m+1\geq\varepsilon n$, and hence
	\[
	\Pr[\Sigma(R)=z]\leq\frac1{\varepsilon n}\leq\frac{\sqrt{C_0}}{\varepsilon}\frac{\sqrt{\log(2+n)}}{n\sqrt m}.
	\]
	
	{\emph{Large values of $m$.}} Suppose that $m>c_0n/\log(2+n)$.  {Put $\ell_n:=\log(2+n)$} and choose a constant $\alpha=\alpha(\varepsilon)>0$ satisfying $\alpha<c_0\varepsilon/4$. For all sufficiently large $n$, the integer
	\[
	h:=\left\lfloor\alpha\frac{n}{{\ell_n}}\right\rfloor
	\]
	satisfies $1\leq h\leq m$. Generate $R$ by first choosing a uniformly random $(m-h)$-subset $R_0\subseteq A$ and then a uniformly random $h$-subset $R_1\subseteq A\setminus R_0$, and put $R=R_0\cup R_1$. Every $m$-subset of $A$ arises from the same number $\binom{m}{h}$ of pairs $(R_0,R_1)$, so this procedure indeed generates a uniform $m$-subset. Conditional on $R_0$, the remaining pool $A':=A\setminus R_0$ has size
	\[
	n':=n-m+h\geq\varepsilon n.
	\]
	For all sufficiently large $n$, one has $h\geq\alpha n/(2{\ell_n})$. Since $n/{\ell_n}^2\to\infty$ and $\log(2+n')\leq {\ell_n}$, this gives
	\[
	h\geq C_0\log(2+n').
	\]
	On the other hand, $n'\geq\varepsilon n$ and $\log(2+n')\leq {\ell_n}$, so
	\[
	\frac{c_0n'}{\log(2+n')}\geq\frac{c_0\varepsilon n}{{\ell_n}}> \frac{\alpha n}{{\ell_n}}\geq h.
	\]
	Thus
	\[
	C_0\log(2+n')\leq h\leq c_0\frac{n'}{\log(2+n')},
	\]
	which is exactly the range required in Lemma~\ref{lem:boolean-slice-fourier-reduction}.
	Applying \eqref{eq:medium-one-slice-N} conditionally to $R_1\subseteq A'$ with target $z-\Sigma(R_0)$ yields
	\[
	\Pr[\Sigma(R)=z\mid R_0]\leq\frac1N+C_{\mathrm{med}}\left(\frac1{n'\sqrt h}+\frac1\lambda\right).
	\]
	Since $n'\geq\varepsilon n$ and $h\geq(\alpha/2)n/{\ell_n}$ for sufficiently large $n$,
	\[
	\frac1{n'\sqrt h}\leq C_\varepsilon\frac{\sqrt{\log(2+n)}}{n^{3/2}}\leq C_\varepsilon\frac{\sqrt{\log(2+n)}}{n\sqrt m}.
	\]
	Averaging over $R_0$ proves \eqref{eq:one-slice-N}. The finitely many values of $n$ excluded in {the argument for large values of $m$} are absorbed by increasing $C_{\varepsilon,T}$.
\end{proof}

{The restriction
\[
C_0\log(2+n)\leq m\leq c_0\frac{n}{\log(2+n)}
\]
in Lemma~\ref{lem:boolean-slice-fourier-reduction} comes directly from its proof. The upper bound ensures that the parts in the random partition have logarithmic size, while the lower bound makes the exponentially small Fourier remainder negligible. Proposition~\ref{prop:one-colour} treats smaller values of $m$ by exposing all but one element, and larger values by conditioning on a random subset of size comparable to $n/\log n$ to which the lemma applies.}

For later use, fix once and for all a constant $\Gamma_T\geq\max\{1,C_{1/2,T}\}$, and define
\begin{equation}
	\label{eq:theta-def}
	\vartheta(u;N,\lambda):=\frac{u}{N}+\Gamma_T\frac{\sqrt{\log(2+u)}}{\sqrt u}+\Gamma_T\frac{u}{\lambda},\qquad u\geq1.
\end{equation}
Any additional constants arising from a fixed number of equations or from a different fixed value of $\varepsilon$ will be kept outside $\vartheta$.}

\subsection{Ordinary interval chains}

{We next record the ordinary chain estimate used in the cyclic local-repair argument.}

{\begin{lemma}[Increment form of nested-chain equations]
\label{lem:chain-increment-form}
Let $A$ be a finite subset of an abelian group, let
\[
R_0:=\varnothing\subseteq R_1\subseteq\cdots\subseteq R_D\subseteq R_{D+1}:=A,\qquad C_i:=R_{i+1}\setminus R_i\quad(0\leq i\leq D),
\]
and fix targets $z_1,\ldots,z_D$. Put $z_0:=0$ and $z_{D+1}:=\Sigma(A)$. Then, for every $j\in\{0,\ldots,D\}$,
\begin{equation}
\label{eq:chain-increment-equivalence}
\{\Sigma(R_k)=z_k\text{ for all }1\leq k\leq D\}=\bigcap_{\substack{0\leq i\leq D\\ i\neq j}}\{\Sigma(C_i)=z_{i+1}-z_i\}.
\end{equation}
\end{lemma}

\begin{proof}
If $\Sigma(R_k)=z_k$ for every $1\leq k\leq D$, then
\[
\Sigma(C_i)=\Sigma(R_{i+1})-\Sigma(R_i)=z_{i+1}-z_i
\]
for $0\leq i\leq D$, where the cases $i=0$ and $i=D$ use $R_0=\varnothing$, $R_{D+1}=A$, $z_0=0$, and $z_{D+1}=\Sigma(A)$. Conversely, suppose that the equations on the right-hand side of \eqref{eq:chain-increment-equivalence} hold. Since $C_0,\ldots,C_D$ partition $A$,
\[
\Sigma(C_j)=\Sigma(A)-\sum_{i\neq j}\Sigma(C_i)=z_{D+1}-z_0-\sum_{i\neq j}(z_{i+1}-z_i)=z_{j+1}-z_j.
\]
Thus the increment equation also holds for $i=j$. Consequently, for every $1\leq k\leq D$,
\[
\Sigma(R_k)=\sum_{i=0}^{k-1}\Sigma(C_i)=\sum_{i=0}^{k-1}(z_{i+1}-z_i)=z_k.
\]
This proves \eqref{eq:chain-increment-equivalence}.
\end{proof}

{
A \emph{uniformly random nested chain} with prescribed cardinalities
$M_1<\cdots<M_D$ means a chain chosen uniformly among all chains
$R_1\subset\cdots\subset R_D\subset A$ with $|R_i|=M_i$. Equivalently,
with $R_0:=\varnothing$, $R_{D+1}:=A$, and
$C_i:=R_{i+1}\setminus R_i$, the increments
$(C_0,\ldots,C_D)$ form a uniformly random ordered partition of $A$
with their prescribed cardinalities.
}

\begin{proposition}[Cyclic chain anticoncentration]
\label{prop:ordinary-chain}
Fix $D\geq1$ and $T\geq1$.  {There exist constants
$C_1=C_1(D,T)>0$ and $C_2=C_2(D,T)>0$} with the following property.  Let $A\subseteq\Z_N$ be a set of size $n$, and
let
\[
        R_1\subseteq R_2\subseteq\cdots\subseteq R_D\subseteq A
\]
be a uniformly random nested chain with prescribed sizes
\[
        |R_i|=M_i,
        \qquad
        1\leq M_1<\cdots<M_D<n.
\]
Put $M_0=0$, $M_{D+1}=n$, and
\[
        \Delta_i=M_{i+1}-M_i,
        \qquad 0\leq i\leq D.
\]
Then, for every $z_1,\ldots,z_D\in\Z_N$,
\begin{equation}
\label{eq:ordinary-chain-pointwise}
        \Pr\bigl[\Sigma(R_i)=z_i\text{ for all }1\leq i\leq D\bigr]
        \leq
        \sum_{j=0}^D
        \prod_{\substack{0\leq i\leq D\\ i\neq j}}
        \left(
        \frac1N
        +{C_1}\frac{\sqrt{\log(2+n)}}{n\sqrt{\Delta_i}}
        +\frac{{C_1}}{\lambda}
        \right).
\end{equation}
Consequently,
\begin{equation}
\label{eq:ordinary-chain-summed}
\sum_{1\leq M_1<\cdots<M_D<n}
        \Pr\bigl[\Sigma(R_i)=z_i\text{ for all }i\bigr]
        \leq
        {C_2\vartheta(n;N,\lambda)^D}.
\end{equation}
\end{proposition}

\begin{proof}
Let
\[
        C_i:=R_{i+1}\setminus R_i,
        \qquad 0\leq i\leq D.
\]
The sets $C_0,\ldots,C_D$ form a uniformly random ordered partition
of $A$ with prescribed sizes $\Delta_0,\ldots,\Delta_D$.  Choose an
index $j$ for which $\Delta_j$ is maximal.  Then
\[
        \Delta_j\geq \frac{n}{D+1}.
\]
Expose the increments in the order
\[
        C_0,C_1,\ldots,C_{j-1},C_D,C_{D-1},\ldots,C_{j+1},
\]
leaving $C_j$ unestimated.  At the moment when $C_i$ is exposed, the
remaining pool contains both $C_i$ and $C_j$.  Hence the slice
$C_i$ occupies at most one half of that pool, while the pool itself
has size at least $\Delta_j\geq n/(D+1)$. Put $z_0:=0$ and $z_{D+1}:=\Sigma(A)$. By Lemma~\ref{lem:chain-increment-form},
\[
\{\Sigma(R_i)=z_i\text{ for all }1\leq i\leq D\}
=
\bigcap_{\substack{0\leq i\leq D\\ i\neq j}}
\{\Sigma(C_i)=z_{i+1}-z_i\}.
\]
{Thus each exposed increment is subject to one deterministic target equation, while $C_j$ is left unestimated.} Proposition~\ref{prop:one-colour}, with $\varepsilon=1/2$, therefore
gives the product in \eqref{eq:ordinary-chain-pointwise}.

{
To justify the use of the common parameter $n$ in every factor on the right-hand side of \eqref{eq:ordinary-chain-pointwise}, let $n_i$ be the
size of the pool from which $C_i$ is exposed.  The preceding construction
gives $n/(D+1)\leq n_i\leq n$, and hence
\[
 \frac{\sqrt{\log(2+n_i)}}{n_i\sqrt{\Delta_i}}
 \leq
 (D+1)\frac{\sqrt{\log(2+n)}}{n\sqrt{\Delta_i}}.
\]
Thus replacing every residual pool size by $n$ changes only the constant
$C_1=C_1(D,T)$.

For completeness, put
\[
        a:=\frac1N+\frac{C_1}{\lambda},
        \qquad
        b:=C_1\frac{\sqrt{\log(2+n)}}{n}.
\]
For a fixed $j\in\{0,\ldots,D\}$, dropping the constraint on the sum of the
increments and using $\sum_{u=1}^n u^{-1/2}\leq2\sqrt n$ gives
\[
\sum_{1\leq M_1<\cdots<M_D<n}
 \prod_{\substack{0\leq i\leq D\\ i\neq j}}
 \left(a+\frac{b}{\sqrt{\Delta_i}}\right)
 \leq
 \left(na+2b\sqrt n\right)^D.
\]
Summing over the $D+1$ possible choices of $j$ and choosing
$C_2=C_2(D,T)>0$ sufficiently large, we obtain
\[
 (D+1)\left(na+2b\sqrt n\right)^D
 \leq C_2\vartheta(n;N,\lambda)^D,
\]
which proves \eqref{eq:ordinary-chain-summed}.
}
\end{proof}

An unrestricted conditioned version of
Proposition~\ref{prop:ordinary-chain} is false in general, since
conditioning on prescribed positions may cause two distinct interval
equations to have the same residual support, or may otherwise reduce
the rank of the system.  In the applications below, however, the
conditioned contributions are confined to a bounded local window,
while the remaining random parts form a strict nested chain extending
to the right or to the left.  The following corollary records precisely
this form.

\begin{corollary}[Ordinary interval chains]
\label{cor:ordinary-interval-chain}
Fix $D\geq1$ and $K\geq0$, and let $1\leq h\leq D$.  Let
$x_1,\ldots,x_n$ be a uniformly random ordering of an $n$-element set
$A\subseteq\Z_N$.

Suppose first that $c\in[n]$, that $F\subseteq[1,c]$ satisfies
$|F|\leq K$, and that
\[
	c<b_1<\cdots<b_h,
	\qquad
	b_h-c<n-|F|.
\]
Let
\[
	\mathcal F_F:=\sigma(x_j:j\in F),
\]
and let $Z_1,\ldots,Z_h$ be $\mathcal F_F$-measurable random
variables.  The variables $Z_\ell$ may also depend deterministically
on the chosen endpoint tuple $(b_1,\ldots,b_h)$.  Then, uniformly in
$z_1,\ldots,z_h\in\Z_N$,
\begin{equation}
\label{eq:right-ordinary-interval-chain}
	\sum_{c<b_1<\cdots<b_h}
	\Pr\left[
		\sum_{j=c+1}^{b_\ell}x_j=z_\ell-Z_\ell
		\text{ for all }1\leq\ell\leq h
	\right]
	\leq
	C_{D,K}\vartheta(n;N,\lambda)^h.
\end{equation}

The analogous left-going estimate also holds.  More precisely, if
$F\subseteq[c,n]$, $|F|\leq K$, and
\[
	a_1<\cdots<a_h<c,
	\qquad
	c-a_1<n-|F|,
\]
then
\begin{equation}
\label{eq:left-ordinary-interval-chain}
	\sum_{a_1<\cdots<a_h<c}
	\Pr\left[
		\sum_{j=a_\ell}^{c-1}x_j=z_\ell-Z_\ell
		\text{ for all }1\leq\ell\leq h
	\right]
	\leq
	C_{D,K}\vartheta(n;N,\lambda)^h.
\end{equation}

Both conclusions remain valid after conditioning on the values in any
subset $F_0\subseteq F$ and after replacing the targets
$z_\ell$ by random variables measurable with respect to
\[
	\mathcal F_{F_0}:=\sigma(x_j:j\in F_0).
\]
Thus, almost surely,
\begin{equation}
\label{eq:ordinary-interval-chain-conditioned}
	\sum_{\mathrm{adm}}
	\Pr\left[
		\sum_{j\in I_\ell^{\mathrm{rem}}}x_j
		=
		W_\ell-Z_\ell
		\text{ for all }1\leq\ell\leq h
		\,\middle|\,
		\mathcal F_{F_0}
	\right]
	\leq
	C_{D,K}\vartheta(n;N,\lambda)^h,
\end{equation}
where the sum is over the admissible ordered remote endpoints,
$W_1,\ldots,W_h$ are $\mathcal F_{F_0}$-measurable, and
\[
	I_\ell^{\mathrm{rem}}=[c+1,b_\ell]
\]
in the right-going case, while
\[
	I_\ell^{\mathrm{rem}}=[a_\ell,c-1]
\]
in the left-going case.
\end{corollary}

\begin{proof}
We prove the right-going conditioned statement; the left-going case
follows by reversing the order of the positions.

First condition on all values occupying the positions in $F$.  Fix
such a realization and put
\[
	A'
	:=
	A\setminus\{x_j:j\in F\},
	\qquad
	n'
	:=
	|A'|
	=
	n-|F|.
\]
Let
\[
	j_1<\cdots<j_{n'}
\]
be the elements of $[n]\setminus F$, and set
\[
	\bar x_s:=x_{j_s},
	\qquad 1\leq s\leq n'.
\]
Conditional on $\mathcal F_F$, the sequence
$\bar x_1,\ldots,\bar x_{n'}$ is a uniformly random ordering of $A'$.

{
Since $F\subseteq[1,c]$, exactly
\[
	q_0:=c-|F|
\]
unconditioned positions lie in $[1,c]$.  For an admissible tuple
$c<b_1<\cdots<b_h$, put
\[
	M_\ell:=b_\ell-c,
	\qquad 1\leq\ell\leq h.
\]
Then
\[
	1\leq M_1<\cdots<M_h<n',
\]
and
\[
	\sum_{j=c+1}^{b_\ell}x_j
	=
	\sum_{s=q_0+1}^{q_0+M_\ell}\bar x_s.
\]
Consequently, the sets
\[
	R_\ell
	:=
	\{\bar x_{q_0+1},\ldots,\bar x_{q_0+M_\ell}\},
	\qquad 1\leq\ell\leq h,
\]
form a uniformly random nested chain
}

\[
	R_1\subset\cdots\subset R_h\subset A'
\]
with prescribed cardinalities $|R_\ell|=M_\ell$.

After conditioning on $\mathcal F_F$, the quantities
$W_\ell-Z_\ell$ are fixed.  Proposition~\ref{prop:ordinary-chain}
therefore applies for every fixed endpoint tuple.  Summing over all
such tuples, equivalently over
\[
	1\leq M_1<\cdots<M_h<n',
\]
and using \eqref{eq:ordinary-chain-summed}, we obtain
\begin{equation}
\label{eq:ordinary-chain-after-full-conditioning}
	\sum_{\mathrm{adm}}
	\Pr\left[
		\sum_{j=c+1}^{b_\ell}x_j
		=
		W_\ell-Z_\ell
		\text{ for all }1\leq\ell\leq h
		\,\middle|\,
		\mathcal F_F
	\right]
	\leq
	{C_{D,K,T}\vartheta(n';N,\lambda)^h}.
\end{equation}
The estimate in Proposition~\ref{prop:ordinary-chain} is uniform in
the targets, so it is harmless that the conditional targets may depend
on the endpoint tuple.

If $n\geq2K+2$, then $n'\geq n/2$, and hence
\[
	\vartheta(n';N,\lambda)
	\leq
	\sqrt2\,\vartheta(n;N,\lambda).
\]
If $n<2K+2$, then $n=O_K(1)$, so there are only
$O_{D,K}(1)$ admissible endpoint tuples.  Since
\[
	\vartheta(n;N,\lambda)
	\geq
	\frac{\sqrt{\log(2+n)}}{\sqrt n},
\]
the same conclusion follows after increasing $C_{D,K}$.
This proves the desired estimate after conditioning on
$\mathcal F_F$.

Finally, if only the values in $F_0\subseteq F$ are conditioned, we
average the preceding estimate over the remaining values occupying
$F\setminus F_0$ and apply the tower property.  This proves
\eqref{eq:ordinary-interval-chain-conditioned}, and the unconditioned
statements follow by taking $F_0=\varnothing$.
\end{proof}}

{
\begin{definition}[Local--remote chain property]
\label{def:local-remote-property}
Fix $D\geq1$, $K\geq0$, and $\Theta>0$.  We say that a uniformly
random ordering $x_1,\ldots,x_n$ has the
\emph{$(D,K,\Theta)$-local--remote chain property}, with constant
$C_{D,K}$, if for every $1\leq h\leq D$ the right-going and left-going
systems of Corollary~\ref{cor:ordinary-interval-chain} satisfy the
corresponding summed bound with right-hand side $C_{D,K}\Theta^h$.
This includes the version obtained after conditioning on any subset
of the at most $K$ local values and after translating the targets by
random variables measurable with respect to those conditioned values.
\end{definition}

Thus Corollary~\ref{cor:ordinary-interval-chain} says that a uniformly
random ordering of $A\subseteq\mathbb Z_N$ has the
$(D,K,\vartheta(n;N,\lambda))$-local--remote chain property.
}

\subsection{The non-periodic branch and the inverse theorem}
\label{subsec:nonperiodic-inverse}
{
We now study the non-periodic branch, in which no coset of a proper $q$-large subgroup contains almost all of the set. In the chain argument, however, the one-slice estimates are applied not to the original set but to random residual pools. We therefore first show that non-periodicity is preserved with high probability under successive exposures, and then establish a quantitative one-slice estimate for every non-periodic residual pool. This estimate also yields the inverse theorem below. Finally, we combine the two ingredients through a stopping-time argument to obtain the nested-chain estimates required in the applications.}

For a positive integer $m$, write \[ \Omega(m):=\sum_{p^\nu\parallel m}\nu \] for the number of prime factors of $m$, counted with multiplicity.

{
Let
\[
        N=tq,\qquad 1\leq t\leq T,\qquad (t,q)=1.
\]
{Recall that a subgroup $H<\mathbb Z_N$ is $q$-large if $|H|$ is divisible by a prime factor of $q$.} If $H$ is not $q$-large, then $|H|\mid t$ and hence $|H|\leq T$.

For $0<\eta\leq1$ and $u\geq1$, put
\begin{equation}
\label{eq:theta-zero-def}
        \vartheta_{0,\eta}(u;N):=\frac{u}{N}+C\eta^{-1/2}\frac{\sqrt{\log(2+u)}}{\sqrt u}.
\end{equation}
The constant $C$ in this notation depends only on the fixed parameters occurring in the statement in which the notation is used.

\begin{lemma}[Stability of non-periodicity under random exposure]
\label{lem:nonperiodicity-random-exposure}
Fix $D,T,L\geq1$ and $\delta>0$. There exist constants $C=C(D,T,L,\delta)$ and $c=c(D,T,L,\delta)>0$ such that the following holds. Let $N=tq$, where $1\leq t\leq T$, $(t,q)=1$, and $q$ has at most $L$ prime factors counted with multiplicity. Let $A\subseteq\mathbb Z_N$ have size $n$, let $0<\eta\leq1/4$, and assume that $\eta n\geq C\log(2+n)$ and
\[
        |A\cap(a+H)|\leq(1-2\eta)n
\]
for every coset $a+H$ of every proper $q$-large subgroup $H<\mathbb Z_N$. Let $(C_0,\ldots,C_{s-1},C_*)$ be a uniformly random ordered partition of $A$ into sets of prescribed cardinalities, where $s\leq D$, and define
\[
        A_i:=A\setminus\bigcup_{r<i}C_r,\qquad 0\leq i\leq s.
\]
Suppose that $|A_i|\geq\delta n$ for all $0\leq i\leq s$. For each $i$, let $\mathcal B_i$ be the event that there exist a proper $q$-large subgroup $H<\mathbb Z_N$ and a coset $a+H$ such that
\[
        |A_i\cap(a+H)|>(1-\eta)|A_i|.
\]
Then
\[
        \Pr\left(\bigcup_{i=0}^s\mathcal B_i\right)\leq C\exp(-c\eta n).
\]
\end{lemma}

\begin{proof}
Fix $i$, a proper $q$-large subgroup $H$, and a coset $Q=a+H$. The marginal distribution of $A_i$ is uniform among the $|A_i|$-subsets of $A$. Hence $X_{i,Q}:=|A_i\setminus Q|$ is hypergeometric with mean
\[
        \mu_{i,Q}=\frac{|A_i|}{n}|A\setminus Q|\geq2\eta|A_i|.
\]
If $|A_i\cap Q|>(1-\eta)|A_i|$, then $X_{i,Q}<\eta|A_i|\leq\mu_{i,Q}/2$. The multiplicative hypergeometric Chernoff inequality therefore gives
\[
        \Pr\bigl[|A_i\cap Q|>(1-\eta)|A_i|\bigr]\leq\exp(-\mu_{i,Q}/8)\leq\exp(-\eta|A_i|/4)\leq\exp(-\delta\eta n/4).
\]
Only boundedly many cosets can witness such an event. Indeed, if the event has positive probability, then necessarily
\[
        |A\cap Q|\geq(1-\eta)|A_i|\geq(1-\eta)\delta n\geq\frac{3\delta n}{4}.
\]
The cosets of a fixed subgroup are disjoint, so at most $4/(3\delta)$ cosets of that subgroup satisfy this inequality. Since $t\leq T$ and $\Omega(q)\leq L$, the number of subgroups of $\mathbb Z_N$ is $O_{T,L}(1)$. A union bound over at most $D+1$ pools, the $O_{T,L}(1)$ proper $q$-large subgroups, and the $O_\delta(1)$ relevant cosets proves the assertion.
\end{proof}

\begin{lemma}[Quantitative non-periodic one-slice estimate]\label{lem:parameter-dependent-nonperiodic-slice}
Fix $T,L\geq1$. There exists a constant $C=C(T,L)>0$ such that the following holds. Let $N=tq$, where $1\leq t\leq T$, $(t,q)=1$, and $\Omega(q)\leq L$. Let $P\subseteq\mathbb Z_N$ have size $u$, let $0<\rho\leq1/4$, and assume that
\begin{equation}
\label{eq:one-slice-rho-size}
        \rho u\geq C\log^2(2+u)
\end{equation}
and
\begin{equation}
\label{eq:one-slice-rho-nonperiodic}
        |P\cap(a+H)|\leq(1-\rho)u
\end{equation}
for every coset $a+H$ of every proper $q$-large subgroup $H<\mathbb Z_N$. If $R$ is a uniformly random $m$-subset of $P$ with $1\leq m\leq u/2$, then, for every $z\in\mathbb Z_N$,
\begin{equation}
\label{eq:parameter-dependent-one-slice}
        \Pr[\Sigma(R)=z]\leq\frac1N+C\rho^{-1/2}\frac{\sqrt{\log(2+u)}}{u\sqrt m}.
\end{equation}
\end{lemma}

\begin{proof}
Write $\ell_u:=\log(2+u)$. By increasing $C=C(T,L)$, we may assume that every choice of $u$ satisfying \eqref{eq:one-slice-rho-size} is larger than a constant depending only on $T$ and $L$. We divide the proof into three ranges.

Suppose first that $m\leq C_0\rho^{-1}\ell_u$, where $C_0$ is a sufficiently large absolute constant. Expose $m-1$ elements of $R$. Conditional on this exposure, the last element is uniform on a set of size $u-m+1$, and therefore
\[
        \Pr[\Sigma(R)=z]\leq\frac1{u-m+1}\leq\frac2u,
\]
where the last inequality uses $m\leq u/2$. Since $m\leq C_0\rho^{-1}\ell_u$, the final quantity is at most the second term on the right-hand side of \eqref{eq:parameter-dependent-one-slice}, after increasing $C$.

We next assume that
\begin{equation}
\label{eq:parameter-medium-range}
        C_0\rho^{-1}\ell_u\leq m\leq c_0\frac{u}{\ell_u},
\end{equation}
where $c_0>0$ is a sufficiently small absolute constant. For $\chi\in\mathbb Z_N$, define
\[
        \Psi(\chi):=\frac{m}{u^2}\sum_{x,y\in P}\|\chi(x-y)\|_N^2,
        \qquad
        B_v:=\{\chi\in\mathbb Z_N:\Psi(\chi)\leq v\}.
\]
{After renaming constants and harmlessly enlarging the error term,
Lemma~\ref{lem:boolean-slice-fourier-reduction}, applied to $P$, gives
absolute constants $C_1,c_1>0$ such that}
\begin{equation}
\label{eq:parameter-fourier-reduction}
        \Pr[\Sigma(R)=z]\leq\frac1N+\frac{C_1}{N}\sum_{\substack{s\in\{1,2,4,\ldots\}\\s\leq m/C_1}}e^{-c_1s}\bigl(|B_{C_1s}|-1\bigr)+Cu^{-8}+Ce^{-c_1m}.
\end{equation}
The subtraction of $1$ removes the trivial character, which has already contributed the main term $1/N$.

We estimate $|B_v|-1$ for $1\leq v\leq c_2\rho m$, where $c_2>0$ is sufficiently small. For $\delta>0$, put
\[
        Q_{v,\delta}:=\left\{x\in\mathbb Z_N:\sum_{\chi\in B_v}\|\chi x\|_N^2<\delta|B_v|\right\},
        \qquad
        Q:=Q_{v,2v/(\rho m)}.
\]
Averaging over $p_0\in P$ and then applying Markov's inequality gives some $p_0\in P$ for which
\begin{equation}
\label{eq:parameter-Q-large}
        |\{p\in P:p-p_0\in Q\}|\geq(1-\rho/2)u.
\end{equation}
Indeed,
\[
        \frac1{u^2}\sum_{p,p_0\in P}\sum_{\chi\in B_v}\|\chi(p-p_0)\|_N^2
        =\frac1m\sum_{\chi\in B_v}\Psi(\chi)
        \leq\frac{v}{m}|B_v|.
\]

We also record the Parseval bound
\begin{equation}
\label{eq:parameter-parseval}
        |Q_{v,1/200}|\leq\frac{5N}{4|B_v|}.
\end{equation}
To see this, note that $B_v$ is symmetric and contains $0$. Hence, for every $x\in Q_{v,1/200}$,
\[
        \sum_{\chi\in B_v}e_N(\chi x)
        =\sum_{\chi\in B_v}\cos(2\pi\chi x/N)
        \geq\sum_{\chi\in B_v}\bigl(1-20\|\chi x\|_N^2\bigr)
        \geq0.9|B_v|.
\]
Character orthogonality gives
\[
        \sum_{x\in\mathbb Z_N}\left|\sum_{\chi\in B_v}e_N(\chi x)\right|^2=N|B_v|,
\]
which implies \eqref{eq:parameter-parseval}.

Choose
\[
        r:=\left\lfloor c_3\sqrt{\frac{\rho m}{v}}\right\rfloor
\]
with $c_3>0$ sufficiently small. By decreasing $c_2$ if necessary, we have $r\geq2$ and
\[
        rQ\subseteq Q_{v,1/200}.
\]
Let $K:=\operatorname{Stab}(rQ)$. If $Q+K$ contains at least two $K$-cosets, Kneser's theorem gives
\[
\begin{aligned}
        |rQ|
        &\geq r|Q+K|-(r-1)|K|\\
        &=r\bigl(|Q+K|-|K|\bigr)+|K|\\
        &\geq\frac r2|Q|
        \geq c_4ru.
\end{aligned}
\]
Together with \eqref{eq:parameter-parseval}, this yields
\begin{equation}
\label{eq:parameter-Bs-bound}
        |B_v|-1\leq C_5\frac{N\sqrt v}{u\sqrt{\rho m}}.
\end{equation}

It remains to consider the case $Q+K=K$, equivalently $Q\subseteq K$. If $K=\mathbb Z_N$, then $rQ$ is a nonempty $\mathbb Z_N$-invariant subset and hence $rQ=\mathbb Z_N$. Therefore $\mathbb Z_N\subseteq Q_{v,1/200}$, and \eqref{eq:parameter-parseval} gives
\[
        N\leq\frac{5N}{4|B_v|}.
\]
Thus $|B_v|\leq5/4$. Since $|B_v|$ is an integer and $0\in B_v$, we obtain $B_v=\{0\}$, so \eqref{eq:parameter-Bs-bound} holds. If $K<\mathbb Z_N$ is $q$-large, then \eqref{eq:parameter-Q-large} places at least $(1-\rho/2)u>(1-\rho)u$ elements of $P$ in the coset $p_0+K$, contradicting \eqref{eq:one-slice-rho-nonperiodic}. If $K$ is not $q$-large, then $|K|\mid t$ and hence $|K|\leq T$, whereas \eqref{eq:parameter-Q-large} implies $(1-\rho/2)u\leq T$; this is incompatible with \eqref{eq:one-slice-rho-size} after increasing $C=C(T,L)$. Consequently, \eqref{eq:parameter-Bs-bound} holds throughout the range $1\leq v\leq c_2\rho m$.

Choose $c_4>0$ sufficiently small that $C_1s\leq c_2\rho m$ whenever $s\leq c_4\rho m$. For these dyadic values, apply \eqref{eq:parameter-Bs-bound} with $v=C_1s$. Since
\[
        \sum_{s\in\{1,2,4,\ldots\}}e^{-c_1s}\sqrt{s}=O(1),
\]
their total contribution to \eqref{eq:parameter-fourier-reduction} is at most $C/(u\sqrt{\rho m})$. For $s>c_4\rho m$, use the trivial estimate $|B_{C_1s}|-1\leq N$; the corresponding dyadic tail is at most $Ce^{-c\rho m}$. By \eqref{eq:parameter-medium-range}, $\rho m\geq C_0\ell_u$, so, after increasing $C_0$, this tail and the terms $u^{-8}$ and $e^{-c_1m}$ are absorbed into $C/(u\sqrt{\rho m})$. Hence
\begin{equation}
\label{eq:parameter-medium-one-slice}
        \Pr[\Sigma(R)=z]\leq\frac1N+\frac{C}{u\sqrt{\rho m}}.
\end{equation}

Finally, suppose that $c_0u/\ell_u<m\leq u/2$. Let
\[
        h:=\left\lfloor c_5\frac{u}{\ell_u}\right\rfloor,
\]
where $c_5>0$ is sufficiently small. Generate $R$ by first choosing a uniformly random $(m-h)$-subset $R_0\subseteq P$, and then choosing a uniformly random $h$-subset $R_1\subseteq P\setminus R_0$, with $R=R_0\cup R_1$. Put
\[
        P':=P\setminus R_0,
        \qquad
        u':=|P'|=u-m+h\geq u/2.
\]
By taking $c_5$ sufficiently small and then increasing the constant in \eqref{eq:one-slice-rho-size}, the integer $h$ satisfies
\[
        C_0(\rho/2)^{-1}\log(2+u')\leq h\leq c_0\frac{u'}{\log(2+u')}.
\]
Apply Lemma~\ref{lem:nonperiodicity-random-exposure} to the ordered partition $(R_0,P')$, with parameter $\rho/2$ and $\delta=1/2$. Its initial hypothesis is exactly \eqref{eq:one-slice-rho-nonperiodic}. Thus, except on an event of probability at most $Ce^{-c\rho u}$,
\[
        |P'\cap(a+H)|\leq(1-\rho/2)u'
\]
for every coset of every proper $q$-large subgroup. Conditional on such a residual set $P'$ and on $R_0$, apply \eqref{eq:parameter-medium-one-slice} to the uniform $h$-subset $R_1\subseteq P'$, with $\rho/2$ in place of $\rho$. Since $u'\asymp u$ and $h\asymp u/\ell_u$, we obtain
\[
        \Pr[\Sigma(R)=z\mid R_0]\leq\frac1N+C\rho^{-1/2}\frac{\sqrt{\ell_u}}{u^{3/2}}.
\]
Averaging over $R_0$ adds at most $Ce^{-c\rho u}$. By \eqref{eq:one-slice-rho-size}, this additive error is smaller than the second term in \eqref{eq:parameter-dependent-one-slice}. Finally, $m\leq u/2$ implies $u^{-3/2}\leq C/(u\sqrt m)$, completing the proof.
\end{proof}

{
The preceding estimate yields the following inverse theorem.
\begin{theorem}[Inverse theorem for Boolean-slice anticoncentration]
\label{thm:boolean-slice-inverse}
Fix $T,L\geq1$ and $0<\varepsilon,\nu<1$. There exists a constant
$C=C(T,L,\varepsilon,\nu)>0$ such that the following holds. Let
\[
        N=tq,\qquad 1\leq t\leq T,\qquad (t,q)=1,
        \qquad \Omega(q)\leq L,
\]
and let $A\subseteq\mathbb Z_N$ have size $n\geq2$. Then at least one of
the following alternatives holds:
\begin{enumerate}[label=\textup{(\roman*)}]
\item there exist a proper $q$-large subgroup $H<\mathbb Z_N$ and a
      coset $a+H$ such that
      \[
              |A\cap(a+H)|\geq(1-\nu)n;
      \]
\item for every integer $m$ satisfying
      \[
              1\leq m\leq(1-\varepsilon)n,
      \]
      every uniformly random $m$-subset $R$ of $A$, and every
      $z\in\mathbb Z_N$, one has
      \[
              \Pr[\Sigma(R)=z]
              \leq
              \frac1N+
              C\frac{\sqrt{\log(2+n)}}{n\sqrt m}.
      \]
\end{enumerate}
When $q=1$, alternative~\textup{(i)} is impossible.
\end{theorem}

\begin{proof}
Put
\[
        \rho:=\min\{\nu,1/4\}.
\]
Suppose that alternative~\textup{(i)} does not hold. Then, for every
coset $a+H$ of every proper $q$-large subgroup $H<\mathbb Z_N$,
\[
        |A\cap(a+H)|<(1-\nu)n\leq(1-\rho)n.
\]
Let $C_0=C_0(T,L)$ be the constant in
Lemma~\ref{lem:parameter-dependent-nonperiodic-slice}. Assume first that
\[
        \rho n\geq C_0\log^2(2+n).
\]
If $1\leq m\leq n/2$, that lemma, applied with $P=A$ and parameter
$\rho$, gives
\[
        \Pr[\Sigma(R)=z]
        \leq
        \frac1N+C_1(T,L)\rho^{-1/2}
        \frac{\sqrt{\log(2+n)}}{n\sqrt m}.
\]

Now suppose that
\[
        \frac n2<m\leq(1-\varepsilon)n.
\]
The complement $R':=A\setminus R$ is a uniformly random
$m'$-subset of $A$, where
\[
        m':=n-m\leq\frac n2,
        \qquad
        m'\geq\varepsilon n.
\]
Moreover,
\[
        \Sigma(R)=z
        \quad\Longleftrightarrow\quad
        \Sigma(R')=\Sigma(A)-z.
\]
Applying the preceding estimate to $R'$ and using
\[
        \frac1{\sqrt{m'}}
        \leq
        \sqrt{\frac{1-\varepsilon}{\varepsilon}}\,
        \frac1{\sqrt m}
\]
gives alternative~\textup{(ii)}, after changing the constant.

It remains to consider the values of $n$ for which
\[
        \rho n<C_0\log^2(2+n).
\]
There are only finitely many such $n$, with a bound depending on
$T,L$, and $\nu$. Since probabilities are at most $1$, all these cases
are absorbed by increasing $C=C(T,L,\varepsilon,\nu)$. This proves
alternative~\textup{(ii)} whenever alternative~\textup{(i)} fails.
\end{proof}

{Theorem~\ref{thm:boolean-slice-inverse} is the conceptual
formulation of the obstruction and serves as the guiding dichotomy for
the rest of the paper.  The later valid-ordering arguments use the
quantitative, boundedly conditioned chain version developed next.}
}

\begin{theorem}[Quantitative non-periodic chain estimates]
\label{thm:nonperiodic-chain}
Fix $D,T,L\geq1$. There exist constants $C=C(D,T,L)$ and $c=c(D,T,L)>0$ such that the following holds. Let $N=tq$, where $1\leq t\leq T$, $(t,q)=1$, and $\Omega(q)\leq L$. Let $A\subseteq\mathbb Z_N$ have size $n$, let $0<\eta\leq1/4$, and assume that
\begin{equation}
\label{eq:eta-size-condition}
        \eta n\geq C\log^2(2+n)
\end{equation}
and
\begin{equation}
\label{eq:nonperiodic-nocoset}
        |A\cap(a+H)|\leq(1-\eta)n
\end{equation}
for every coset $a+H$ of every proper $q$-large subgroup $H<\mathbb Z_N$. Let $R_1\subseteq\cdots\subseteq R_D\subseteq A$ be a uniformly random nested chain with prescribed sizes $|R_i|=M_i$, where $1\leq M_1<\cdots<M_D<n$. Put $M_0=0$, $M_{D+1}=n$, and $\Delta_i=M_{i+1}-M_i$. Then, for every $z_1,\ldots,z_D\in\mathbb Z_N$,
\begin{align}
\label{eq:nonperiodic-chain-product}
&\Pr[\Sigma(R_i)=z_i\text{ for all }1\leq i\leq D]\\
&\qquad\leq\sum_{j=0}^D\prod_{\substack{0\leq i\leq D\\i\neq j}}\left(\frac1N+C\eta^{-1/2}\frac{\sqrt{\log(2+n)}}{n\sqrt{\Delta_i}}\right)+C\exp(-c\eta n).\nonumber
\end{align}
Consequently,
\begin{equation}
\label{eq:nonperiodic-chain-summed}
\sum_{1\leq M_1<\cdots<M_D<n}\Pr[\Sigma(R_i)=z_i\text{ for all }i]\leq C_D\vartheta_{0,\eta}(n;N)^D+C\exp(-c\eta n/2).
\end{equation}
Moreover, fix an integer $K\geq0$. After allowing the constants $C$ and $c$ to depend additionally on $K$, the following conditioned version holds. Let $x_1,\ldots,x_n$ be a uniformly random ordering of $A$, let $F\subseteq[n]$ satisfy $|F|\leq K$, and put
\[
        \mathcal F_F:=\sigma(x_j:j\in F),
        \qquad
        A_F:=A\setminus\{x_j:j\in F\},
        \qquad
        n_F:=|A_F|=n-|F|.
\]
Conditional on $\mathcal F_F$, let
\[
        R^F_1\subseteq\cdots\subseteq R^F_h\subseteq A_F,
        \qquad 1\leq h\leq D,
\]
be a uniformly random nested chain with prescribed sizes, and let $W_1,\ldots,W_h$ be $\mathcal F_F$-measurable $\mathbb Z_N$-valued random variables. Then, almost surely, the estimates \eqref{eq:nonperiodic-chain-product} and \eqref{eq:nonperiodic-chain-summed} hold with $D$, $A$, $n$, $\eta$, and $z_i$ replaced respectively by $h$, $A_F$, $n_F$, $\eta/2$, and $W_i$. In particular, after changing the constants, their right-hand sides may be written in terms of the original parameters $n$ and $\eta$.
\end{theorem}

\begin{proof}
Set $R_0:=\varnothing$ and $R_{D+1}:=A$, and let $C_i:=R_{i+1}\setminus R_i$ for $0\leq i\leq D$. The sets $C_0,\ldots,C_D$ form a uniformly random ordered partition of $A$ with prescribed sizes $\Delta_0,\ldots,\Delta_D$. Choose $j$ with $\Delta_j=\max_i\Delta_i$, {leave $C_j$ unexposed}, and expose the other increments in the order
\[
        C_{\iota_0},\ldots,C_{\iota_{D-1}}:=C_0,\ldots,C_{j-1},C_D,C_{D-1},\ldots,C_{j+1}.
\]
For $0\leq r\leq D-1$, let $\mathcal F_r$ be the sigma-algebra generated by $C_{\iota_0},\ldots,C_{\iota_{r-1}}$, and put
\[
        P_r:=A\setminus\bigcup_{s<r}C_{\iota_s}.
\]
Then $P_r$ is $\mathcal F_r$-measurable and, conditional on $\mathcal F_r$, the set $C_{\iota_r}$ is a uniformly random $\Delta_{\iota_r}$-subset of $P_r$. Since $P_r$ contains both $C_{\iota_r}$ and {$C_j$}, one has
\begin{equation}
\label{eq:stopped-pool-size}
        \Delta_{\iota_r}\leq\frac{|P_r|}{2},\qquad |P_r|\geq\Delta_j\geq\frac{n}{D+1}.
\end{equation}

Call $P_r$ bad if there exist a proper $q$-large subgroup $H<\mathbb Z_N$ and a coset $a+H$ such that
\begin{equation}
\label{eq:stopped-bad-pool}
        |P_r\cap(a+H)|>(1-\eta/2)|P_r|.
\end{equation}
Define
\[
        \tau:=\min\{0\leq r\leq D-1:P_r\text{ is bad}\},
\]
with $\tau=\infty$ if no such stage exists. This is a stopping time for $(\mathcal F_r)$. Relabel the exposed increments as $(C_{\iota_0},\ldots,C_{\iota_{D-1}})$ and {the remaining increment as $C_*$}. Lemma~\ref{lem:nonperiodicity-random-exposure}, applied with parameter $\eta/2$ and $\delta=1/(2D+2)$, gives
\begin{equation}
\label{eq:stopping-time-tail}
        \Pr(\tau<\infty)\leq C\exp(-c\eta n).
\end{equation}
Indeed, the initial condition required by that lemma is exactly \eqref{eq:nonperiodic-nocoset}.

Fix $0\leq r\leq D-1$ and a concrete $\mathcal F_r$-history with $r<\tau$. Then $P_r$ is determined and satisfies
\[
        |P_r\cap(a+H)|\leq(1-\eta/2)|P_r|
\]
for every coset of every proper $q$-large subgroup. By \eqref{eq:eta-size-condition} and \eqref{eq:stopped-pool-size}, the hypothesis \eqref{eq:one-slice-rho-size} of Lemma~\ref{lem:parameter-dependent-nonperiodic-slice} holds for $P_r$ with $\rho=\eta/2$. Therefore, for every $\mathcal F_r$-measurable target $w\in\mathbb Z_N$,
\begin{equation}
\label{eq:stopped-one-step}
        \Pr\left[\sum_{x\in C_{\iota_r}}x=w\,\middle|\,\mathcal F_r\right]\leq\frac1N+C\eta^{-1/2}\frac{\sqrt{\log(2+n)}}{n\sqrt{\Delta_{\iota_r}}}.
\end{equation}
Here we used $|P_r|\asymp_Dn$ to replace $|P_r|$ by $n$ in the denominator.

Put $z_0:=0$ and $z_{D+1}:=\Sigma(A)$, and, for $0\leq r\leq D-1$, define
\[
w_r:=z_{\iota_r+1}-z_{\iota_r}.
\]
By Lemma~\ref{lem:chain-increment-form},
\begin{equation}
\label{eq:stopped-chain-increment-form}
\{\Sigma(R_i)=z_i\text{ for all }1\leq i\leq D\}=\bigcap_{r=0}^{D-1}\{\Sigma(C_{\iota_r})=w_r\}.
\end{equation}
The targets $w_r$ are deterministic and therefore $\mathcal F_r$-measurable. This also explains the reverse part of the exposure: if $\iota_r\in\{j+1,\ldots,D-1\}$, the equation is $\Sigma(C_{\iota_r})=z_{\iota_r+1}-z_{\iota_r}$, while for $\iota_r=D$ it is $\Sigma(C_D)=\Sigma(A)-z_D$; none of these targets depends on an increment that has not yet been exposed. Define
\[
\mathcal G_r:=\{\Sigma(C_{\iota_s})=w_s\text{ for all }0\leq s<r\}\cap\{\tau>r\},\qquad 0\leq r\leq D,
\]
where, for $r=D$, the condition $\tau>D$ is equivalent to $\tau=\infty$, since every finite value of $\tau$ lies in $\{0,\ldots,D-1\}$. For $0\leq r\leq D-1$, the event $\mathcal G_r$ belongs to $\mathcal F_r$. Moreover, $P_0=A$ is not bad, because \eqref{eq:nonperiodic-nocoset} gives
\[
|A\cap(a+H)|\leq(1-\eta)n<(1-\eta/2)n
\]
for every proper $q$-large subgroup coset. Hence $\tau>0$ deterministically and $\Pr(\mathcal G_0)=1$. Put
\[
        b_r:=\frac1N+C\eta^{-1/2}\frac{\sqrt{\log(2+n)}}{n\sqrt{\Delta_{\iota_r}}}.
\]
For $0\leq r\leq D-1$, the tower property and \eqref{eq:stopped-one-step} give
\[
\begin{aligned}
        \Pr(\mathcal G_{r+1})&=\mathbb E\left[\mathbf 1_{\mathcal G_r}\Pr\bigl(\Sigma(C_{\iota_r})=w_r,\ \tau>r+1\mid\mathcal F_r\bigr)\right]\\
        &\leq\mathbb E\left[\mathbf 1_{\mathcal G_r}\Pr\bigl(\Sigma(C_{\iota_r})=w_r\mid\mathcal F_r\bigr)\right]\leq b_r\Pr(\mathcal G_r).
\end{aligned}
\]
Iteration yields
\[
        \Pr[\Sigma(R_i)=z_i\text{ for all }i,\ \tau=\infty]\leq\prod_{r=0}^{D-1}b_r.
\]
Combining this with \eqref{eq:stopping-time-tail} proves \eqref{eq:nonperiodic-chain-product} for {the chosen index $j$}, and hence proves the stated sum over all $j$.

Summing the principal term over the cut points proceeds as in Proposition~\ref{prop:ordinary-chain}, using $\sum_{u\leq n}u^{-1/2}\leq2\sqrt n$. The additive exponential error contributes at most
\[
        Cn^D\exp(-c\eta n).
\]
By \eqref{eq:eta-size-condition}, one has $D\log n+\log C\leq c\eta n/2$ for all sufficiently large $n$, and therefore
\[
        Cn^D\exp(-c\eta n)\leq\exp(-c\eta n/2).
\]
This proves \eqref{eq:nonperiodic-chain-summed}.

It remains to prove the conditioned assertion. Fix $K\geq0$, let $x_1,\ldots,x_n$ be a uniformly random ordering of $A$, and let $F\subseteq[n]$ satisfy $|F|\leq K$. Put
\[
        \mathcal F_F:=\sigma(x_j:j\in F),
        \qquad
        A_F:=A\setminus\{x_j:j\in F\},
        \qquad
        n_F:=|A_F|=n-|F|.
\]
After increasing the constant in \eqref{eq:eta-size-condition}, depending additionally on $K$, we may assume that
\[
        \eta n\geq4K
        \qquad\text{and}\qquad
        n_F\geq\frac n2.
\]
For every coset $a+H$ of every proper $q$-large subgroup $H<\mathbb Z_N$, one has
\[
\begin{aligned}
        |A_F\cap(a+H)|
        &\leq |A\cap(a+H)|\\
        &\leq(1-\eta)n\\
        &\leq(1-\eta/2)n_F.
\end{aligned}
\]
Indeed,
\[
        (1-\eta/2)n_F-(1-\eta)n
        =\frac{\eta n}{2}-(1-\eta/2)|F|
        \geq\frac{\eta n}{2}-K
        \geq0.
\]
Moreover,
\[
        \frac{\eta}{2}n_F
        \geq\frac{\eta n}{4}
        \geq C'\log^2(2+n_F)
\]
after increasing the constant in \eqref{eq:eta-size-condition} once more.

Conditional on $\mathcal F_F$, the set $A_F$ is fixed. Every uniformly random nested chain in $A_F$ therefore satisfies the already proved unconditioned assertion, with $\eta/2$ in place of $\eta$ and with chain depth $h\leq D$. The conditional targets $W_1,\ldots,W_h$ are fixed after conditioning on $\mathcal F_F$, so the uniformity in the targets applies.

Finally, since
\[
        \frac n2\leq n_F\leq n,
        \qquad
        (\eta/2)^{-1/2}=\sqrt2\,\eta^{-1/2},
\]
the factors obtained with $n_F$ and $\eta/2$ satisfy
\[
        \frac1N+C(\eta/2)^{-1/2}
        \frac{\sqrt{\log(2+n_F)}}{n_F\sqrt{\Delta_i}}
        \leq
        \frac1N+C_K\eta^{-1/2}
        \frac{\sqrt{\log(2+n)}}{n\sqrt{\Delta_i}}.
\]
Likewise,
\[
        \exp\left(-c\frac{\eta}{2}n_F\right)
        \leq \exp(-c_K\eta n).
\]
Thus, after changing the constants, the conditional estimates have the same form as \eqref{eq:nonperiodic-chain-product} and \eqref{eq:nonperiodic-chain-summed} in terms of $n$ and $\eta$. This proves the conditioned assertion.
\end{proof}

\begin{corollary}[Non-periodic interval estimates]
\label{cor:nonperiodic-interval-chain}
{Assume the hypotheses of
Theorem~\ref{thm:nonperiodic-chain}, and fix $K_0\geq0$. For every
$1\leq h\leq D$, every right-going or left-going system satisfying the
hypotheses of Corollary~\ref{cor:ordinary-interval-chain} with $K=K_0$,
and every choice of targets, one has}
\begin{equation}
\label{eq:nonperiodic-local-remote-with-error}
\sum_{\mathrm{adm}}\Pr\left[\sum_{j\in I^{\mathrm{rem}}_\ell}x_j=z_\ell-Z_\ell\text{ for all }1\leq\ell\leq h\right]\leq C_{D,K_0}\vartheta_{0,\eta}(n;N)^h+C_{D,K_0}\exp(-c\eta n/2).
\end{equation}
{Here $\sum_{\mathrm{adm}}$, the remote intervals
$I_\ell^{\mathrm{rem}}$, and the local variables $Z_\ell$ have exactly
the meanings specified in Corollary~\ref{cor:ordinary-interval-chain}.
The same estimate holds after conditioning on any subset of the values
in the local window and after translating the targets by random variables
measurable with respect to those conditioned values.}

In particular, after decreasing $c=c(D,T,L)>0$ if necessary and defining
\begin{equation}
\label{eq:absorbed-nonperiodic-theta}
\widehat\vartheta_{0,\eta,D}(n;N):=\vartheta_{0,\eta}(n;N)+\exp\left(-\frac{c\eta n}{2D}\right),
\end{equation}
one has the pure-power estimate
\begin{equation}
\label{eq:nonperiodic-local-remote-pure-power}
\sum_{\mathrm{adm}}\Pr\left[\sum_{j\in I^{\mathrm{rem}}_\ell}x_j=z_\ell-Z_\ell\text{ for all }1\leq\ell\leq h\right]\leq C'_{D,K_0}\widehat\vartheta_{0,\eta,D}(n;N)^h,
\end{equation}
with the same conditioned version.
{Equivalently, the random ordering has the
$(D,K_0,\widehat\vartheta_{0,\eta,D}(n;N))$-local--remote chain
property of Definition~\ref{def:local-remote-property}.}
\end{corollary}

\begin{proof}
{Let $F$ be the local set of positions in the given
system, as in Corollary~\ref{cor:ordinary-interval-chain}, so that
$|F|\leq K_0$, and condition first on all values $(x_j)_{j\in F}$.} Put
\[
        \mathcal F_F:=\sigma(x_j:j\in F),
        \qquad
        A_F:=A\setminus\{x_j:j\in F\},
        \qquad
        n_F:=|A_F|=n-|F|.
\]
Conditional on $\mathcal F_F$, the entries occupying the positions in $[n]\setminus F$ form a uniformly random ordering of $A_F$, and every local random variable $Z_\ell$ is fixed.

Fix an admissible ordered tuple of remote endpoints. In the right-going case, the remote supports
\[
        I_\ell^{\mathrm{rem}}=[c+1,b_\ell],
        \qquad 1\leq\ell\leq h,
\]
are disjoint from $F\subseteq[1,c]$ and form a strict nested family. Hence the sets
\[
        R_\ell^F:=\{x_j:j\in I_\ell^{\mathrm{rem}}\},
        \qquad 1\leq\ell\leq h,
\]
form, conditional on $\mathcal F_F$, a uniformly random nested chain in $A_F$ with the prescribed cardinalities. In the left-going case the same conclusion holds after reversing the order of the indices, since
\[
        [a_h,c-1]\subset\cdots\subset[a_1,c-1]
\]
and these intervals are disjoint from $F\subseteq[c,n]$.

Apply the conditioned version of Theorem~\ref{thm:nonperiodic-chain}, with chain depth $h$ and conditional targets $z_\ell-Z_\ell$. Its estimates are uniform in the realized local values and in the targets. Summing over the admissible remote endpoints exactly as in the proof of Corollary~\ref{cor:ordinary-interval-chain} gives
\[
\sum_{\mathrm{adm}}
\Pr\left[
\sum_{j\in I_\ell^{\mathrm{rem}}}x_j=z_\ell-Z_\ell
\text{ for all }1\leq\ell\leq h
\,\middle|\,
\mathcal F_F
\right]
\leq
C_{D,K_0}\vartheta_{0,\eta}(n;N)^h
+
C_{D,K_0}\exp(-c\eta n/2).
\]
Averaging over the values occupying $F$ proves \eqref{eq:nonperiodic-local-remote-with-error}.

More generally, suppose that only the values in a subset $F_0\subseteq F$ are conditioned and that the translated targets are $\mathcal F_{F_0}$-measurable. Condition additionally on the values occupying $F\setminus F_0$, apply the preceding estimate, and then average over these additional local values. The tower property gives the asserted partial-conditioning version.

Finally, since $1\leq h\leq D$,
\[
\exp(-c\eta n/2)=\left(\exp\left(-\frac{c\eta n}{2D}\right)\right)^D\leq\left(\exp\left(-\frac{c\eta n}{2D}\right)\right)^h\leq\widehat\vartheta_{0,\eta,D}(n;N)^h.
\]
Combining this with \eqref{eq:nonperiodic-local-remote-with-error} proves \eqref{eq:nonperiodic-local-remote-pure-power}.
\end{proof}

\begin{proposition}[Polynomial non-periodic chain estimates]
\label{prop:polynomial-nonperiodic-chain}
Fix $D,T,L\geq1$ and $0<\kappa<1$. Let $N=tq$, where $1\leq t\leq T$, $(t,q)=1$, and $\Omega(q)\leq L$. Let $A\subseteq\mathbb Z_N$ have size $n$ and assume that
\[
        |A\cap(a+H)|\leq(1-n^{-\kappa})n
\]
for every coset of every proper $q$-large subgroup. Then, for all sufficiently large $n$, the conclusions of Theorem~\ref{thm:nonperiodic-chain} hold with
\[
        \frac1N+C\frac{n^{\kappa/2}\sqrt{\log(2+n)}}{n\sqrt{\Delta_i}}
\]
in each factor and with an additive error term $C\exp(-cn^{1-\kappa})$. Consequently, the summed interval-chain parameter is
\begin{equation}
\label{eq:polynomial-chain-theta}
        \vartheta_{0,\kappa}(n;N):=\frac nN+Cn^{\kappa/2}\frac{\sqrt{\log(2+n)}}{\sqrt n}.
\end{equation}
The same statement holds after conditioning on $O_D(1)$ positions.
\end{proposition}

\begin{proof}
Apply Theorem~\ref{thm:nonperiodic-chain} with
\[
        \eta=n^{-\kappa}.
\]
Then $\eta n=n^{1-\kappa}\geq C\log^2(2+n)$ for all sufficiently large $n$, $\eta^{-1/2}=n^{\kappa/2}$, and the additive error term is $C\exp(-cn^{1-\kappa})$. The conditioned statement is included in that theorem.
\end{proof}

\begin{remark}
\label{rem:polynomial-nonperiodic-parameters}
For every fixed $\kappa<1$, the second term in \eqref{eq:polynomial-chain-theta} is $n^{-(1-\kappa)/2+o(1)}$. Hence the non-periodic local-repair argument retains a positive polynomial saving even when the periodic alternative is required to concentrate all but $n^{1-\kappa}$ elements in one coset.
\end{remark}}

\section{Polynomial-range applications}
\label{sec:polynomial-applications}

{
We first abstract the local-repair argument of Pham--Sauermann in a form driven by interval-chain estimates, and then derive the cyclic polynomial-range theorem.
}

\subsection{Cyclic groups}
\label{subsec:cyclic-application}

{For a bijection $\sigma:[n]\to A$ and a set of positions $J\subseteq[n]$, write $\Sigma_\sigma(J):=\sum_{j\in J}\sigma(j)$.

\begin{lemma}[Local repair from interval-chain estimates]
\label{lem:cyclic-local-repair}
Fix $\beta>0$. There exists an integer $D=D(\beta)\geq2$ with the following property. Let $A\subseteq G\setminus\{0\}$ be an $n$-element subset of a finite abelian group, let $\sigma:[n]\to A$ be a uniformly random bijection, and let $0<\Theta\leq1$. {
Suppose that $\sigma$ has the
$(D,100D^2,\Theta)$-local--remote chain property of
Definition~\ref{def:local-remote-property}, with constant $C_D$.}
{If $\Theta\leq n^{-\beta}$, then there exists an integer
\[
n_0=n_0(\beta,C_D)
\]
such that $A$ admits a valid ordering whenever $n\geq n_0$.}

\end{lemma}

\begin{proof}
{Choose, for instance, an integer $D>3/\beta$.}
For a bijection $\sigma:[n]\to A$, define
\[
B(\sigma):=\{b\in[n]:\Sigma_\sigma([a,b])=0\text{ for some }2\leq a<b\}.
\]
Since $0\notin A$, no singleton interval has sum zero, and $\sigma$ is valid precisely when $B(\sigma)=\varnothing$.

 {For $u<v$, let $\pi_{u,v}$ denote the transposition of $u$ and $v$.  A permutation of $[n]$ is admissible if it is a product of pairwise disjoint transpositions $\pi_{u,v}$ with $u,v\in[n]$ and $0<v-u\leq5D$.} Let $b\in B(\sigma)$, let $\pi$ be admissible and fix $[b-1]$, and let
{
\[
        b<\RepairPos\leq\min\{n,b+5D\}.
\]
}
The position $\RepairPos$ is blocked for $(\sigma,b,\pi)$ if there is an interval $[s,t]$, with $2\leq s<t\leq n$, such that
\[
\Sigma_{\sigma\circ\pi\circ\pi_{b,\RepairPos}}([s,t])=0
\]
and either $s\in\{b+1,\ldots,\RepairPos\}$ or $t\in\{b,\ldots,\RepairPos-1\}$.

{
Any interval $[s,t]$ with these properties is called a
\emph{witnessing interval}, or simply a \emph{witness}, for the blocked
position $\RepairPos$.  It is called \emph{right-going} if
$s\in\{b+1,\ldots,\RepairPos\}$ and \emph{left-going} if
$t\in\{b,\ldots,\RepairPos-1\}$.  If both conditions hold, either orientation
may be assigned to the witness.
}

{
The four bad events below encode the possible failures of the repair
algorithm.  Event $\mathcal B_1$ guarantees room to move a bad
endpoint to the right, $\mathcal B_2$ controls the local density of bad
endpoints, $\mathcal B_3$ guarantees an unblocked partner, and
$\mathcal B_0$ provides the local subset-sum injectivity used to
separate remote witnesses.  By the local--remote chain property, every
system of $h\leq D$ remote equations arising below satisfies
\begin{equation}
\label{eq:abstract-cyclic-chain}
        \sum_{\mathrm{adm}}\Pr[\text{the $h$ equations hold}]
        \leq C_D\Theta^h,
\end{equation}
including after the bounded local conditioning used in the proof.
{In \eqref{eq:abstract-cyclic-chain}, $\sum_{\mathrm{adm}}$ denotes the
sum over the admissible ordered remote endpoints in the relevant
right-going or left-going system, in the sense of
Corollary~\ref{cor:ordinary-interval-chain}.}
}

Let $\mathcal B_0$ be the event that there exist $b\in B(\sigma)$ with $b\leq n-30D$ and distinct subsets $J,J'\subseteq\{b,b+1,\ldots,b+20D\}$ satisfying $\sum_{j\in J}\sigma(j)=\sum_{j\in J'}\sigma(j)$. Let $\mathcal B_1$ be the event that $B(\sigma)\cap\{n-30D,\ldots,n\}\neq\varnothing$. Let $\mathcal B_2$ be the event that some interval of $20D+1$ consecutive positions contains more than $D$ elements of $B(\sigma)$. Let $\mathcal B_3$ be the event that there exist $b\in B(\sigma)$ and an admissible permutation $\pi$ fixing $[b-1]$ for which at least $2D$ positions in
{$\{b+1,\ldots,\min\{n,b+5D\}\}$}
are blocked for $(\sigma,b,\pi)$.

Assume that $\mathcal B_1,\mathcal B_2,\mathcal B_3$ do not occur, and write $B(\sigma)=\{b_1>\cdots>b_m\}$.  {We construct distinct positions $\RepairPos_i$} not in $B(\sigma)$ with $b_i<\RepairPos_i\leq b_i+5D$ such that, after the first $i$ transpositions, every zero-sum interval has right endpoint in $\{b_{i+1},\ldots,b_m\}$. Suppose that $\RepairPos_1,\ldots,z_{i-1}$ have been chosen and put $\pi:=\pi_{b_1,\RepairPos_1}\cdots\pi_{b_{i-1},z_{i-1}}$. Then $\pi$ fixes $[b_i]$. The interval $\{b_i+1,\ldots,b_i+5D\}$ is contained in $[n]$ because $\mathcal B_1$ does not occur. It contains at most $D$ elements of $B(\sigma)$ and at most $D$ previously chosen partners: if $\RepairPos_j\in[b_i+1,b_i+5D]$, then $|b_j-b_i|\leq|b_j-\RepairPos_j|+|\RepairPos_j-b_i|\leq10D$, and $\mathcal B_2$ applies. At most $2D$ of the remaining positions are blocked because $\mathcal B_3$ does not occur. Hence at least $5D-D-D-2D=D$ positions remain, and one of them is chosen as $\RepairPos_i$.

Set $\sigma_i:=\sigma\circ\pi\circ\pi_{b_i,\RepairPos_i}$. If $[a,c]$ is a zero-sum interval for $\sigma_i$, then the fact that $\RepairPos_i$ is not blocked implies $a\notin\{b_i+1,\ldots,\RepairPos_i\}$ and $c\notin\{b_i,\ldots,\RepairPos_i-1\}$. Consequently $[a,c]$ contains either both of $b_i,\RepairPos_i$ or neither of them, and its sum is unchanged by $\pi_{b_i,\RepairPos_i}$. Thus $[a,c]$ was zero-sum for $\sigma\circ\pi$, so the induction hypothesis gives $c\in\{b_i,\ldots,b_m\}$. The equality $c=b_i$ is impossible, since then $[a,b_i]$ would certify that $\RepairPos_i$ is blocked. Hence $c\in\{b_{i+1},\ldots,b_m\}$. After the last transposition no zero-sum interval remains.

We next estimate the four bad events. We first record a consequence of \eqref{eq:abstract-cyclic-chain}. Let $b'=b+5D$, let $u_1,\ldots,u_h\in[b,b']$, and let $\rho_1,\ldots,\rho_h$ be permutations supported on $[b,b']$, where $1\leq h\leq D$. Then
\begin{equation}
\label{eq:right-remote-witness}
\Pr\Bigl[\exists\ b'<x_1<\cdots<x_h\text{ and an admissible }\pi:\ \Sigma_{\sigma\circ\pi\circ\rho_i}([u_i,x_i])=0\text{ for all }i\Bigr]\leq C_D\Theta^h.
\end{equation}
The analogous estimate holds for $x_1<\cdots<x_h<b$ and the equations $\Sigma_{\sigma\circ\pi\circ\rho_i}([x_i,u_i])=0$.

{
Fix $x_1<\cdots<x_h$ and put
\[
        E_i:=\rho_i([u_i,x_i]).
\]
If a transposition occurring in $\pi$ has either both endpoints or
neither endpoint in every $E_i$, deleting it does not change any of the
$h$ sums.  We may therefore delete all such transpositions.  Since
$\rho_i$ is supported on $[b,b']$, the set $E_i$ differs from the
interval $[b'+1,x_i]$ only inside the bounded window $[b,b']$.
Consequently every transposition which remains has both endpoints
within distance $5D$ either of that window or of one of the remote
endpoints $x_1,\ldots,x_h$.  For the fixed endpoint tuple there are
therefore at most $C_D$ actual reduced admissible permutations $\pi$,
not merely $C_D$ abstract incidence types.

Fix one of these reduced permutations and write
$\tau:=\sigma\circ\pi$.  Then $\tau$ is a uniformly random bijection.
Moreover,
\[
\Sigma_{\sigma\circ\pi\circ\rho_i}([u_i,x_i])
=
\Sigma_\tau([b'+1,x_i])+Z_i,
\qquad
Z_i:=\Sigma_\tau\bigl(\rho_i([u_i,b'])\bigr).
\]
The variable $Z_i$ is determined by the values of $\tau$ in $[b,b']$.
After conditioning on these values, the intervals $[b'+1,x_i]$ form a
strict nested chain and the conditioned form of
\eqref{eq:abstract-cyclic-chain} applies.  A union bound over the at
most $C_D$ reduced permutations, followed by summation over
$b'<x_1<\cdots<x_h$, proves \eqref{eq:right-remote-witness}.
Reversing the order of the positions gives the left-endpoint version.
}

{
Fix $b\leq n-30D$ and distinct $J,J'\subseteq[b,b+20D]$, and let
\[
        E_{J,J'}:=
        \left\{\sum_{j\in J}\sigma(j)=\sum_{j\in J'}\sigma(j)\right\}.
\]
Choose $i\in J\triangle J'$.  Conditional on the values in
$[b,b+20D]\setminus\{i\}$, the equation defining $E_{J,J'}$ determines
at most one value of $\sigma(i)$.  Since at least $n-20D$ values remain
available, $\Pr(E_{J,J'})\leq2/n$ for all sufficiently large $n$.
Let $\mathcal F$ be generated by the values in $[b,b+20D]$.  For every
realization of $\mathcal F$, the one-equation conditioned form of
\eqref{eq:abstract-cyclic-chain} gives
\[
        \Pr[b\in B(\sigma)\mid\mathcal F]\leq C_D\Theta.
\]
The tower property therefore yields
\[
\begin{aligned}
&\Pr\Bigl[b\in B(\sigma),\
\sum_{j\in J}\sigma(j)=\sum_{j\in J'}\sigma(j)\Bigr]\\
&\qquad=
\mathbb E\!\left[
\mathbf 1_{E_{J,J'}}
\Pr[b\in B(\sigma)\mid\mathcal F]
\right]
\leq \frac{C_D\Theta}{n},
\end{aligned}
\]
after enlarging $C_D$.
}
A union bound over $b,J,J'$ yields
\begin{equation}
\label{eq:B0-cyclic-bound}
\Pr(\mathcal B_0)\leq C_D\Theta.
\end{equation}
The one-equation estimate and a union bound over the final $30D+1$ positions similarly give
\begin{equation}
\label{eq:B1-cyclic-bound}
\Pr(\mathcal B_1)\leq C_D\Theta.
\end{equation}

Suppose that $\mathcal B_2\setminus(\mathcal B_0\cup\mathcal B_1)$ occurs. Let $b_0$ be the smallest bad endpoint in a window containing more than $D$ bad endpoints. {Since $\mathcal B_1$ does not occur, $b_0\leq n-30D$, so the definition of $\mathcal B_0$ applies with $b=b_0$.} Choose distinct $b_1,\ldots,b_D\in B(\sigma)\cap[b_0+1,b_0+20D]$. Choose $a_i$ with $2\leq a_i<b_i$ and $\Sigma_\sigma([a_i,b_i])=0$. The $a_i$ are distinct, since $a_i=a_j$ would give a nonempty zero-sum subset of $[b_0,b_0+20D]$ by subtracting the two equations. Moreover, every $a_i<b_0$, since otherwise $[a_i,b_i]$ itself would be a nonempty zero-sum subset of that window. Relabel so that $a_1<\cdots<a_D<b_0$. Conditional on the values in $[b_0,b_0+20D]$, the equations are
\[
\Sigma_\sigma([a_i,b_0-1])=-\Sigma_\sigma([b_0,b_i]),\qquad 1\leq i\leq D.
\]
The intervals on the left form a nested family. Applying \eqref{eq:abstract-cyclic-chain} and summing over the $n$ choices of $b_0$ and the $O_D(1)$ local data gives
\begin{equation}
\label{eq:B2-cyclic-bound}
\Pr(\mathcal B_2)\leq C_D\bigl(\Theta+n\Theta^D\bigr).
\end{equation}

Suppose that $\mathcal B_3\setminus(\mathcal B_0\cup\mathcal B_1)$ occurs, witnessed by $b$ and $\pi$. Since $\pi$ fixes $[b-1]$ and is admissible, it maps every subset of $[b,b+5D]$ to a subset of $[b,b+10D]$. The failure of $\mathcal B_0$ implies that the map $J\mapsto\sum_{j\in\pi(J)}\sigma(j)$ is injective on the subsets of $[b,b+5D]$. In particular, a blocked witness cannot be contained in $[b,b+5D]$. {Assign each blocked position one witnessing interval and, if that interval satisfies both alternatives in the definition of blocked, assign either orientation.  At least $D$ of the blocked positions then have witnesses of one common orientation.}

Consider $D$ right-going witnesses. There are distinct $\RepairPos_1,\ldots,\RepairPos_D\in[b+1,b+5D]$, positions $s_i\in[b+1,\RepairPos_i]$, and positions $t_i>b+5D$ such that
\[
\Sigma_{\sigma\circ\pi\circ\pi_{b,\RepairPos_i}}([s_i,t_i])=0,
\qquad 1\leq i\leq D.
\]
The $t_i$ are distinct. Indeed, if $t_i=t_j$ and $\RepairPos_i<\RepairPos_j$, then cancellation of the common part $[b+5D+1,t_i]$ gives equal sums under $\sigma\circ\pi$ for the two local subsets $\pi_{b,\RepairPos_i}([s_i,b+5D])$ and $\pi_{b,\RepairPos_j}([s_j,b+5D])$. These subsets are distinct: the first contains $\RepairPos_j$, whereas the second does not. This contradicts the local injectivity above. After relabeling, assume $t_1<\cdots<t_D$. For fixed $b$, $\RepairPos_1,\ldots,\RepairPos_D$, and $s_1,\ldots,s_D$, estimate \eqref{eq:right-remote-witness} with $\rho_i=\pi_{b,\RepairPos_i}$ gives probability at most $C_D\Theta^D$. Summing over $b$ and the $O_D(1)$ local choices gives $C_Dn\Theta^D$. {For left-going witnesses the same cancellation argument shows that their remote left endpoints are distinct; after relabeling, the left-endpoint version of \eqref{eq:right-remote-witness} applies.} Therefore
\begin{equation}
\label{eq:B3-cyclic-bound}
\Pr(\mathcal B_3)\leq C_D\bigl(\Theta+n\Theta^D\bigr).
\end{equation}

{Combining \eqref{eq:B0-cyclic-bound}--\eqref{eq:B3-cyclic-bound},
\[
\Pr(\mathcal B_0\cup\mathcal B_1\cup\mathcal B_2\cup\mathcal B_3)
\leq C_D\bigl(\Theta+n\Theta^D\bigr).
\]
If $\Theta\leq n^{-\beta}$, then
\[
C_D\bigl(\Theta+n\Theta^D\bigr)
\leq
C_D\bigl(n^{-\beta}+n^{1-\beta D}\bigr).
\]
Since $D>3/\beta$, both exponents $-\beta$ and $1-\beta D$ are negative. Hence there exists
\[
n_0=n_0(\beta,C_D)
\]
such that
\[
C_D\bigl(n^{-\beta}+n^{1-\beta D}\bigr)<1
\]
whenever $n\geq n_0$. For such $n$, some ordering avoids all four bad events, and the transpositions constructed above produce a valid ordering of $A$.}
\end{proof}

}

\begin{proof}[Proof of \cref{thm:cyclic-small}]
Put
\[
        N=tq,
        \qquad
        n=|S|,
        \qquad
        \lambda=P^-(q).
\]
{
After increasing $C=C(\alpha,T)$, if necessary, so that $C>T$, the
case $q=1$ is vacuous. We may therefore assume that $q>1$, and hence
$\lambda\leq q\leq N$.
}
{Choose}
\[
        0<\beta<
        \min\left\{
        \frac12,\frac{\alpha}{1-\alpha}
        \right\}.
\]
By Corollary~\ref{cor:ordinary-interval-chain}, the hypothesis of
Lemma~\ref{lem:cyclic-local-repair} holds with a constant
$C_D=C_D(D,T)$ and with
\begin{equation}
\label{eq:Theta-cyclic}
        \Theta_C
        :=
        \vartheta(n;N,\lambda)
        =
        \frac nN
        +{\Gamma_T}\frac{\sqrt{\log(2+n)}}{\sqrt n}
        +{\Gamma_T}\frac n\lambda.
\end{equation}
Since $n\leq\lambda^{1-\alpha}$ and $\lambda\leq N$, we have
\[
        \frac nN
        \leq
        \frac n\lambda
        \leq
        \lambda^{-\alpha}
        \leq
        n^{-\alpha/(1-\alpha)}.
\]
{
Because $\beta<1/2$ and
$\beta<\alpha/(1-\alpha)$, each of the three terms in
\eqref{eq:Theta-cyclic} is $o(n^{-\beta})$.  After increasing the lower
threshold $C=C(\alpha,T)$ appearing in
Theorem~\ref{thm:cyclic-small}, if necessary, so that also
$C\geq n_0(\beta,C_D)$, we therefore obtain
\[
        \Theta_C\leq n^{-\beta}
        \qquad\text{and}\qquad
        n\geq n_0(\beta,C_D).
\]
}
Lemma~\ref{lem:cyclic-local-repair} gives the desired ordering.
\end{proof}

\subsection{Covered moduli}
\label{subsec:covered-moduli}

{
Let $c>0$ be the constant in Theorem~\ref{thm:large-known}.  For
$k=tq$, where $(t,q)=1$ and $t$ is fixed,
Theorem~\ref{thm:cyclic-small} covers all sufficiently large
cardinalities up to
\[
        (P^-(q))^{1-\alpha},
\]
while Theorem~\ref{thm:large-known} covers
\[
        |S|\geq(tq)^{1-c}.
\]
Hence these two ranges overlap whenever
\begin{equation}
\label{eq:cyclic-criterion}
        (P^-(q))^{1-\alpha}\geq(tq)^{1-c}
\end{equation}
for some $\alpha<c$.

It remains only to cover the bounded cardinalities below the constant
in Theorem~\ref{thm:cyclic-small}.  When $q=p$ is prime and
$\mathbb Z_t$ is strongly sequenceable,
Theorem~\ref{thm:kravitz-direct} supplies a logarithmic range which
contains every fixed threshold for all sufficiently large $p$.
Since \eqref{eq:cyclic-criterion} also holds for fixed $t$ and
sufficiently large $p$, all cardinalities are covered.  As observed
above, $\mathbb Z_t$ is strongly sequenceable for every
$2\leq t\leq21$.  Since there are only finitely many such $t$, the
prime threshold may be chosen uniformly over this range.  The same
argument for any fixed strongly sequenceable $\mathbb Z_t$ proves the
more general assertion in Corollary~\ref{cor:prime-cases}.
}

\section{Comparable bounded prime-power moduli via layered local repair}
\label{sec:bounded-prime-power}

\journalchg{We prove Theorem~\ref{thm:comparable-prime-power} by a reverse-absorption argument for cyclic groups with a fixed number of
comparable prime factors, counted with multiplicity.}  The periodic
alternative may now recur inside successively smaller subgroups.  At
each stage, either non-periodic Kneser growth gives the
anticoncentration needed for local repair, or almost all unused elements
lie in one coset of a proper subgroup; in the latter case the new
exceptions are appended to the existing head and the construction
descends.  Since each descent decreases the number of prime factors of
the current subgroup, the process terminates after boundedly many steps.

\journalchg{Two points are important.}  A regular tail in a proper coset may wrap around
the quotient many times, but the quotient coordinate acts as a clock:
internal collisions use intervals whose length is divisible by the
order of that coset in the quotient.  Moreover, the recursive statements
must retain the path already constructed.  In the
subgroup-dominant branch one exceptional element is reserved as a bridge
into a fresh quotient coset; a Kneserized terminal correction then either
permits the cycle trick or produces the next genuine descent.  We treat
these ingredients in order and conclude with the proof of
Theorem~\ref{thm:comparable-prime-power}.

Fix throughout $L\geq2$ and $\gamma\geq1$ as in that theorem, and write
\[
 G\simeq\Z_k,\qquad
 k=\prod_{i=1}^s p_i^{e_i},\qquad
 e_1+\cdots+e_s\leq L,
 \qquad
 p_{\min}:=p_1<\cdots<p_s\leq\gamma p_{\min}.
\]
Every subgroup and quotient used below has prime bases drawn from this
list, hence still $\gamma$-comparable and at least $p_{\min}$.  All lower
bounds on $p_{\min}$ will be absorbed into the final threshold
$p_0(L,\gamma)$.  For a cyclic group $M$, let
$\Omega(M):=\Omega(|M|)$ be the number of prime factors of $|M|$, with
multiplicity; it strictly decreases on passing to a proper subgroup.  For
$K\leq G$ write $\overline x^{\,K}:=x+K\in G/K$, omitting the superscript
when clear.

For $v_0\in G$ and a finite sequence
$w=(w_1,\ldots,w_t)\in(G\setminus\{0\})^t$ with pairwise distinct
terms, put
\[
 \PathVar(v_0;w)=(v_0,v_1,\ldots,v_t),
 \qquad v_i=v_0+\sum_{j=1}^i w_j.
\]
We call $\PathVar(v_0;w)$ the \emph{partial-sum walk} associated with $w$
from $v_0$.  It is a \emph{path} if $v_0,\ldots,v_t$ are pairwise
distinct, and a \emph{cycle} if $v_t=v_0$ while
$v_0,\ldots,v_{t-1}$ are pairwise distinct.  The sequence $w$ is
\emph{valid from $v_0$} if $v_1,\ldots,v_t$ are pairwise distinct.
Thus its partial-sum walk is either a path, or it visits $v_0$ exactly
once after time $0$ and has no other repeated vertex; when this return
occurs at the final step, the walk is a cycle.  For $v_0=0$, validity
from $v_0$ is exactly the usual notion of a valid ordering.  The empty
sequence is a valid ordering of the empty set.  We use
\[
 V(\PathVar)=\{v_0,\ldots,v_t\},\quad
 V^+(\PathVar)=\{v_1,\ldots,v_t\},\quad
 \operatorname{end}(\PathVar)=v_t,\quad
 \operatorname{lab}(\PathVar)=\{w_1,\ldots,w_t\}.
\]
A set is unused relative to $\PathVar$ if it is disjoint from
$\operatorname{lab}(\PathVar)$, and an extension is obtained by concatenating
label sequences.  A \emph{path extension} is one for which the resulting
partial-sum walk is a path.  \journalchg{A coset $C$ is \emph{fresh for $\PathVar$} if
$C\cap(V(\PathVar)\setminus\{v_0\})=\varnothing$.}
The terms regular set, exceptional set, helper, and pending bridge are
used descriptively for the explicit data in the recursive statements.
We use $O_D(\cdot)$, $\gg_D$, and $\asymp_D$ for estimates whose
implicit constants may depend on $D$ and on the fixed parameters
$L,\gamma$; explicitly, $X\gg_D Y$ means $X\geq c_DY$ for some
$c_D>0$ with this dependence.  Conditioning on $O_D(1)$ positions is
uniform over all admissible choices and distinct assigned entries.

Comparability gives, whenever $K$ is non-trivial and $\Omega(K)=r$,
\begin{equation}
\label{eq:least-prime-subgroup}
 P^-(|K|)\geq c_{L,\gamma}|K|^{1/r}.
\end{equation}
If $g\notin K$ and $d=\ord_{G/K}(g+K)$, then likewise
\begin{equation}
\label{eq:clock-lower-bound}
 d\geq c_{L,\gamma}|K|^{1/r}.
\end{equation}
Indeed, each prime factor of $|K|$ is at most
$\gamma P^-(|K|)$, while every non-trivial quotient order is divisible
by a prime at least $p_{\min}$; one may take
$c_{L,\gamma}=\gamma^{-1}$.

Set
\[
 a_0:=1,\qquad a_r:=4^{-r}\quad(1\leq r\leq L),
\]
and choose $\beta=\beta(L)>0$ so that, for every $1\leq r\leq L$,
\begin{equation}
\label{eq:exponent-hierarchy}
 a_r+20\beta<a_{r-1},\qquad
 a_r+20\beta<\frac1r,
 \qquad
 20\beta<\min\left\{a_L,\frac1L,1-a_1\right\}.
\end{equation}
{Such a choice of $\beta$ is possible because the finitely many quantities
\[
a_{r-1}-a_r,\qquad \frac1r-a_r\qquad(1\leq r\leq L),
\]
together with $a_L$, $1/L$, and $1-a_1$, are all positive. Here $a_r$ is the exponent controlling the maximum allowed size $m^{a_r}$ of the exceptional set at a recursive stage with $\Omega(K)=r$. The parameter $\beta$ provides uniform slack for the growth of the already constructed path and for the losses accumulated over at most $L$ descents; the convention $a_0=1$ is used when comparing a stage with the next lower value of $\Omega(K)$. We shall repeatedly use the consequences
\[
a_{r-1}-a_r>20\beta,\qquad 2(a_r+\beta)<1,
\]
and
\[
a_r+3\beta-\frac1r<-17\beta,\qquad a_r+3\beta-\frac1r-\frac{a_r}{2}<-17\beta,\qquad a_r+3\beta-1<-17\beta
\]
for every $1\leq r\leq L$.}
All implicit constants in this section may depend on $L$ and $\gamma$.
A \emph{descent} will simply mean a recursive transition from a regular
set in a coset of $K$ to a concentrated regular subset in a coset of a
proper subgroup $K'<K$, after appending the intervening exceptions and
helpers to the current path.  Hence every descent strictly decreases
$\Omega(K)$.

\subsection{Polynomial non-periodicity and layered local repair}

The recursive descent alternates between probabilistic completion and
further structural concentration.  At a stage with a regular set of
size $m$ in a subgroup $K$ satisfying $\Omega(K)=r$, the first lemma
quantifies the non-periodic branch: unless all but at most $m^{a_r}$
elements lie in a coset of a proper subgroup, it gives an explicit upper
bound for the summed interval-chain parameter and for the accompanying
additive exponential error.  These bounds are kept in polynomial form so
that they remain effective through at most $L$ recursive levels.  We
also need to preserve the part of the ordering already built: when a new
dominant coset is found, its exceptional elements are appended to the
current path rather than restarting the construction.  The two lemmas
below provide precisely these two pieces of the recursive state.

{
\begin{lemma}[Polynomial Kneser alternative]
\label{lem:polynomial-kneser-alternative}
For every fixed $D\geq1$ there exist constants
\[
m_0=m_0(D,L),\qquad C=C(D,L)>0,\qquad c=c(D,L)>0
\]
such that the following holds. Let $K$ be a non-trivial cyclic group with $1\leq\Omega(K)=r\leq L$, and let $X\subseteq K$ have size $m\geq m_0$. Then one of the following alternatives holds.
\begin{enumerate}[label=(\roman*)]
\item There are a non-trivial proper subgroup $K'<K$ and a coset $x+K'$ such that
\[
|X\setminus(x+K')|\leq m^{a_r}.
\]
\item For every non-trivial proper subgroup $H<K$ and every coset $z+H$,
\begin{equation}
\label{eq:polynomial-np-condition}
|X\cap(z+H)|\leq m-m^{a_r}.
\end{equation}
In this case, for every $1\leq h\leq D$, every $z_1,\ldots,z_h\in K$, and every uniformly random nested chain
\[
R_1\subset\cdots\subset R_h\subset X
\]
with prescribed cardinalities $|R_i|=M_i$, one has
\begin{equation}
\label{eq:polynomial-np-chain-summed}
\sum_{1\leq M_1<\cdots<M_h<m}\Pr\bigl[\Sigma(R_i)=z_i\text{ for all }1\leq i\leq h\bigr]\leq C\vartheta_{0,1-a_r}(m;|K|)^h+C\exp(-cm^{a_r}),
\end{equation}
where
\begin{equation}
\label{eq:polynomial-np-parameter}
\vartheta_{0,1-a_r}(m;|K|)=\frac{m}{|K|}+Cm^{-a_r/2}\sqrt{\log(2+m)}.
\end{equation}
Moreover, let $x_1,\ldots,x_m$ be a uniformly random ordering of $X$, and let $F$ be the local set of positions in a right-going or left-going system covered by Corollary~\ref{cor:nonperiodic-interval-chain}, with $|F|\leq100D^2$. Put
\[
\mathcal F_{\mathrm{loc}}:=\sigma(x_j:j\in F).
\]
Then, almost surely, every such local--remote system with $1\leq h\leq D$ satisfies
\begin{equation}
\label{eq:polynomial-np-local-remote}
\sum_{\mathrm{adm}}
\Pr\left[
\sum_{j\in I^{\mathrm{rem}}_\ell}x_j
=z_\ell-Z_\ell
\text{ for all }1\leq\ell\leq h
\,\middle|\,
\mathcal F_{\mathrm{loc}}
\right]
\leq
C\vartheta_{0,1-a_r}(m;|K|)^h
+C\exp(-cm^{a_r}).
\end{equation}
The estimate remains valid when the targets $z_\ell$ are replaced by $\mathcal F_{\mathrm{loc}}$-measurable random variables.
\end{enumerate}
We call~\textup{(ii)} the \emph{polynomial non-periodic alternative}.
\end{lemma}

\begin{proof}
If \eqref{eq:polynomial-np-condition} fails, there exist a non-trivial proper subgroup $H<K$ and a coset $z+H$ such that
\[
|X\cap(z+H)|>m-m^{a_r}.
\]
Taking $K'=H$ and $x=z$ gives~\textup{(i)}.

Suppose that \eqref{eq:polynomial-np-condition} holds. Apply Proposition~\ref{prop:polynomial-nonperiodic-chain} with
\[
T=1,\qquad N=|K|,\qquad \kappa=1-a_r.
\]
Since $t=1$ and $q=|K|$, every non-trivial subgroup of $K$ is $q$-large, and the concentration hypothesis in that proposition is exactly \eqref{eq:polynomial-np-condition}. Moreover,
\[
\frac{m}{|K|}+Cm^{(1-a_r)/2}\frac{\sqrt{\log(2+m)}}{\sqrt m}=\frac{m}{|K|}+Cm^{-a_r/2}\sqrt{\log(2+m)}.
\]
The summed chain estimate in Proposition~\ref{prop:polynomial-nonperiodic-chain} therefore gives \eqref{eq:polynomial-np-chain-summed}, after changing the constants.

For the local--remote conclusion, put
\[
\eta:=m^{a_r-1}.
\]
Since $0<a_r<1$, after increasing $m_0=m_0(D,L)$ if necessary, we may assume that
\[
0<\eta\leq\frac14
\]
and
\[
\eta m=m^{a_r}\geq C\log^2(2+m).
\]
Moreover,
\[
\eta^{-1/2}\frac{\sqrt{\log(2+m)}}{\sqrt m}
=m^{-a_r/2}\sqrt{\log(2+m)}.
\]
The hypotheses of Corollary~\ref{cor:nonperiodic-interval-chain} are therefore satisfied. Its boundedly conditioned conclusion gives \eqref{eq:polynomial-np-local-remote}. Since the local bound $100D^2$ depends only on $D$, all constants continue to depend only on $D$ and $L$.
\end{proof}
}

We shall repeatedly extend a path which has already been constructed.
The next greedy lemma records this operation explicitly.

\begin{lemma}[Greedy extension of a pre-existing head]
\label{lem:preexisting-head}
Let $\PathVar$ be a path in a finite abelian group $G$, with endpoint
$v$ and vertex set $V(\PathVar)$, including the initial vertex.  Let $A,E$ be
disjoint subsets of $G\setminus\{0\}$ which are unused relative to $\PathVar$,
and let $Z\subseteq G$ be a set of prescribed vertices.  If
\[
        |A|\geq 20\bigl(|V(\PathVar)|+|E|+|Z|+1\bigr),
\]
then $\PathVar$ has an extension  $\PathExtVar$ whose associated walk is a path, which uses every element of $E$,
uses at most $|E|+1$ elements of $A$, and satisfies
\[
        \bigl(V(\PathExtVar)\setminus V(\PathVar)\bigr)\cap Z=\varnothing.
\]
The elements of $E$ may be consumed in any prescribed order. If one requires the extension to use at least one element of $A$,
this is automatic when $E\neq\varnothing$; when $E=\varnothing$ and
$A\neq\varnothing$, one admissible unused element of $A$ can be appended.
\end{lemma}

\begin{proof}
Write the prescribed order of $E$ as $e_1,\ldots,e_t$.  Suppose that the
pairs $a_1,e_1,\ldots,a_{i-1},e_{i-1}$ have already been appended, and
let $\PathVar_{i-1}$ be the current path.  To append $a_i,e_i$, it is
enough to require that
\[
 \operatorname{end}(\PathVar_{i-1})+a_i,
 \qquad
 \operatorname{end}(\PathVar_{i-1})+a_i+e_i
\]
avoid $V(\PathVar_{i-1})\cup Z$.  Each vertex in this set forbids at most one
value of $a_i$ through each of the two displayed expressions.  Since
$|V(\PathVar_{i-1})|\leq |V(\PathVar)|+2(i-1)$ and fewer than $i$ elements of $A$ have
already been used, the total number of forbidden choices is at most
\[
 2\bigl(|V(\PathVar)|+2t+|Z|\bigr)+t
 <20\bigl(|V(\PathVar)|+t+|Z|+1\bigr)\leq |A|.
\]
Thus an unused admissible $a_i$ exists.  The two new vertices are
distinct because $e_i\neq0$, so the resulting walk remains a path and creates no
new vertex in $Z$.

After all elements of $E$ have been used, if $E=\varnothing$ and the optional helper is required, append one further unused admissible element of $A$.  At that stage at most
$|V(\PathVar)|+2t+|Z|+t$ values are forbidden by old vertices, prescribed
vertices, and previously used helpers, which is again smaller than
$|A|$.  The resulting path has all the asserted properties.
\end{proof}

We next isolate the local-repair input used in the proper-coset branch.
Let $K\leq G$ have order $N$, let $X\subseteq K$ have size $m$, and fix
$g\in G$ such that $g+K$ has order $d$ in $G/K$.  Put $c:=dg\in K$.

{
Put
\[
 \operatorname{Ord}(X):=\{\sigma:[m]\to X:\sigma\text{ is bijective}\}.
\]
For every $\sigma\in\operatorname{Ord}(X)$ and $0\leq b\leq m$, define
\[
 x_i^\sigma:=\sigma(i),\qquad
 P_b(\sigma):=\sum_{i=1}^b x_i^\sigma,\qquad
 V_b(\sigma):=bg+P_b(\sigma),\qquad
 Q_b(\sigma):=P_b(\sigma)+\left\lfloor\frac bd\right\rfloor c,
\]
with $P_0(\sigma)=V_0(\sigma)=Q_0(\sigma)=0$.  If $b=\nu d+t$ with $0\leq t<d$, then
$V_b(\sigma)=tg+Q_b(\sigma)$.  The terminal value
\[
 Q_m(\sigma)=\sum_{x\in X}x+\left\lfloor\frac md\right\rfloor c
\]
is independent of $\sigma$; we denote this common value simply by $Q_m$.
}

For each $t\in\{0,\ldots,d-1\}$, let $W_t\subseteq K$ be the set of
$K$-coordinates to be avoided by prefixes whose length is congruent to
$t$ modulo $d$, and put
\[
        \mathbf W=(W_0,\ldots,W_{d-1}).
\]
Thus the prefix ending at $b$ is forbidden when
\quad
{$Q_b(\sigma)\in W_{\operatorname{res}_d(b)}$}, where
$\operatorname{res}_d(b)\in\{0,\ldots,d-1\}$ is the least
non-negative residue of $b$ modulo $d$.  We assume the terminal
compatibility conditions
\begin{equation}
\label{eq:terminal-compatibility}
 Q_m\notin W_{\operatorname{res}_d(m)},
 \qquad
 m\equiv0\pmod d\Longrightarrow Q_m\neq0.
\end{equation}
Here $Q_m$, and hence these conditions, is independent of the ordering
of $X$.  The first condition prevents the terminal tail vertex from
meeting the forbidden set in its layer; the second prevents the full
tail from returning to its initial vertex when the quotient clock
returns to its initial coset.

For $0<a<1$, the required non-periodicity condition is
\begin{equation}
\label{eq:clock-chain-nonperiodicity}
        |X\cap(z+H)|\leq m-m^a
\end{equation}
for every non-trivial proper subgroup $H<K$ and every coset $z+H$.
Define
\[
 \vartheta_{d,a}(m,N):=
 \frac{m}{dN}
 +\frac{\sqrt{\log(2+m)}}{d m^{a/2}}
 +\frac1m,
 \qquad
 \Wnorm:=\sum_{t=0}^{d-1}|W_t|,
\]
\[
 \Theta_{d,a}(m,N,\mathbf W):=
 (1+\Wnorm)\vartheta_{d,a}(m,N).
\]

The following statement is a layered version of the local-repair argument developed by Pham and Sauermann in the proof of their polynomial-range theorem \cite[Section~5]{PhamSauermann}. In the present setting, the image of the tail in (G/K) provides a deterministic clock, while the set of forbidden values for a prefix, determined by the vertices already visited by the head, depends on its clock layer. The estimates must also remain valid after conditioning on boundedly many coordinates and under the local permutations arising in the repair procedure, as required by the successive stages of the recursive descent. Its proof is obtained by substantially adapting the Pham--Sauermann procedure to this layered setting, but requires a lengthy bookkeeping framework for local and remote cut points, conditioned chain estimates, and blocker configurations. Since the recursive argument below uses only the resulting statement, we defer the proof to {Section~\ref{sec:layered-local-repair-proof}}.
{
\begin{lemma}[Layered local repair]
\label{lem:layered-local-repair}
Fix $0<a<1$ and $\rho>0$. There exist integers $D=D(\rho)\geq1$ and $m_1=m_1(a,\rho,L)\geq1$ such that the following holds. Retain the notation introduced above: $K\leq G$ is cyclic with
\[
N:=|K|,\qquad \Omega(K)\leq L,
\]
the set $X\subseteq K$ has size $m\geq m_1$, the element $g\in G$ has image of order
\[
d:=\operatorname{ord}_{G/K}(g+K)
\]
in $G/K$, and
\[
c:=dg\in K.
\]

{Let $\mathbf W=(W_0,\ldots,W_{d-1})$.  The hypotheses below depend only on $K$, $X$, $g$, and $\mathbf W$, not on a preliminary ordering of $X$.} Assume
\[
0\notin g+X,
\]
conditions \eqref{eq:clock-chain-nonperiodicity} and \eqref{eq:terminal-compatibility}, and
\begin{equation}
\label{eq:clock-repair-smallness}
\Theta_{d,a}(m,N,\mathbf W)\leq m^{-\rho}.
\end{equation}

{Then there exists $\sigma_*\in\operatorname{Ord}(X)$ such that the pairs
\[
(Q_j(\sigma_*),\operatorname{res}_d(j)),\qquad 0\leq j\leq m,
\]
are pairwise distinct. Equivalently, the tail vertices
$V_0(\sigma_*),\ldots,V_m(\sigma_*)$ are pairwise distinct. Moreover,
\[
Q_j(\sigma_*)\notin W_{\operatorname{res}_d(j)}
\]
for every $1\leq j\leq m$.}
\end{lemma}
}

\subsection{Kneserized terminal correction}

\journalchg{We now treat the subgroup-dominant part of reverse absorption.  The exceptional elements are ordered
in the quotient and the final one is kept as a bridge into a quotient
coset not visited by the preceding head.  Before crossing that bridge we
must modify the regular set inside the subgroup so that the remaining
elements can be completed by the cycle trick.  The correction lemma has
a structural alternative: either at most three regular elements can be
placed before the bridge so that the remaining set satisfies the hypothesis
of the cycle trick, or ordinary Kneser theory shows that almost all regular
elements lie in a coset of a proper subgroup.  The latter outcome supplies
the next recursive descent.}

\journalchg{We first record the quotient-ordering lemma used below.}

\begin{lemma}[Small quotient head]
\label{lem:small-quotient-head}
Let $t\geq1$, and let $Q$ be a finite abelian group in which every
non-zero element has order greater than $t$.  Every multiset of $t$ non-zero elements of $Q$
can be ordered as $q_1,\ldots,q_t$ so that
\[
        q_1+\cdots+q_j\neq q_1+\cdots+q_t
        \qquad(1\leq j<t).
\]
Equivalently, after reversing the order, every partial sum with fewer
than $t$ terms is non-zero.
\end{lemma}

\begin{proof}
We prove the equivalent non-zero-prefix formulation by induction on
$t$.   {Let $\Sigma_{\rm tot}$ be the total sum.}  If $t=1$ there is nothing to prove.
For $t>1$, choose an entry $a\neq{\Sigma_{\rm tot}}$.  Such an entry exists: if every
entry were equal to ${\Sigma_{\rm tot}}$, then $(t-1){\Sigma_{\rm tot}}=0$, contrary to the order
assumption.  Order the remaining multiset inductively and put $a$ last.
All proper prefixes contained in the first $t-1$ entries are non-zero by
induction, while their full sum is ${\Sigma_{\rm tot}}-a\neq0$.  Reversing this
ordering gives the stated formulation.
\end{proof}

\journalchg{The following lemma provides the terminal correction needed when
the subgroup has composite order.  Its proof combines
Theorem~\ref{thm:huicochea-restricted} with ordinary Kneser theory: in
the remaining case, the first result forces the restricted and ordinary
double sumsets to coincide, and the second yields concentration in a
coset of a proper subgroup.}

\begin{lemma}[Terminal correction or genuine descent]
\label{lem:terminal-correction-descent}
Let $G$ be a finite abelian group, let $K\leq G$ be a cyclic subgroup
of odd order, and let $\PathVar$ be a path in $G$ with endpoint $v$ and
vertex set $V(\PathVar)$.  Let $R\subseteq K\setminus\{0\}$ be a set of unused
elements, let $e\in G\setminus K$ be unused, and assume that the coset
$v+e+K$ is fresh for $\PathVar$.  Let $Z\subseteq G$ satisfy $|Z|\leq1$.  Put
\[
        m:=|R|,
        \qquad
        u:=|V(\PathVar)|+|Z|.
\]
Assume
\begin{equation}
\label{eq:terminal-correction-size}
        m\geq64(u+2)^2.
\end{equation}
Then one of the following holds.
\begin{enumerate}[label=(\roman*)]
\item There are distinct elements $d_1,\ldots,d_t\in R$, with
      $0\leq t\leq3$, such that the extension
      \[
            \PathVar':=\PathVar\text{ extended by }(d_1,\ldots,d_t,e)
      \]
      is a path and
      $(V(\PathVar')\setminus V(\PathVar))\cap Z=\varnothing$.  For
      $\Delta:=\{d_1,\ldots,d_t\}$ and $C:=R\setminus\Delta$ one has
      \[
            \SumC:=\sum_{c\in C}c\neq0,
            \qquad
            -\SumC\notin C.
      \]
\item There are a non-trivial proper subgroup $H<K$ and a coset $x+H$
      such that
      \[
            |R\setminus(x+H)|\leq u.
      \]
\end{enumerate}
\end{lemma}

\begin{proof}
Put
\[
        \SumR:=\sum_{r\in R}r
\]
and let
\[
        W:=\{a\in K:v+a\in V(\PathVar)\cup Z\}.
\]
Thus $|W|\leq u$.  We first treat the dense case
\begin{equation}
\label{eq:terminal-dense}
        |K|\leq m+u+4.
\end{equation}
By inclusion--exclusion,
\[
        |R\cap(\SumR-R)|\geq2m-|K|\geq m-u-4.
\]
Since $K$ has odd order, the equation $2z=\SumR$ has at most one
solution.  Moreover,
\[
        m-u-4>|W|+1
\]
by \eqref{eq:terminal-correction-size}.  We may therefore choose
$z\in R\cap(\SumR-R)$ outside $W$ and outside the possible solution
of $2z=\SumR$.
Put $y:=\SumR-z$.  Then $z,y$ are distinct elements of $R$ and
\[
        z+y=\SumR,
        \qquad
        v+z\notin V(\PathVar)\cup Z.
\]
Choose
\[
        x\in R\setminus\{z,y,-y\}
\]
so that the vertices
\[
        v+z+x,
        \qquad
        v+z+x+y
\]
avoid $V(\PathVar)\cup Z$, and so that the vertex
$v+z+x+y+e$ avoids the initial vertex of $\PathVar$ and the possible element of
$Z$.  The first two requirements exclude at most $2u$ values of $x$.
The first new vertex $v+z$ already avoids $V(\PathVar)\cup Z$ because
$z\notin W$.  Since $x,y\neq0$, the second new vertex differs from the
first and the third differs from the second; the third equals the first
only when $x=-y$, which has already been excluded.

The final vertex belongs to $v+e+K$.  This coset contains no non-initial
vertex of $\PathVar$ because it is fresh for $\PathVar$, and it is disjoint from
$v+K$, which contains the three preceding new vertices, because
$e\notin K$.  Requiring the final vertex to differ from the initial
vertex and from the possible element of $Z$ excludes at most two further
values of $x$.  Together with the three explicitly removed values
$z,y,-y$, at most $2u+5$ values are forbidden.  The size condition
\eqref{eq:terminal-correction-size} therefore guarantees such a choice
of $x$.  Append the labels $z,x,y,e$ to $\PathVar$, and put
$\Delta:=\{z,x,y\}$.  Then
\[
        \sum_{c\in R\setminus\Delta}c=-x.
\]
Thus $\SumC=-x\neq0$ and $-\SumC=x\notin C$, so
conclusion~(i) holds.

Assume now that \eqref{eq:terminal-dense} fails.  Put
\[
        R_0:=R\setminus W,
        \qquad
        m_0:=|R_0|.
\]
Let $B$ consist of
\[
        W,
        \qquad
        \SumR+(R\cup\{0\}),
\]
together with those values $s\in K$ for which $v+s+e$ is either the
initial vertex of $\PathVar$ or an element of $Z$.  There is at most one value
of $s$ for the initial vertex and at most one for the possible element
of $Z$.  Thus
\begin{equation}
\label{eq:terminal-forbidden-size}
        |B|\leq m+u+3<|K|.
\end{equation}
If there is
\[
        s\in\widehat2R_0\setminus B,
\]
write $s=x+y$ with distinct $x,y\in R_0$, append the labels $x,y,e$ to
$\PathVar$, and put $\Delta:=\{x,y\}$.  Since $x\notin W$ and $s\notin W$, the
vertices $v+x$ and $v+s$ do not belong to $V(\PathVar)\cup Z$; they are
distinct because $y\neq0$.  The final vertex $v+s+e$ lies in the fresh
coset $v+e+K$, and therefore differs from every non-initial vertex of
$\PathVar$.  It also differs from $v+x$ and $v+s$, which lie in $v+K$, and the
definition of $B$ ensures that it is neither the initial vertex nor an
element of $Z$.  Thus the resulting extension is a path and creates no
new vertex in $Z$.  Moreover, for $C=R\setminus\{x,y\}$,
\[
        \SumC=\SumR-s\neq0,
        \qquad
        -\SumC=s-\SumR\notin R,
\]
so conclusion~(i) follows.

It remains to suppose that
\begin{equation}
\label{eq:restricted-contained-forbidden}
        \widehat2R_0\subseteq B.
\end{equation}
The size condition gives $m_0\geq m-u\geq2$.  If
$\widehat2R_0\neq R_0+R_0$, Theorem~\ref{thm:huicochea-restricted},
applied with $A=R_0$, gives
\[
        |\widehat2R_0|
        \geq2m_0-2\sqrt{2m_0}-3.
\]
Since $m_0\geq m-u$ and $m\geq64(u+2)^2$, we have
$u\leq\sqrt m/8-2$ and hence
\[
 2m_0-2\sqrt{2m_0}-3
 \geq 2m-2u-2\sqrt{2m}-3
 >m+u+3.
\]
This contradicts
\eqref{eq:terminal-forbidden-size} and
\eqref{eq:restricted-contained-forbidden}.  Consequently,
\begin{equation}
\label{eq:restricted-equals-ordinary}
        \widehat2R_0=R_0+R_0.
\end{equation}

Let
\[
        H:=\Stab(R_0+R_0),
        \qquad
        h:=|H|,
        \qquad
        q:=|R_0+H|/h.
\]
Kneser's theorem, together with
\eqref{eq:restricted-contained-forbidden} and
\eqref{eq:restricted-equals-ordinary}, gives
\begin{equation}
\label{eq:terminal-kneser-count}
        (2q-1)h
        \leq |R_0+R_0|
        \leq m+u+3
        \leq m_0+2u+3.
\end{equation}
On the other hand $m_0\leq qh$.  If $q\geq2$, then
\eqref{eq:terminal-kneser-count} yields
\[
        (q-1)h\leq2u+3.
\]
Since $q\geq2$, also $h\leq(q-1)h$, and therefore
\[
        qh=h+(q-1)h\leq2(q-1)h\leq4u+6.
\]
This contradicts $m_0\geq m-u>4u+6$, which follows from
\eqref{eq:terminal-correction-size}.  Hence $q=1$, so $R_0$ is contained
in a single coset $x+H$.

The subgroup $H$ is proper: otherwise
$R_0+R_0=K$, and then
\eqref{eq:restricted-contained-forbidden} and
\eqref{eq:restricted-equals-ordinary} would give $K\subseteq B$,
contrary to \eqref{eq:terminal-forbidden-size}.  It is non-trivial
because $m_0>1$ and $R_0$ is contained in a single $H$-coset.  Finally,
\[
        |R\setminus(x+H)|
        \leq |R\setminus R_0|
        \leq |W|
        \leq u,
\]
which is conclusion~(ii).
\end{proof}

\begin{lemma}[Cycle trick]
\label{lem:cycle-trick}
Let $K$ be a finite abelian group and $C\subseteq K\setminus\{0\}$.
Put
\[
        d_0:=-\sum_{c\in C}c
\]
and assume $d_0\notin C\cup\{0\}$.  If $C\cup\{d_0\}$ has a valid
ordering, then $C$ has an ordering whose non-empty partial sums are
pairwise distinct and non-zero.
\end{lemma}

\begin{proof}
The total sum of $C\cup\{d_0\}$ is zero.  Hence a valid ordering of
this set traces a cycle whose vertices, apart from the coincidence of the
initial and final vertices, are pairwise distinct.  Start this cycle
immediately after the edge labelled $d_0$.  The resulting cyclic rotation
is again valid and has the form
\[
        c_1,\ldots,c_m,d_0.
\]
The vertices reached after $c_1,\ldots,c_j$, $1\leq j\leq m$, are
therefore pairwise distinct.  None is the initial vertex, since the final
edge $d_0$ returns to that vertex and validity makes the final vertex
distinct from all preceding non-initial vertices.  Thus
$c_1,\ldots,c_m$ has pairwise distinct non-zero partial sums.
\end{proof}

\subsection{Recursive completion}

The three propositions below implement the recursive completion.  The
first treats a set concentrated in a proper non-zero coset, the second
treats a subgroup tail together with one element outside the subgroup
that is reserved for the final bridge, and the third reduces the
subgroup-dominant case to the second.  Every recursive step replaces the current subgroup by a proper
subgroup and therefore decreases $\Omega$.  We choose the constants in
the following order: first the proper-coset constants, inductively in
$r$; then the pending-bridge constants, again inductively in $r$ and
using the proper-coset constants for smaller values; finally the
subgroup-dominant constants.

We begin with the proper-coset case.  The initial path is included in the
statement because the proposition is also used after a descent.

\begin{proposition}[Proper-coset completion with an initial head]
\label{prop:proper-coset-recursive}
For every $1\leq r\leq L$ there exist constants
$M_{\rm pc}=M_{\rm pc}(r,L,\gamma)$ and
$P_{\rm pc}=P_{\rm pc}(r,L,\gamma)$ such that the following holds when
the smallest prime base satisfies $p_{\min}\geq P_{\rm pc}$.
Let $K<G$ be cyclic with $\Omega(K)=r$, let $g\notin K$, and let
$\PathVar$ be a path.  Let $A,E\subseteq G\setminus\{0\}$ be disjoint from
each other and from $\operatorname{lab}(\PathVar)$, and assume that
$A\cup E$ is exactly the set of elements that remain to be appended.
Write
\[
        A=g+X_0\subseteq g+K,
        \qquad X_0\subseteq K,
\]
and put
\[
        \Sigma:=\operatorname{end}(\PathVar)+\sum_{a\in A}a+
                   \sum_{e\in E}e.
\]
Assume that $|A|\geq M_{\rm pc}$, that
$\Sigma\notin V^+(\PathVar)$, and
\begin{equation}
\label{eq:proper-state-bound}
        |V(\PathVar)|+|E|+1\leq |A|^{a_r+\beta}.
\end{equation}
Then $\PathVar$ has an extension using every element of $A\sqcup E$ exactly
once and having pairwise distinct non-initial vertices.
\end{proposition}

\begin{proof}
{
We prove the proposition by induction on $r$. Assume that
$M_{\rm pc}(r',L,\gamma)$ and $P_{\rm pc}(r',L,\gamma)$ have already
been chosen for every $1\leq r'<r$. Put
\[
D_r:=D(2\beta),
\]
where $D(2\beta)$ is supplied by
Lemma~\ref{lem:layered-local-repair}. Choose
\[
P_{\rm pc}(r,L,\gamma)\geq
\max_{1\leq r'<r}P_{\rm pc}(r',L,\gamma),
\]
where the maximum is omitted when $r=1$. Next choose
$M_{\rm pc}(r,L,\gamma)$ sufficiently large that
\[
M_{\rm pc}(r,L,\gamma)\geq
4\max_{1\leq r'<r}M_{\rm pc}(r',L,\gamma),
\]
where again the maximum is omitted when $r=1$, and
\[
\frac12M_{\rm pc}(r,L,\gamma)\geq
\max\bigl\{m_0(D_r,L),m_1(a_r,2\beta,L)\bigr\}.
\]
Increase $M_{\rm pc}(r,L,\gamma)$ further, if necessary, so that all
numerical inequalities used below hold whenever
$|A|\geq M_{\rm pc}(r,L,\gamma)$.
}

Since $a_r+\beta<1$, the state bound
\eqref{eq:proper-state-bound} implies the size hypothesis of
Lemma~\ref{lem:preexisting-head} by the choice of $M_{\rm pc}(r,L,\gamma)$.  Apply
that lemma with $Z=\{\Sigma\}$.  It gives a path extension  $\PathExtVar$ of $\PathVar$
which uses every element of $E$ and at least one element of $A$, and
satisfies $\Sigma\notin V^+(\PathExtVar)$.  The elements of $A$ not used in this
extension form a set
\[
        g+X,
        \qquad
        m:=|X|.
\]
Since the extension uses at most $|E|+1$ elements of $A$, the state
bound and the choice of $M_{\rm pc}(r,L,\gamma)$ give
\[
        m\geq |A|-|E|-1
        \geq |A|-|A|^{a_r+\beta}
        \geq |A|/2
        \geq M_{\rm pc}(r,L,\gamma)/2.
\]
{
The extension from $\PathVar$ to $\PathExtVar$ creates at most $2|E|+1$ new vertices. Hence
\[
|V(\PathExtVar)|\leq |V(\PathVar)|+2|E|+1\leq2\bigl(|V(\PathVar)|+|E|+1\bigr)\leq2|A|^{a_r+\beta}.
\]
Since $m\geq|A|/2$, after increasing $M_{\rm pc}(r,L,\gamma)$ we have $2|A|^{a_r+\beta}\leq m^{a_r+3\beta}$. Therefore
\begin{equation}
\label{eq:head-size-proper}
|V(\PathExtVar)|\leq m^{a_r+3\beta}.
\end{equation}

Apply Lemma~\ref{lem:polynomial-kneser-alternative} to $X\subseteq K$ with chain depth $D=D_r$. Since
\[
m\geq\frac{|A|}{2}\geq\frac{M_{\rm pc}(r,L,\gamma)}{2}\geq m_0(D_r,L),
\]
its size hypothesis is satisfied.
}

Assume first that the non-periodic conclusion holds.  Put
\[
        d:=\ord_{G/K}(g+K),
        \qquad
        c:=dg\in K.
\]
For $0\leq t<d$, define
\begin{equation}
\label{eq:layered-head-targets}
        W_t:=\{w-\operatorname{end}(\PathExtVar)-tg:
        w\in V^+(\PathExtVar),\ 
        w-\operatorname{end}(\PathExtVar)-tg\in K\},
\end{equation}
where $V^+(\PathExtVar)$ denotes the vertices of $\PathExtVar$ other than its initial
vertex.  Then
\[
        \Wnorm:=\sum_{t=0}^{d-1}|W_t|
        \leq |V^+(\PathExtVar)|
        \leq m^{a_r+3\beta}.
\]
Indeed, the quotient classes $tg+K$, $0\leq t<d$, are distinct, so a
head vertex contributes to at most one of the sets $W_t$.

{
Since $g+X$ is exactly the set of labels remaining after $\PathExtVar$, the terminal tail vertex is
{$\Sigma$.} Equivalently, using only order-independent quantities,
\[
{\operatorname{end}(\PathExtVar)+mg+\sum_{x\in X}x
=\operatorname{end}(\PathExtVar)+\operatorname{res}_d(m)g+Q_m=\Sigma.}
\]
Consequently, $Q_m\in W_{\operatorname{res}_d(m)}$ would imply $\Sigma\in V^+(\PathExtVar)$, contrary to the construction of $\PathExtVar$. This proves the first condition in \eqref{eq:terminal-compatibility}. If $m\equiv0\pmod d$ and $Q_m=0$, then the terminal tail vertex equals $\operatorname{end}(\PathExtVar)$. Since the construction of $\PathExtVar$ uses at least one element of $A$, one has $\operatorname{end}(\PathExtVar)\in V^+(\PathExtVar)$, again contradicting the first terminal condition. Thus the second condition also holds.

}

By \eqref{eq:clock-lower-bound} and $m\leq |K|$,
\[
        d\geq c_{L,\gamma}|K|^{1/r}
        \geq c_{L,\gamma}m^{1/r}.
\]
By the inequalities in \eqref{eq:exponent-hierarchy},
\[
 a_r+3\beta-\frac1r<-17\beta,
 \qquad
 a_r+3\beta-\frac1r-\frac{a_r}{2}<-17\beta,
 \qquad
 a_r+3\beta-1<-17\beta.
\]
Consequently, by the choice of $M_{\rm pc}(r,L,\gamma)$ so that
$\sqrt{\log(2+m)}\leq m^{\beta}$,
\begin{align*}
        \Theta_{d,a_r}(m,|K|,\mathbf W)
        &\leq
        C m^{a_r+3\beta}
        \left(m^{-1/r}
        +m^{-1/r-a_r/2}\sqrt{\log(2+m)}
        +m^{-1}\right)\\
        &\leq m^{-2\beta}.
\end{align*}
Since $g+X\subseteq A\subseteq G\setminus\{0\}$, one has
$0\notin g+X$.
Lemma~\ref{lem:layered-local-repair}, applied with
$a=a_r$ and $\rho=2\beta$, gives an ordering of $X$ whose tail vertices
are pairwise distinct and avoid $V^+(\PathExtVar)$ by the definition of the sets
in \eqref{eq:layered-head-targets}.  Hence the new non-initial vertices
are distinct from one another and from all non-initial vertices of $\PathExtVar$.
A possible equality with the initial vertex of $\PathVar$ can occur for at most
one tail vertex and is allowed by the definition of validity.  Appending
the corresponding elements $g+x$ therefore gives the required extension
of $\PathExtVar$.

{In the periodic conclusion there are a non-trivial proper subgroup $K'<K$ and a coset $x_0+K'$ such that
\[
        |X\setminus(x_0+K')|\leq m^{a_r}.
\]
In particular, this periodic branch cannot occur when $r=1$. } Set
\[
        A':=g+(X\cap(x_0+K'))
        \subseteq (g+x_0)+K'
\]
and put
\[
        E':=(g+X)\setminus A'.
\]
Since $g\notin K$, we have $g+x_0\notin K'$.  Put $r':=\Omega(K')$, so that $1\leq r'\leq r-1$.  For all sufficiently large $m$ one has $|A'|\geq m-m^{a_r}\geq m/2$. 

{Using \eqref{eq:head-size-proper} and $|E'|\leq m^{a_r}$, we have, for all sufficiently large $m$,
\[
|V(\PathExtVar)|+|E'|+1\leq m^{a_r+3\beta}+m^{a_r}+1\leq 3m^{a_r+3\beta}.
\]
Since $(a_{r-1}+\beta)-(a_r+3\beta)=a_{r-1}-a_r-2\beta>18\beta$ and $|A'|\geq m/2$, our choice of $M_{\rm pc}(r,L,\gamma)$ ensures both $|A'|\geq M_{\rm pc}(r',L,\gamma)$ and
\[
|V(\PathExtVar)|+|E'|+1\leq 3m^{a_r+3\beta}\leq |A'|^{a_{r-1}+\beta}\leq |A'|^{a_{r'}+\beta}.
\]}

The final vertex obtained after using all remaining elements is still
$\Sigma$, and $\Sigma\notin V^+(\PathExtVar)$.  Hence the induction hypothesis for
parameter $r'$, with $K'$ in place of $K$, applies.  Since $r'<r$, the
recursion terminates.
\end{proof}

We next treat a subgroup tail together with one element outside the
subgroup, reserved for the final bridge.  We assume that the coset entered
through this element is disjoint from $V^+(\PathVar)$.

\begin{proposition}[Subgroup tail with a pending bridge]
\label{prop:subgroup-pending-bridge}
For every $1\leq r\leq L$ there exist constants
$M_{\rm br}=M_{\rm br}(r,L,\gamma)$ and
$P_{\rm br}=P_{\rm br}(r,L,\gamma)$ such that the following holds when
$p_{\min}\geq P_{\rm br}$.  Let $K<G$ be cyclic with $\Omega(K)=r$, and
assume that, for every subgroup $H\leq K$ and every subset
$T\subseteq H\setminus\{0\}$, the set $T$ has a valid ordering.  Let $\PathVar$ be a path
with endpoint $v$, let $R\subseteq K\setminus\{0\}$ be disjoint from
$\operatorname{lab}(\PathVar)$, and let
$e\notin K\cup\operatorname{lab}(\PathVar)$.  Assume that
\[
        (v+e+K)\cap V^+(\PathVar)=\varnothing.
\]
Put
\[
        \Sigma:=v+e+\sum_{x\in R}x
\]
and assume that $|R|\geq M_{\rm br}$, that
$\Sigma\notin V^+(\PathVar)$, and
\begin{equation}
\label{eq:subgroup-state-bound}
        |V(\PathVar)|+1\leq |R|^{a_r+\beta}.
\end{equation}
Then $\PathVar$ has an extension using $e$ and every element of $R$ exactly
once and having pairwise distinct non-initial vertices.
\end{proposition}

\begin{proof}
{We prove the proposition by induction on $r$. Assume that the constants
$M_{\rm br}(r',L,\gamma)$ and $P_{\rm br}(r',L,\gamma)$ have already
been chosen for every $1\leq r'<r$, and recall that all proper-coset
constants have already been chosen. Choose
\[
P_{\rm br}(r,L,\gamma)\geq
\max\left\{
3,\,
\max_{1\leq r'<r}P_{\rm br}(r',L,\gamma),\,
\max_{1\leq r'<r}P_{\rm pc}(r',L,\gamma)
\right\},
\]
where the maxima over $1\leq r'<r$ are omitted when $r=1$. Next choose
\[
M_{\rm br}(r,L,\gamma)\geq
2\max_{1\leq r'<r}
\left\{
M_{\rm br}(r',L,\gamma),
M_{\rm pc}(r',L,\gamma)
\right\},
\]
where this maximum is omitted when $r=1$. Increase
$M_{\rm br}(r,L,\gamma)$ further, if necessary, so that the size
condition in Lemma~\ref{lem:terminal-correction-descent}, all numerical
inequalities used below, and all lower-cardinality requirements hold
whenever $|R|\geq M_{\rm br}(r,L,\gamma)$.} Since
$p_{\min}\geq P_{\rm br}(r,L,\gamma)\geq3$, every prime divisor of
$|K|$ is odd; hence $K$ has odd order, as required in
Lemma~\ref{lem:terminal-correction-descent}.

Apply Lemma~\ref{lem:terminal-correction-descent} to $\PathVar,R,e$ with
$Z=\{\Sigma\}$.  Its numerical hypothesis follows from
\eqref{eq:subgroup-state-bound} and
$2(a_r+\beta)<1$, which follows from
\eqref{eq:exponent-hierarchy}: for all sufficiently large $|R|$,
\[
  |V(\PathVar)|+|Z|+2\leq |R|^{a_r+\beta}+2
  \leq \frac18|R|^{1/2},
\]
and hence $|R|\geq64(|V(\PathVar)|+|Z|+2)^2$.

Suppose first that conclusion~\textup{(i)} of
Lemma~\ref{lem:terminal-correction-descent} gives distinct elements
$d_1,\ldots,d_t$, in the displayed order, and put
\[
        \Delta:=\{d_1,\ldots,d_t\},
        \qquad
        C:=R\setminus\Delta,
        \qquad
        d_0:=-\sum_{c\in C}c.
\]
By the hypothesis on subgroups of $K$, the set
$C\cup\{d_0\}$ has a valid ordering.
Lemma~\ref{lem:cycle-trick} therefore gives an ordering of $C$ with
pairwise distinct non-zero partial sums.  Append $d_1,\ldots,d_t,e$ and
then the ordering of $C$ supplied by Lemma~\ref{lem:cycle-trick}; when
$t=0$, the first block consists only of $e$.  By
Lemma~\ref{lem:terminal-correction-descent}\textup{(i)}, appending the
first block $d_1,\ldots,d_t,e$ produces a path and none of its new
vertices equals $\Sigma$.  The vertices added before $e$ lie in $v+K$,
whereas all vertices added after $e$ lie in $v+e+K$.  These two cosets are
disjoint, and the latter contains no vertex of $V^+(\PathVar)$ by hypothesis.
The ordering of $C$ has pairwise distinct non-zero partial sums, so its
vertices are pairwise distinct and do not return to the vertex reached
after $e$.  The last vertex of the full extension is $\Sigma$.  Since
$\Sigma\notin V^+(\PathVar)$, if the initial vertex of $\PathVar$
lies in $v+e+K$, at most one of the pairwise distinct tail vertices can
equal it, which is allowed because validity only requires the non-initial vertices to be pairwise distinct.  Hence all
non-initial vertices of the completed extension are pairwise distinct.

Suppose next that the lemma gives a proper subgroup $H<K$ and a coset
$x+H$ containing all but at most
\[
        u:=|V(\PathVar)|+1
\]
elements of $R$.  Put
\[
        A_0:=R\cap(x+H),
        \qquad
        F:=R\setminus A_0.
\]
Let $v_0$ be the initial vertex of $\PathVar$ and put
\[
        Z':=\{\Sigma,\,v_0-e,\,\Sigma-e\}.
\]
Since $|F|\leq u$, $|A_0|\geq |R|-u$, and
$u\leq |R|^{a_r+\beta}$, the size hypothesis of
Lemma~\ref{lem:preexisting-head} holds for $(\PathVar,A_0,F,Z')$ by the choice
of $M_{\rm br}(r,L,\gamma)$.  Apply that lemma with forbidden set $Z'$ to extend
$\PathVar$ by every element of $F$ and by at most $|F|+1$ elements chosen from
$A_0$; require at least one such element. 

{
Let $\PathVar_1$ be the resulting path, let $v_1:=\operatorname{end}(\PathVar_1)$, and let
\[
        A_1:=A_0\setminus\operatorname{lab}(\PathVar_1)
\]
be the elements of $A_0$ that remain to be used.

Since $\PathVar_1$ is obtained from $\PathVar$ by appending every element of $F$ and precisely the elements of $A_0\setminus A_1$, while $R=F\sqcup A_0$, conservation of the total sum gives
\[
        v_1+e+\sum_{a\in A_1}a
        =\operatorname{end}(\PathVar)+e+\sum_{x\in R}x
        =\Sigma.
\]
Moreover, $\Sigma\notin V^+(\PathVar)$ by hypothesis, and every new vertex of $\PathVar_1$ avoids $Z'\supseteq\{\Sigma\}$. Hence $\Sigma\notin V^+(\PathVar_1)$.

The endpoint $v_1$ is one of the new vertices and the new vertices avoid $Z'$, so $v_1\notin\{v_0-e,\Sigma-e\}$.
}

Hence appending $e$ cannot
produce either $v_0$ or $\Sigma$.  Every vertex in
$V(\PathVar_1)\setminus V(\PathVar)$ lies in $v+K$, while
$v_1+e+K=v+e+K$.  Therefore
\[
        (v_1+e+K)\cap V^+(\PathVar_1)=\varnothing.
\]
Indeed, the intersection with $V^+(\PathVar)$ is empty by the hypothesis on $\PathVar$, and
the new vertices of $\PathVar_1$ lie in the disjoint coset $v+K$.  Moreover,
at most $|F|+1\leq u+1$ elements of $A_0$ have been used in constructing
$\PathVar_1$; hence, for all sufficiently large $|R|$,
\[
        |A_1|\geq |R|-2u-1\geq |R|/2.
\]
Put $r':=\Omega(H)$, so that $1\leq r'\leq r-1$. {
The construction of $\PathVar_1$ adds at most $2|F|+1$ vertices. Since $|F|\leq u$, where $u=|V(\PathVar)|+1\leq |R|^{a_r+\beta}$, and since $u\geq2$, we have
\[
|V(\PathVar_1)|+2\leq |V(\PathVar)|+2|F|+3\leq3u+2\leq4|R|^{a_r+\beta}.
\]
Since $|A_1|\geq |R|/2$ and $(a_{r-1}+\beta)-(a_r+\beta)=a_{r-1}-a_r>20\beta$, our choice of $M_{\rm br}(r,L,\gamma)$ ensures the lower-cardinality requirements for the recursive call and
\begin{equation}
\label{eq:subgroup-descent-size}
|V(\PathVar_1)|+2\leq |A_1|^{a_{r-1}+\beta}\leq |A_1|^{a_{r'}+\beta}.
\end{equation}
}

If $x\in H$, then $A_1\subseteq H$.  Apply the induction hypothesis for
this proposition with parameter $r'$, subgroup $H$, path $\PathVar_1$, outside
element $e$, and remaining set $A_1$.  The valid-ordering hypothesis is
inherited by the subgroups of $H$, and
\[
        (v_1+e+H)\cap V^+(\PathVar_1)=\varnothing
\]
because $v_1+e+H\subseteq v_1+e+K$.

{If $x\notin H$, append $e$ to $\PathVar_1$ and denote the resulting walk by $\PathVar_2$. Its new endpoint belongs to $v_1+e+K$, which is disjoint from $V^+(\PathVar_1)$, and it is not the initial vertex because $\PathVar_1$ avoids $v_0-e$. Hence $\PathVar_2$ is a path. By the preceding total-sum identity and $\operatorname{end}(\PathVar_2)=v_1+e$, one has
\[
        \operatorname{end}(\PathVar_2)+\sum_{a\in A_1}a=\Sigma.
\]
The remaining elements $A_1$ lie in the proper non-zero coset $x+H$. Since $\Sigma\notin V^+(\PathVar_1)$ and the new endpoint $v_1+e$ is not $\Sigma$, one also has $\Sigma\notin V^+(\PathVar_2)$.
}
Moreover,
\eqref{eq:subgroup-descent-size} gives the state bound required by the
instance of Proposition~\ref{prop:proper-coset-recursive} with parameter
$r'$, because
$|V(\PathVar_2)|+1=|V(\PathVar_1)|+2$.  Applying that proposition with $A_1$ as its proper-coset set and with
$E=\varnothing$ completes the ordering.  In both cases
$r'<r$, so the recursion terminates.
\end{proof}

\begin{proposition}[Subgroup-dominant completion]
\label{prop:subgroup-layered-completion}
For every $1\leq r\leq L$ there exist constants
$M_{\rm sg}=M_{\rm sg}(r,L,\gamma)$ and
$P_{\rm sg}=P_{\rm sg}(r,L,\gamma)$ such that the following holds when
$p_{\min}\geq P_{\rm sg}$.  Let $K<G$ be a proper cyclic subgroup with
$\Omega(K)=r$, and assume that, for every subgroup $H\leq K$ and every subset
$T\subseteq H\setminus\{0\}$, the set $T$ has a valid ordering.  Let $S\subseteq G\setminus\{0\}$ and put
\[
        S=A\sqcup E,
        \qquad
        A:=S\cap K,
        \qquad
        E:=S\setminus K.
\]
Assume that $|A|\geq M_{\rm sg}$ and
\[
        |E|\leq |A|^{a_r}.
\]
Then $S$ has a valid ordering.
\end{proposition}

\begin{proof}
Choose
$P_{\rm sg}(r,L,\gamma)\geq P_{\rm br}(r,L,\gamma)$ and choose
$M_{\rm sg}(r,L,\gamma)$ large enough that every remaining subset
$R\subseteq A$ with $|R|\geq |A|/2$ occurring below satisfies
$|R|\geq M_{\rm br}(r,L,\gamma)$, as well as all displayed asymptotic
inequalities.  If $E=\emptyset$, apply the assumed valid-ordering hypothesis with
$H=K$ and $T=A$.  Assume $E\neq\emptyset$ and put $t:=|E|$.  The projections of the elements of
$E$ to $G/K$ are non-zero.  Since $K<G$, the quotient is non-trivial and
\[
        P^-(|G/K|)\geq p_{\min}.
\]
Moreover, $|A|\leq |K|\leq(\gamma p_{\min})^r$, whence
$p_{\min}\geq\gamma^{-1}|A|^{1/r}$.  The second inequality in
\eqref{eq:exponent-hierarchy} gives $a_r<1/r$.  By choosing
$M_{\rm sg}(r,L,\gamma)$ sufficiently large, the hypothesis
$|E|\leq |A|^{a_r}$ therefore yields
\[
 t=|E|\leq |A|^{a_r}
 <\gamma^{-1}|A|^{1/r}
 \leq p_{\min}
 \leq P^-(|G/K|).
\]
Every non-zero element of $G/K$ consequently has order greater than $t$.
{
For each $e\in E$, write
\[
\overline e:=e+K\in G/K.
\]
Lemma~\ref{lem:small-quotient-head} gives an ordering $e_1,\ldots,e_t$ of $E$ such that
\[
\overline e_1+\cdots+\overline e_j
\neq
\overline e_1+\cdots+\overline e_t
\qquad(1\leq j<t).
\]
Put
\[
\BarSum:=\overline e_1+\cdots+\overline e_t\in G/K.
\]
}

Do not use $e_t$ in the initial path construction, and distinguish two cases.  If
$\BarSum\neq0$, start with the path $\PathVar(0;\varnothing)=(0)$.  The
hypothesis $|E|\leq |A|^{a_r}$ and $a_r<1$ imply the size
condition of Lemma~\ref{lem:preexisting-head} by the choice of
$M_{\rm sg}(r,L,\gamma)$.  Apply that lemma, with $Z=\{\Sigma(S)\}$, to the elements
$e_1,\ldots,e_{t-1}$ and elements selected from $A$. { This gives a path $\PathVar$ with $\Sigma(S)\notin V^+(\PathVar)$. Every helper selected from $A\subseteq K$ has zero image in $G/K$. Hence appending a helper leaves the quotient coordinate unchanged, while appending $e_i$ advances it from $\overline e_1+\cdots+\overline e_{i-1}$ to $\overline e_1+\cdots+\overline e_i$. Therefore every quotient class represented in $V^+(\PathVar)$ is either $0$ or a proper partial sum of $\overline e_1,\ldots,\overline e_t$, and is consequently different from $\BarSum$.}

If $\BarSum=0$, then necessarily $t\geq2$.  Begin instead with the path
$\PathVar_0=\PathVar(0;(e_1))=(0,e_1)$, before selecting any element of $A$.  Since $\Sigma(S)\in K$ whereas $e_1\notin K$, this
first vertex is not $\Sigma(S)$.  The same size estimate verifies the
hypothesis of Lemma~\ref{lem:preexisting-head} for $P_0$.  Apply that
lemma to $P_0$, the elements
$e_2,\ldots,e_{t-1}$, the set $A$, and
$Z=\{\Sigma(S)\}$. 
{Again, every helper selected from $A\subseteq K$ has zero image in $G/K$. Thus every quotient class represented by a non-initial vertex of $\PathVar$ is one of the proper partial sums $\overline e_1,\overline e_1+\overline e_2,\ldots,\overline e_1+\cdots+\overline e_{t-1}$. These elements are all non-zero because the total sum is zero and the ordering supplied by Lemma~\ref{lem:small-quotient-head} has no proper partial sum equal to the total. Hence
\[
        K\cap V^+(\PathVar)=\varnothing.
\]
}

Thus, in both cases, $\PathVar$ is a path, $\Sigma(S)\notin V^+(\PathVar)$, and the
coset
\[
        \operatorname{end}(\PathVar)+e_t+K
\]
is disjoint from $V^+(\PathVar)$.  Let
\[
        R:=A\setminus\operatorname{lab}(\PathVar)
\]
be the elements of $A$ not used in constructing $\PathVar$.  In either case the
construction selects at most $|E|$ elements from $A$ and creates at most
$2|E|$ new vertices.  Thus
\[
        |R|\geq |A|-|E|\geq |A|/2,
        \qquad
        |V(\PathVar)|+1\leq 2|E|+2\leq 2|A|^{a_r}+2
\]
for all sufficiently large $|A|$. {By the choice of $M_{\rm sg}(r,L,\gamma)$,
\[
        |V(\PathVar)|+1\leq |R|^{a_r+\beta}.
\]
Since $\PathVar$ uses the elements $e_1,\ldots,e_{t-1}$ and precisely the elements of $A\setminus R$, conservation of the total sum gives
\[
        \operatorname{end}(\PathVar)+e_t+\sum_{x\in R}x
        =\sum_{s\in S}s
        =\Sigma(S).
\]
Thus the terminal value denoted by $\Sigma$ in Proposition~\ref{prop:subgroup-pending-bridge} is exactly $\Sigma(S)$, which does not belong to $V^+(\PathVar)$. That proposition, applied with $e=e_t$ and the remaining set $R$, completes the ordering.
}
\end{proof}

\begin{proof}[Proof of Theorem~\ref{thm:comparable-prime-power}]
Fix $L\geq2$ and $\gamma\geq1$.  Let $c_*>0$ be the absolute constant
in Theorem~\ref{thm:large-known}, and set
\[
        c_0:=\min\{c_*,1/2\}.
\]
The large-set theorem remains valid with $c_0$ in place of $c_*$.  Put
\[
        \eta_0:=\frac14\min\left\{a_L,
        \frac{c_0}{1-c_0}\right\}>0.
\]
{Let $D_*:=D(\eta_0)$ be the chain depth supplied by Lemma~\ref{lem:cyclic-local-repair}. Let
\[
m_{\rm np}:=m_0(D_*,L),
\]
and let $C_{\rm np}=C_{\rm np}(D_*,L)$ and $c_{\rm np}=c_{\rm np}(D_*,L)$ be constants for which all conclusions of Lemma~\ref{lem:polynomial-kneser-alternative}, including the boundedly conditioned local--remote estimate, hold at depth $D_*$. Put
\[
C_{\rm LR}:=C_{\rm np}+1.
\]
Let
\[
n_{\rm lr}:=n_0(\eta_0,C_{\rm LR}),
\]
where $n_0$ is the threshold supplied by Lemma~\ref{lem:cyclic-local-repair}.
}

Choose an integer $n_*=n_*(L,\gamma)$ larger than
\[
  2\max_{1\leq r\leq L}
  \{M_{\rm pc}(r,L,\gamma),M_{\rm br}(r,L,\gamma),
    M_{\rm sg}(r,L,\gamma)\},
\]
and larger than $m_{\rm np}$ and $n_{\rm lr}$.  Choose it sufficiently
large that, for every $n\geq n_*$ and every $1\leq r\leq L$,
\begin{equation}
\label{eq:top-level-np-smallness}
 C_{\rm np}n^{-a_r/2}\sqrt{\log(2+n)}
 \leq \frac12n^{-\eta_0},
 \qquad
 C_{\rm np}e^{-c_{\rm np}n^{a_r}}
 \leq n^{-\eta_0D_*},
 \qquad
 n^{-c_0/(1-c_0)}\leq \frac12n^{-\eta_0}.
\end{equation}
Choose $n_*$ also so that, for every $n\geq n_*$ and every
$1\leq r<\ell\leq L$,
{
\begin{equation}
\label{eq:top-level-elementary-sizes}
 n-n^{a_\ell}\geq \frac n2,
 \qquad n^{a_\ell}+2\leq (n/2)^{a_r+\beta},
 \qquad n^{a_\ell}\leq (n/2)^{a_r}.
\end{equation}
These inequalities are possible because $a_\ell<1$ and $a_r\geq a_{\ell-1}>a_\ell+20\beta$.

Choose $0<\alpha_0<c_0$ and apply Theorem~\ref{thm:cyclic-small} with $t=1$ and $\alpha=\alpha_0$. Since $\alpha_0<c_0$, its polynomial range
\[
        |S|\leq p^{1-\alpha_0}
\]
overlaps the large-set range
\[
        |S|\geq p^{1-c_0}
\]
in Theorem~\ref{thm:large-known}. The finitely many smaller cardinalities are covered by Theorem~\ref{thm:cdf-known} once $p$ is large enough.  Hence there is a threshold $P_{\rm cyc}$ such that every subset of $\mathbb Z_p\setminus\{0\}$ has a valid ordering for every prime $p\geq P_{\rm cyc}$.  Finally choose $p_*=p_*(L,\gamma)$ larger than
}
\[
 \max\left\{
 P_{\rm cyc},\,3,\,\frac{(n_*-1)!}{2},\,
 \max_{1\leq r\leq L}
 \{P_{\rm pc}(r,L,\gamma),P_{\rm br}(r,L,\gamma),
   P_{\rm sg}(r,L,\gamma)\}\right\}.
\]
For every integer $1\leq s<n_*$ one has
\[
        p_{\min}\geq p_*>\frac{(n_*-1)!}{2}\geq\frac{s!}{2}.
\]
Every prime divisor of $|G|$ is at least $p_{\min}$, so the cyclic part
of Theorem~\ref{thm:cdf-known} applies to every set of cardinality less
than $n_*$.  We prove the theorem under the hypothesis
$p_{\min}\geq p_*$ and finally set $p_0(L,\gamma):=p_*$.

Call a non-trivial cyclic group $H$
\emph{$(L,\gamma,p_*)$-admissible} when
$\Omega(H)\leq L$ and, writing its distinct prime divisors as
$q_1<\cdots<q_t$, one has
\[
        p_*\leq q_1<\cdots<q_t\leq\gamma q_1.
\]
Every non-trivial subgroup of an $(L,\gamma,p_*)$-admissible cyclic
group is again $(L,\gamma,p_*)$-admissible, because its distinct prime
divisors form a nonempty sublist of those of the ambient group.  We use
strong induction on
$\ell=\Omega(G)$.  If $\ell=1$, then $G$ has prime order and the result
follows from the choice of $P_{\rm cyc}\leq p_*$.  Assume now that
$2\leq\ell\leq L$ and that the result holds for every
$(L,\gamma,p_*)$-admissible cyclic group with fewer than $\ell$ prime
factors, counted with multiplicity.
Let
\[
        N:=|G|,
        \qquad S\subseteq G\setminus\{0\},
        \qquad n:=|S|.
\]
If $n=0$, the empty ordering is valid.  If $1\leq n<n_*$,
Theorem~\ref{thm:cdf-known} gives a valid ordering by the preceding
factorial inequality.  If
$n\geq N^{1-c_0}$, Theorem~\ref{thm:large-known} gives one.  Hence we may
assume
\begin{equation}
\label{eq:top-level-middle-range}
        n_*\leq n<N^{1-c_0}.
\end{equation}

Apply Lemma~\ref{lem:polynomial-kneser-alternative} to $S\subseteq G$
with chain depth $D_*$.  Suppose first that its polynomial non-periodic
alternative holds.  From \eqref{eq:top-level-middle-range},
\[
        \frac nN
        \leq n^{-c_0/(1-c_0)}.
\]
Since $a_\ell\geq a_L$, the definition of $\eta_0$ and
\eqref{eq:top-level-np-smallness} imply
\[
        \vartheta_{0,1-a_\ell}(n;N)\leq n^{-\eta_0}
\]
for $n\geq n_*$.  The additive term in the conditioned chain bound from
Lemma~\ref{lem:polynomial-kneser-alternative} is
\[
        C_{\rm np}\exp(-c_{\rm np}n^{a_\ell})
        \leq n^{-\eta_0D_*}
\]
by \eqref{eq:top-level-np-smallness}. { Set
\[
\Theta_{\rm LR}:=n^{-\eta_0}.
\]
For every $1\leq h\leq D_*$, the boundedly conditioned local--remote conclusion of Lemma~\ref{lem:polynomial-kneser-alternative} gives
\[
\sum_{\mathrm{adm}}\Pr[\text{the $h$ equations hold}]
\leq C_{\rm np}\vartheta_{0,1-a_\ell}(n;N)^h
+C_{\rm np}\exp(-c_{\rm np}n^{a_\ell}).
\]
By the preceding estimates,
\[
\vartheta_{0,1-a_\ell}(n;N)\leq\Theta_{\rm LR}
\]
and
\[
C_{\rm np}\exp(-c_{\rm np}n^{a_\ell})
\leq n^{-\eta_0D_*}
\leq\Theta_{\rm LR}^h.
\]
Consequently,
\[
\sum_{\mathrm{adm}}\Pr[\text{the $h$ equations hold}]
\leq(C_{\rm np}+1)\Theta_{\rm LR}^h
=C_{\rm LR}\Theta_{\rm LR}^h
\]
for every $1\leq h\leq D_*$. The same estimate holds after conditioning on the values in at most $100D_*^2$ prescribed positions and after translating the targets by random variables measurable with respect to those values. Since $n\geq n_*\geq n_{\rm lr}$ and $\Theta_{\rm LR}=n^{-\eta_0}$, Lemma~\ref{lem:cyclic-local-repair} applies and gives a valid ordering of $S$.

It remains to consider the structural alternative. There are a non-trivial proper subgroup $K<G$ and a coset $g+K$ such that, with
\[
        A:=S\cap(g+K),\qquad E:=S\setminus A,
\]
one has
\begin{equation}
\label{eq:top-level-periodic-split}
        |E|\leq n^{a_\ell},\qquad |A|\geq n-n^{a_\ell}\geq n/2.
\end{equation}
Put $r:=\Omega(K)\leq\ell-1$. Since $|A|\geq n/2$, the second inequality in \eqref{eq:top-level-elementary-sizes} gives
\begin{equation}
\label{eq:top-level-recursive-state}
        |E|+2\leq n^{a_\ell}+2\leq |A|^{a_r+\beta}.
\end{equation}
}
If $g\notin K$, apply
Proposition~\ref{prop:proper-coset-recursive} with the initial path
$\PathVar(0;\varnothing)=(0)$.  Its state condition is exactly
\eqref{eq:top-level-recursive-state}, and its condition on the final
vertex is vacuous because $V^+(\PathVar(0;\varnothing))=\varnothing$.

If $g\in K$, then $g+K=K$ and $A=S\cap K$.  Every non-trivial
subgroup $H\leq K$ has fewer than $\ell$ prime factors and is
$(L,\gamma,p_*)$-admissible, so the induction hypothesis gives a valid
ordering for every subset of $H\setminus\{0\}$.  For the trivial
subgroup the only such subset is empty, which is valid by convention.
Moreover, since $|A|\geq n/2$, the third inequality in
\eqref{eq:top-level-elementary-sizes} gives
\[
        |E|\leq n^{a_\ell}\leq |A|^{a_r}.
\]
Hence Proposition~\ref{prop:subgroup-layered-completion} applies and completes
the ordering.

Both branches yield a valid ordering of $S$.  Taking
$p_0(L,\gamma)=p_*$ proves the theorem.
\end{proof}

\section{\texorpdfstring{{Proof of the layered local-repair lemma}}{Proof of the layered local-repair lemma}}
\label{sec:layered-local-repair-proof}

{This section proves Lemma~\ref{lem:layered-local-repair}.}  We retain
all notation and hypotheses from its statement.

{Let $\sigma$ be uniformly distributed on $\operatorname{Ord}(X)$ and write $x_i:=\sigma(i)$.  For a fixed permutation $\varpi$ of the positions, put
\[
 \sigma^\varpi:=\sigma\circ\varpi.
\]
Thus the entry in position $i$ of $\sigma^\varpi$ is $x_{\varpi(i)}$.  Since $\sigma\mapsto\sigma^\varpi$ is a bijection of $\operatorname{Ord}(X)$,}
\begin{equation}
\label{eq:permutation-absorption}
 {\Pr[\sigma^\varpi\in\mathcal A]=\Pr[\sigma\in\mathcal A]}
\end{equation}
for every event $\mathcal A$.

{{Throughout this section} we abbreviate
$P_b:=P_b(\sigma)$, $V_b:=V_b(\sigma)$, and $Q_b:=Q_b(\sigma)$ for
the coordinates of the currently considered ordering.  Whenever the
ordering is changed explicitly, the argument will instead use the
corresponding ordering symbol, such as $\sigma^\varpi$ or
$\widetilde\sigma$.}

{
\subsection*{Overview of the proof and the role of the profiles}

The proof has two logically distinct layers.  The probabilistic layer
starts from a uniformly random ordering and shows that there is an
ordering in which every local defect has sufficiently many admissible
repair choices.  The deterministic layer then processes the defects
from right to left and removes them by short disjoint transpositions.
The formal apparatus developed below is the interface between these two
layers.

There are two kinds of defects.  A collision between two tail vertices
becomes an affine equation supported on an interval whose length is a
multiple of the quotient order $d$; a visit to a vertex already used by
the head becomes a prefix target equation.  We call both of them
\emph{obstruction equations}.  An endpoint at which either kind of
defect occurs will be called a bad endpoint.  When such an endpoint is
repaired, all positions moved by the relevant transposition lie in a
short window $J$, whose length is bounded in terms of $D$ only.  Pulling
every affected obstruction back to the original random ordering
separates the contribution of positions outside $J$, which we call its
\emph{remote part}, from a correction involving only coordinates whose
positions lie in $J$.

The central deterministic step treats any family of at most $D$
obstruction equations associated with one repair window.  It performs
two successive triangularizations.  First, equations having the same
dependence on positions outside $J$ are grouped together, and one
chosen representative is subtracted from the other equations in the
group.  This cancels their common remote part and produces local
equations, each of which determines a distinct coordinate of $J$; these
coordinates, exposed from right to left, will be called the local
pivots.  Second, the remaining representative equations are ordered
according to where their remote prefix sums end and are replaced by
successive differences.  Their remote parts then become sums over
consecutive, disjoint blocks outside $J$; the resulting ordered sequence
will be called a remote chain.  Lemma~\ref{lem:clock-triangularization}
formalizes both triangularizations and shows that they can be selected
uniformly before the numerical endpoints and target values are
inserted.

The logical flow can be summarized as follows:
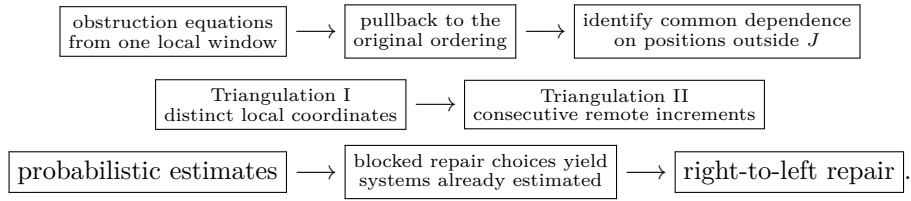
\[
 \boxed{\substack{\text{obstruction equations}\\
                   \text{from one local window}}}
 \longrightarrow
 \boxed{\substack{\text{pullback to the}\\
                   \text{original ordering}}}
 \longrightarrow
 \boxed{\substack{\text{identify common dependence}\\
                   \text{on positions outside }J}}
\]
\[
 \boxed{\substack{\text{Triangulation I}\\
                   \text{distinct local coordinates}}}
 \longrightarrow
 \boxed{\substack{\text{Triangulation II}\\
                   \text{consecutive remote increments}}}
\]
\[
 \boxed{\text{probabilistic estimates}}
 \longrightarrow
 \boxed{\substack{\text{blocked repair choices yield}\\
                   \text{systems already estimated}}}
 \longrightarrow
 \boxed{\text{right-to-left repair}}.
\]

The formal notion of a complete $J$-local profile, defined below,
records the combinatorial template of size bounded in terms of $D$ that
determines these row operations.  It does not record the numerical
positions at which the remote prefix sums end, nor the target values.
Those parameters are inserted only after the profile is fixed and are
then summed probabilistically. This separation prevents the outer union bound from paying an additional polynomial factor for the numerical choices of remote cut points and target values: these choices are summed within the probability estimate for a fixed combinatorial profile.

After triangularization, the probability estimate follows a sequential
exposure of the coordinates involved.  Once all entries of $J$ except
the distinct pivot coordinates have been exposed, each local equation
forces one still-unexposed coordinate and costs $O(1/m)$.  Writing $k$
for the number of local pivots and $h_{\rm rem}$ for the number of
remote-chain equations, the local contribution is $m^{-k}$.  After the
whole window has been exposed, the remaining coordinates still form a
uniform ordering of the residual set, and the remote equations are
estimated one increment at a time.  Since every collision interval must
have length divisible by $d$, each admissible remote endpoint varies in
one prescribed residue class modulo $d$.  Summing over that progression
produces the saving encoded by $\vartheta_{d,a}(m,N)$.  Thus the basic
heuristic, made uniform in Lemma~\ref{lem:clock-chain}, is
\[
 \underbrace{m^{-k}}_{\text{local pivots}}
 \underbrace{\vartheta_{d,a}(m,N)^{h_{\rm rem}}}_{\text{remote chain}}
 \leq \vartheta_{d,a}(m,N)^{k+h_{\rm rem}}.
\]
The residue-restricted estimate for one remote increment is isolated as
Lemma~\ref{lem:residue-restricted-one-step};
Lemma~\ref{lem:clock-chain} iterates that estimate along the remote
chain, and Lemma~\ref{lem:layered-targets} then incorporates the layered
forbidden targets.

The remaining part of {this section} reconnects this abstract machinery
with the repair algorithm.  A candidate transposition fails precisely
when it creates a new obstruction containing exactly one of its two
positions; such an obstruction is called a blocker.  Lemma~\ref{lem:blocker-systems}
shows that every family of at most $D$ blockers arising in the algorithm
canonically generates one of the profiles already estimated.
Lemma~\ref{lem:deterministic-layered-repair} is then purely
combinatorial.  It shows that right-to-left repair succeeds provided
that bad endpoints do not occur near the boundary, do not cluster too
densely, and do not have too many blocked partners.  The events
\textup{(E1)}, \textup{(E2)}, and \textup{(E3)} are exactly the
probabilistic failures of these three deterministic conditions.  The
auxiliary event \textup{(E0)} is used only for one class of blockers, to
ensure that its second relevant endpoint lies outside the local repair
window.

\paragraph{\textbf{Return to the proof of Lemma~\ref{lem:layered-local-repair}.}}
We now introduce formally the obstruction equations and the local data
summarized above.  The definitions below uniformly encode the solution-preserving row operations and the sequential exposure just described, under bounded local conditioning and in the presence of the local permutations occurring in the repair analysis.
}

{All equations below are understood in $K$.  Since $g+K$ has order $d$ in $G/K$, the equality
$V_{s-1}(\sigma)=V_t(\sigma)$ can hold only if
$t\equiv s-1\pmod d$.  For $1\leq s\leq t\leq m$ satisfying this congruence, the equality is equivalent in $K$ to the \emph{clock-collision equation}}
\begin{equation}
\label{eq:clock-collision-equation}
 {\mathcal C([s,t]):\qquad
 \sum_{i=s}^t x_i+\frac{t-s+1}{d}c=0.}
\end{equation}
{For $t\in[m]$ and $y\in W_{\operatorname{res}_d(t)}$, the associated
\emph{target equation} is}
\begin{equation}
\label{eq:clock-target-equation}
 {\mathcal T(t,y):\qquad Q_t(\sigma)=y.}
\end{equation}
{We call either \eqref{eq:clock-collision-equation} or
\eqref{eq:clock-target-equation} an \emph{obstruction equation}, with
support $[s,t]$ or $[1,t]$, respectively, and right endpoint $t$; we
write $\supp(E)$ for its support.  The variable cut points of
$\mathcal C([s,t])$ are $s-1$ and $t$, called its left and right cut
points; the unique variable cut point of $\mathcal T(t,y)$ is its right
cut point $t$.}

{
For every integer $D\geq1$, set $\Lambda_D:=40D+1$. A \emph{$D$-local window} is an interval $J=[j_-,j_+]\subseteq[m]$ with $|J|\leq\Lambda_D$; put}
\[
 J^-:=\{j_--1,\ldots,j_+-1\},\qquad
 J^\partial:=J^-\cup J.
\]
A left cut point is $J$-local if it lies in $J^-$, whereas a right
cut point is $J$-local if it lies in $J$.  All other variable cut points
are $J$-remote.  Thus every $J$-local cut point lies in $J^\partial$.

\begin{example}[A $D$-local window]
\label{ex:D-local-window}
Suppose that $m\geq10$, $\Lambda_D\geq4$, and $J=[7,10]$.  Then
\[
 j_-=7,\qquad j_+=10,\qquad
 J^-=\{6,7,8,9\},\qquad J^\partial=\{6,7,8,9,10\}.
\]
The clock-collision equation $\mathcal C([4,9])$ has cut points $3$ and $9$:
its left cut point is $J$-remote and its right cut point is $J$-local.
The target equation $\mathcal T(9,y)$ has the single $J$-local cut point
$9$.
\end{example}

Let $E$ be an obstruction equation as defined above.  In the
local-repair arguments below, we shall compare the affine equation
associated with $E$ before and after rearranging entries whose positions
lie in the local window $J$.  
{For a given comparison, let
\[
 \varpi_E\in\operatorname{Sym}([m]),
 \qquad
 \supp(\varpi_E):=\{i\in[m]:\varpi_E(i)\neq i\}\subseteq J,
\]
be the corresponding permutation of positions.  Thus $\varpi_E$ fixes
every position outside $J$.  Put
\[
 \widetilde x_i^E:=x_{\varpi_E(i)}
 \qquad (i\in[m]).
\]

Write the obstruction equation in the rearranged ordering
$\widetilde x^E=(\widetilde x_1^E,\ldots,\widetilde x_m^E)$ as}
\[
 \overline F_E({\widetilde x^E})
 :=
 \sum_{i=1}^m\overline\varepsilon_E(i){\widetilde x_i^E}+\kappa_E=0,
 \qquad
 \overline\varepsilon_E(i)\in\{-1,0,1\},
\]
where $\kappa_E\in K$ denotes its constant term.
Its pullback to the original variables
$x=(x_1,\ldots,x_m)$ is the affine form
\begin{equation}
\label{eq:pulled-back-equation}
 F_E(x)
 :=
 \overline F_E({\widetilde x^E}(x))
 =
 \sum_{j=1}^m\varepsilon_E(j)x_j+\kappa_E,
 \qquad
 \varepsilon_E(j)
 =
 \overline\varepsilon_E({\varpi_E}^{-1}(j)).
\end{equation}

For an affine form $F$ in the variables $x_1,\ldots,x_m$, write
$[x_j]F$ for the coefficient of $x_j$.  We say that $F$ is
\emph{$J$-local} if
\[
 [x_j]F=0
 \qquad\text{for every }j\notin J.
\]
Let
\[
 F_E^{\mathrm{id}}(x)
 :=
 \sum_{j=1}^m\overline\varepsilon_E(j)x_j+\kappa_E
\]
be the affine form before the local rearrangement, and define
\begin{equation}
\label{eq:local-correction}
 H_E(x):=F_E(x)-F_E^{\mathrm{id}}(x)
 =\sum_{j\in J}
   \bigl(\varepsilon_E(j)-\overline\varepsilon_E(j)\bigr)x_j.
\end{equation}
We call $H_E$ the \emph{local correction} produced by ${\varpi_E}$.
In particular, $H_E$ is $J$-local.
Whenever an obstruction equation $E$ is included below in a system
relative to $J$, it is accompanied by a specified permutation
${\varpi_E}$ supported on $J$; the notation $F_E$, $\varepsilon_E$, and
$H_E$ always refers to that choice.

\begin{example}[Pullback of a locally permuted obstruction equation]
\label{ex:pullback-local-permutation}
Let $d=3$, let $J=[7,10]$, and consider the clock-collision equation
\[
 E=\mathcal C([4,9]).
\]
Since the interval has length $6$, 
{the equation in a generic ordering
$\mathbf u=(u_1,\ldots,u_m)$ is}
\[
 {\overline F_E(\mathbf u)
 =u_4+u_5+u_6+u_7+u_8+u_9+2c=0.}
\]
Thus
\[
 \overline\varepsilon_E(i)
 =
 \begin{cases}
  1,&\text{if }4\leq i\leq9,\\
  0,&\text{otherwise},
 \end{cases}
 \qquad
 \kappa_E=2c.
\]
Take ${\varpi_E}=(8\ 10)$, viewed as a permutation of $[m]$ fixing all
other positions.  Its support is contained in $J$.  
{Writing $\widetilde x_i^E:=x_{\varpi_E(i)}$, one has}
\[
{\widetilde x_8^E=x_{10},\qquad
\widetilde x_{10}^E=x_8,\qquad
\widetilde x_i^E=x_i\quad(i\notin\{8,10\}).}
\]
Pulling the equation back to the original ordering
gives
\[
 \begin{aligned}
 F_E(x)
  &=x_4+x_5+x_6+x_7+x_{10}+x_9+2c\\
  &=P_9-P_3-x_8+x_{10}+2c.
 \end{aligned}
\]
For instance,
\[
 [x_8]F_E=0,
 \qquad
 [x_{10}]F_E=1.
\]
Since $F_E^{\mathrm{id}}=P_9-P_3+2c$, the local correction is
\[
 H_E=-x_8+x_{10}.
\]
In the terminology introduced below, this is an example of type
$\mathsf{LRC}$: the left cut point $3$ is $J$-remote, whereas the
right cut point $9$ is $J$-local.
\end{example}

\begin{lemma}[Cut-point configurations under a local permutation]
\label{lem:local-cut-point-configurations}
{
Let $E$ be an obstruction equation and let ${\varpi_E}$ be supported on $J$. If $E={\mathcal C([s,t])}$ is a clock-collision equation with at least one $J$-local cut point, then exactly one of the first three rows of Table~\eqref{eq:J-normal-forms} applies. If $E={\mathcal T(t,y)}$ is a target equation with ${t}\in J$ or ${t}>j_+$, then exactly one of the last two rows applies.
}

In the table, $r_E$ denotes the unique $J$-remote cut point, when one is
present.  For a clock-collision equation ${\mathcal C([s,t])}$, we write
$u_E:={s-1}$ and $v_E:={t}$ whenever the corresponding cut point is
$J$-local.  For a target equation ${\mathcal T(t,y)}$, we write
$v_E:={t}$ when its cut point is $J$-local.
\begin{equation}
\label{eq:J-normal-forms}
\begin{array}{c|c|c}
\nu_J(E)&\text{cut-point configuration}&F_E\\ \hline
\mathsf{LC}&u_E={s-1}\in J^-,\ v_E={t}\in J
  &P_{v_E}-P_{u_E}+H_E+\kappa_E\\
\mathsf{LRC}&v_E={t}\in J,\ r_E={s-1}<j_--1
  &P_{v_E}-P_{r_E}+H_E+\kappa_E\\
\mathsf{RRC}&u_E={s-1}\in J^-,\ r_E={t}>j_+
  &P_{r_E}-P_{u_E}+H_E+\kappa_E\\
\mathsf{AT}&v_E={t}\in J
  &P_{v_E}+H_E+\kappa_E\\
\mathsf{RRT}&r_E={t}>j_+
  &P_{r_E}+H_E+\kappa_E.
\end{array}
\end{equation}
\end{lemma}

\begin{proof}
By \eqref{eq:local-correction}, the effect of ${\varpi_E}$ is the
$J$-local correction $H_E$.  For a clock-collision equation one has
\[
 F_E^{\mathrm{id}}={P_t-P_{s-1}}+\kappa_E.
\]
If $u_E={s-1}\in J^-$ and $v_E={t}\in J$ (equivalently, if
${s,t}\in J$), both cut points are local and the first row follows.
If ${t}\in J$ but ${s}\notin J$, then ${s-1<t\leq j_+}$ and the left cut
point can be remote only to the left of $J^-$; hence
${s-1<j_--1}$, giving type $\mathsf{LRC}$.  Symmetrically, if ${s}\in J$
but ${t}\notin J$, then ${t>j_+}$, giving type $\mathsf{RRC}$.
Substituting the corresponding cut points into
${P_t-P_{s-1}}+H_E+\kappa_E$ gives the first three displayed forms.

{
It remains to explain why no type with two remote cut points is included. If neither cut point is $J$-local, then either ${[s,t]}\cap J=\varnothing$ or $J\subseteq{[s,t]}$. In the first case a permutation supported on $J$ changes no coordinate in the interval sum, while in the second case it merely permutes coordinates within the interval. Hence
\[
\sum_{i={s}}^{{t}} x_{{\varpi_E}(i)}=\sum_{i={s}}^{{t}} x_i
\]
in either case. Such an equation cannot be created or destroyed by a rearrangement supported on $J$ and is therefore excluded from the local profile. Thus, the three clock types in Table~\eqref{eq:J-normal-forms} cover precisely the clock-collision equations that have at least one $J$-local cut point.
}

For a target equation,
\[
 F_E^{\mathrm{id}}={P_t}+\kappa_E.
\]
When $ {t}\in J$ this gives type $\mathsf{AT}$.  When ${t>j_+}$, the prefix
contains all of $J$, so in fact $H_E=0$; we nevertheless retain this
right-remote target as type $\mathsf{RRT}$ because such equations occur
in the layered-target estimates and may be considered jointly with
local equations.  A target with ${t}<j_-$ is supported entirely before
$J$, is unchanged by the local rearrangement, and plays no role in a
new local obstruction.  This proves the classification and the five
forms.
\end{proof}

{

We now introduce the reduction classes and local markings needed for the
first triangularization; the second triangularization will be formalized
after the complete profiles have been defined.
}

An obstruction equation covered by
Lemma~\ref{lem:local-cut-point-configurations} is called
\emph{$J$-admissible}.  For such an equation $E$, let
\[
 \nu_J(E)\in
 \{\mathsf{LC},\mathsf{LRC},\mathsf{RRC},\mathsf{AT},\mathsf{RRT}\}
\]
denote its type, namely the corresponding row of
\eqref{eq:J-normal-forms}.  We refer to
$\mathsf{LRC}$, $\mathsf{RRC}$, and $\mathsf{RRT}$ as the
\emph{remote types}.

A \emph{$D$-bounded obstruction system of size $h$} is an ordered family
$\mathcal E=(E_1,\ldots,E_h)$ of $J$-admissible equations, with
$1\leq h\leq D$.  We use $E$ to denote one of the equations $E_i$ in
this ordered family.  After passing to a profile, the same symbol serves
as the index of the corresponding row; its numerical remote cut point
and target value are supplied later by $\mathbf r$ and $\mathbf y$.

For $E,E'\in\mathcal E$, write $E\sim_J E'$ if either
\[
 \nu_J(E)=\nu_J(E')\in\{\mathsf{LC},\mathsf{AT}\},
\]
or
\[
 \nu_J(E)=\nu_J(E')
 \in\{\mathsf{LRC},\mathsf{RRC},\mathsf{RRT}\}
 \qquad\text{and}\qquad
 r_E=r_{E'}.
\]
The equivalence classes of $\sim_J$ are called the
\emph{$J$-reduction classes}, and $[E]_J$ denotes the class containing
$E$.

\begin{lemma}[Locality within a reduction class]
\label{lem:reduction-class-locality}
If $E\sim_J E'$, then the affine form $F_E-F_{E'}$ is $J$-local.
\end{lemma}

\begin{proof}
If both equations have type $\mathsf{LC}$ or $\mathsf{AT}$, all their
variable cut points are $J$-local, so the difference of their prefix
terms is supported on $J$.  If they have one of the remote types
$\mathsf{LRC}$, $\mathsf{RRC}$, or $\mathsf{RRT}$, their common remote
cut point cancels, while the remaining cut-point terms and the local
corrections are supported on $J$.  The claim follows from
\eqref{eq:J-normal-forms}.
\end{proof}

{
\begin{example}[Two reduction classes and the two triangularizations]
\label{ex:two-triangularizations}
We now follow one bounded system through the entire reduction.  The
example is concrete enough to display all the operations, but it is
presented before the formal notions of marking and complete profile so
that those notions can subsequently be read as a precise record of the
same procedure.

Let $d=2$, let $J=[7,10]$, and assume $m\geq16$.  A genuine
realization with two remote levels is the following.  Choose even cut points
\[
 10<r_\Gamma<r_\Delta\leq m,
 \qquad r_\Gamma\equiv r_\Delta\equiv0\pmod2{.}
\]
{Because the rows may come from obstruction equations pulled back
from different candidate repairs,} each row below has its own prescribed
local permutation: the identity
for $R_\Gamma,R_\Delta,E_1,E_2$, and $(8\ 9)$ for $E_3$, which is the
pullback of a second copy of $\mathcal C([9,r_\Gamma])$.  Thus consider
\[
 \begin{aligned}
  R_\Gamma
   &:=P_{r_\Gamma}-P_6+\frac{r_\Gamma-6}{2}c,
     &&\text{from }\mathcal C([7,r_\Gamma]),\\
  R_\Delta
   &:=P_{r_\Delta}-P_8+\frac{r_\Delta-8}{2}c,
     &&\text{from }\mathcal C([9,r_\Delta]),\\
  E_1&:=x_9+x_{10}+c,
     &&\text{from }\mathcal C([9,10]),\\
  E_2
   &:=P_{r_\Gamma}-P_8+\frac{r_\Gamma-8}{2}c,
     &&\text{from }\mathcal C([9,r_\Gamma]),\\
  E_3
   &:=P_{r_\Gamma}-P_8-x_9+x_8
        +\frac{r_\Gamma-8}{2}c,
     &&\text{the pullback of $\mathcal C([9,r_\Gamma])$ through }(8\ 9).
 \end{aligned}
\]
Indeed, the transposition fixes $P_{r_\Gamma}$ and sends $P_8$ to
$P_8-x_8+x_9$, which gives the displayed expression for $E_3$.
Thus $E_1$ has type $\mathsf{LC}$.  The rows
$R_\Gamma,E_2,E_3$ have type $\mathsf{RRC}$ and share the same remote
right cut point $r_\Gamma$, while $R_\Delta$ represents a second
$\mathsf{RRC}$ reduction class with remote cut point $r_\Delta$.

Choose $R_\Gamma$ as representative of the first reduction class and
retain $R_\Delta$ as the representative of the second.  Subtracting
$R_\Gamma$ from the two non-representative rows gives
\[
 \begin{aligned}
  Z_1&:=E_1=x_9+x_{10}+c,\\
  Z_2&:=E_2-R_\Gamma=-x_7-x_8-c,\\
  Z_3&:=E_3-R_\Gamma=-x_7-x_9-c.
 \end{aligned}
\]

With pivot order $10>8>7$, the coefficient matrix on
$x_{10},x_8,x_7$ is
\[
 \begin{pmatrix}
  1&0&0\\
  0&-1&-1\\
  0&0&-1
 \end{pmatrix}.
\]
After the non-pivot entries of $J$ are exposed, the equations can be
solved in reverse order: $Z_3$ determines $x_7$, then $Z_2$ determines
$x_8$, and finally $Z_1$ determines $x_{10}$.  These are the three
local-pivot rows and contribute a factor $O_D(m^{-3})$.

The two representatives contain the remaining remote information.
Replace them by
\[
 G_1:=R_\Gamma,
 \qquad
 G_2:=R_\Delta-R_\Gamma.
\]
For these rows one has
\[
 G_2=P_{r_\Delta}-P_{r_\Gamma}
       +(P_6-P_8)+\frac{r_\Delta-r_\Gamma-2}{2}c.
\]
Once the entries in $J$ are exposed, $G_1$ controls the remote increment
$(10,r_\Gamma]$ and $G_2$ controls $(r_\Gamma,r_\Delta]$.  These two
consecutive increments form the remote chain.  Both cut points range
in the prescribed residue class $0$ modulo $2$.

The transformed five-equation system therefore has three local rows
and two remote rows, leading to a bound of the form
\[
 m^{-3}\vartheta_{d,a}(m,N)^2
 \leq \vartheta_{d,a}(m,N)^5.
\]
The complete profile introduced below records the two reduction
classes $\Gamma,\Delta$, their representatives, the marking $10,8,7$,
and the order $\Gamma<\Delta$, but not the numerical values of
$r_\Gamma,r_\Delta$.
\end{example}
}

For later row reduction we use the following standard notions.  An
\emph{invertible triangular row transformation} of an ordered system
$F_1=\cdots=F_s=0$ has the form
\begin{equation}
\label{eq:triangular-row-transformation}
 G_i=\epsilon_iF_i+\sum_{j<i}a_{ij}F_j,
 \qquad \epsilon_i\in\{\pm1\},\quad a_{ij}\in\mathbb Z.
\end{equation}
Its matrix is lower triangular with diagonal $\pm1$, hence unimodular
and solution-preserving over $K$.  An ordered family of $J$-local rows
$(F_i,p_i)_{i=1}^k$ is a \emph{right-to-left local-pivot family} if
\begin{equation}
\label{eq:local-pivot-family}
 p_1>\cdots>p_k,\qquad [x_{p_i}]F_i\in\{\pm1\},
 \qquad [x_{p_j}]F_i=0\quad(j<i).
\end{equation}
Then $F_i$ is a local pivot equation with pivot $p_i$; solving these
equations in reverse order leaves $x_{p_i}$ as the only pivot variable
in row $i$ whose value has not already been determined.

Choose one representative $R_\Gamma$ in every $J$-reduction class
other than the $\mathsf{LC}$ class.  Let
\[
 \mathcal L
 :=\{E:\nu_J(E)=\mathsf{LC}\}
 \cup\{E:\nu_J(E)\neq\mathsf{LC},\ E\neq R_{[E]_J}\}.
\]
For each $E\in\mathcal L$, define the $J$-local row
\[
 G_E:=
 \begin{cases}
  F_E,
    &\text{if }\nu_J(E)=\mathsf{LC},\\[1mm]
  F_E-F_{R_{[E]_J}},
    &\text{if }\nu_J(E)\neq\mathsf{LC}.
 \end{cases}
\]
A \emph{right-to-left marking} consists of an injective map
$\ell:\mathcal L\to J$ and an ordering
$E_1,\ldots,E_k$ of $\mathcal L$ such that the single ordered family
\[
 \bigl(G_{E_i},\ell(E_i)\bigr)_{i=1}^k
\]
is a right-to-left local-pivot family.  The system is
\emph{right-to-left} if such representatives and a marking exist.

\begin{example}[{The marking in Example~\ref{ex:two-triangularizations}}]
\label{ex:marking-two-triangularizations}

{In Example~\ref{ex:two-triangularizations}, mark $Z_1,Z_2,Z_3$ respectively by
$10,8,7$.}  Their displayed coefficient matrix shows directly that
$(Z_1,10),(Z_2,8),(Z_3,7)$ is a right-to-left local-pivot family.
\end{example}

A \emph{complete $J$-local profile} $\mathfrak P$ of size $h$
for a right-to-left system consists of the following data:
\begin{enumerate}[label=(\alph*),leftmargin=2.4em]
\item $J$, the ordered equation indices $E_1,\ldots,E_h$, the set $\mathcal L$, the marking
      $\ell$, and its pivot order;
\item for every equation $E$, its type $\nu_J(E)$, the local cut points appearing
      in \eqref{eq:J-normal-forms}, and $\varepsilon_E|_J$;
\item the reduction classes and their chosen representatives;
\item within each remote type, the strict order of its reduction classes
      induced by their numerical remote cut points.
\end{enumerate}
The \emph{profile row order} is obtained by listing the representative of
the $\mathsf{AT}$ class, when present, then the representatives of the
remote reduction classes of types $\mathsf{LRC}$, $\mathsf{RRC}$, and
$\mathsf{RRT}$, in this order and in the strict class orders specified by
$\mathfrak P$, and finally the equations in $\mathcal L$ in pivot order.

{
The preceding definition separates the bounded combinatorial data determining the row reduction from the numerical parameters inserted only when the associated probabilistic events are formed.  The distinction is summarized below:
\begin{center}
\begin{tabular}{p{0.43\textwidth}|p{0.43\textwidth}}
\textbf{Recorded by the profile} & \textbf{Inserted and summed afterwards}\\ \hline
Types, local cut points and coefficients; reduction classes and their
representatives; pivot positions and pivot order; relative order of
remote classes.
& Numerical remote cut points $\mathbf r$; $\mathsf{RRT}$ layer
assignment $\tau$; target tuple $\mathbf y$; values of the exposed
coordinates.\\
\end{tabular}
\end{center}
Thus a profile fixes the row-reduction pattern, whereas
$\mathbf r,\tau,$ and $\mathbf y$ parametrize the probabilistic events
to which that fixed pattern is applied.
}

\begin{example}[{The complete profile associated with Example~\ref{ex:two-triangularizations}}]
\label{ex:profile-two-triangularizations}

{For Example~\ref{ex:two-triangularizations}, the complete profile records the}
$\mathsf{LC}$ row $E_1$, the $\mathsf{RRC}$ classes
$\Gamma=\{R_\Gamma,E_2,E_3\}$ and $\Delta=\{R_\Delta\}$, their chosen
representatives, the marking $10,8,7$, and the class order
$\Gamma<\Delta$.  Changing $r_\Gamma,r_\Delta$ while preserving their
admissibility and order leaves the profile unchanged.
\end{example}

\begin{remark}[Number of complete profiles]
\label{rem:number-complete-local-profiles}
For fixed $J$, there are only $O_D(1)$ complete $J$-local profiles of
size at most $D$.  Indeed, $h\leq D$ and $|J^\partial|=O_D(1)$, so there
are only $O_D(1)$ choices for the types, local cut points, the
restrictions $\varepsilon_E|_J$, the set $\mathcal L$, the injective marking
and pivot order, the reduction classes and their representatives, and
the strict orders of the remote classes.

\end{remark}

Fix a complete $J$-local profile $\mathfrak P$ of size $h$, with
ordered equation indices $E_1,\ldots,E_h$.  An
\emph{$\mathsf{RRT}$-layer assignment} is a map
\[
 \tau:\{E_i:i\in[h],\ \nu_J(E_i)=\mathsf{RRT}\}
 \longrightarrow\{0,\ldots,d-1\}
\]
that is constant on each $\mathsf{RRT}$ reduction class.

Let $\operatorname{RC}_{\rm rem}(\mathfrak P)$ denote the set of remote
$J$-reduction classes occurring in $\mathfrak P$.  A tuple
\[
 \mathbf r=(r_\Gamma)_{\Gamma\in
              \operatorname{RC}_{\rm rem}(\mathfrak P)}
\]
assigns one numerical remote cut point $r_\Gamma$ to each remote
reduction class $\Gamma$.  We denote by $\mathcal R(\mathfrak P,\tau)$
the set of all such tuples satisfying the following conditions:
\begin{enumerate}[label=(\roman*),leftmargin=2.5em]
\item for every remote class $\Gamma$,
\[
 r_\Gamma\in\{0,1,\ldots,m\},
\]
with $r_\Gamma<j_--1$ if $\Gamma$ has type $\mathsf{LRC}$, and
$r_\Gamma>j_+$ if $\Gamma$ has type $\mathsf{RRC}$ or $\mathsf{RRT}$;
\item within each remote type, the values $r_\Gamma$ respect the strict
order of the reduction classes specified by $\mathfrak P$;
\item for every remote class $\Gamma$ and every $E\in\Gamma$,
\[
 \begin{cases}
 r_\Gamma\equiv u_E\pmod d,
   &\text{if }\nu_J(E)=\mathsf{RRC},\\
 r_\Gamma\equiv v_E\pmod d,
   &\text{if }\nu_J(E)=\mathsf{LRC},\\
 r_\Gamma\equiv \tau(E)\pmod d,
   &\text{if }\nu_J(E)=\mathsf{RRT}.
 \end{cases}
\]
\end{enumerate}
For a remote equation $E$, let $[E]_J$ denote its $J$-reduction class and
define
\[
 r_E(\mathbf r):=r_{[E]_J}.
\]
Thus $r_E(\mathbf r)$ is the component of $\mathbf r$ indexed by the
reduction class containing $E$.  For example, if $E,E'\in\Gamma$, then
\[
 r_E(\mathbf r)=r_{E'}(\mathbf r)=r_\Gamma.
\]
For every equation index $E$ of type $\mathsf{AT}$ or $\mathsf{RRT}$,
introduce a target coordinate $y_E$.  Let $\mathcal Y(\mathfrak P,\tau)$ be
the set of tuples
\[
 \mathbf y=(y_{E_i})_{\substack{i\in[h]\\
              \nu_J(E_i)\in\{\mathsf{AT},\mathsf{RRT}\}}}
\]
satisfying
\[
 y_E\in
 \begin{cases}
  W_{\operatorname{res}_d(v_E)},
    &\text{if }\nu_J(E)=\mathsf{AT},\\
  W_{\tau(E)},
    &\text{if }\nu_J(E)=\mathsf{RRT}.
 \end{cases}
\]
If no equation has type $\mathsf{AT}$ or $\mathsf{RRT}$, then
$\mathcal Y(\mathfrak P,\tau)$ consists of the unique empty tuple.

Fix $\mathbf r\in\mathcal R(\mathfrak P,\tau)$ and
$\mathbf y\in\mathcal Y(\mathfrak P,\tau)$.  For every equation $E$,
define its prefix part $B_E^{\mathbf r}$ and affine constant
$\kappa_E^{\mathbf r,\mathbf y}$ by
\[
\begin{array}{c|c|c}
\nu_J(E)&B_E^{\mathbf r}&\kappa_E^{\mathbf r,\mathbf y}\\ \hline
\mathsf{LC}&P_{v_E}-P_{u_E}&\dfrac{v_E-u_E}{d}\,c\\
\mathsf{LRC}&P_{v_E}-P_{r_E(\mathbf r)}&\dfrac{v_E-r_E(\mathbf r)}{d}\,c\\
\mathsf{RRC}&P_{r_E(\mathbf r)}-P_{u_E}&\dfrac{r_E(\mathbf r)-u_E}{d}\,c\\
\mathsf{AT}&P_{v_E}&\lfloor v_E/d\rfloor c-y_E\\
\mathsf{RRT}&P_{r_E(\mathbf r)}&\lfloor r_E(\mathbf r)/d\rfloor c-y_E.
\end{array}
\]

\begin{remark}[Well-definedness of the associated affine forms]
\label{rem:profile-forms-well-defined}
For $j\in J$, the coefficient $[x_j]B_E^{\mathbf r}$ is independent
of the numerical values in $\mathbf r$.  Indeed, in type
$\mathsf{LRC}$ the remote cut point lies to the left of $J$, so
$[x_j]P_{r_E(\mathbf r)}=0$; in types $\mathsf{RRC}$ and
$\mathsf{RRT}$ it lies to the right of $J$, so
$[x_j]P_{r_E(\mathbf r)}=1$.  The remaining types have no numerical
remote cut point.

The coefficients multiplying $c$ in the rows of types $\mathsf{LC}$,
$\mathsf{LRC}$, and $\mathsf{RRC}$ are integers.
For type $\mathsf{LC}$ this follows from the defining condition
$d\mid v_E-u_E$ for a clock-collision equation.  For types
$\mathsf{LRC}$ and $\mathsf{RRC}$ it follows, respectively, from
$r_E(\mathbf r)\equiv v_E\pmod d$ and
$r_E(\mathbf r)\equiv u_E\pmod d$, which are part of the definition of
$\mathcal R(\mathfrak P,\tau)$.
\end{remark}

The profile therefore determines the $J$-local correction
\[
 H_E^{\mathfrak P}:=
 \sum_{j\in J}\bigl(\varepsilon_E(j)-[x_j]B_E^{\mathbf r}\bigr)x_j,
\]
which is independent of the chosen
$\mathbf r\in\mathcal R(\mathfrak P,\tau)$.  The affine form associated with $E$ is
\begin{equation}
\label{eq:profile-affine-form}
 F_E^{\mathbf r,\mathbf y}
 :=B_E^{\mathbf r}+H_E^{\mathfrak P}
  +\kappa_E^{\mathbf r,\mathbf y}.
\end{equation}
For $\mathbf r\in\mathcal R(\mathfrak P,\tau)$ and
$\mathbf y\in\mathcal Y(\mathfrak P,\tau)$, define the associated system
and event by
\[
 \mathcal F(\mathfrak P,\tau,\mathbf r,\mathbf y)
 :=(F_{E_i}^{\mathbf r,\mathbf y}=0)_{i\in[h]},
\]
\begin{equation}
\label{eq:profile-system-event}
 \mathsf A(\mathfrak P,\tau,\mathbf r,\mathbf y)
 :=\bigcap_{i\in[h]}\{F_{E_i}^{\mathbf r,\mathbf y}(x)=0\}.
\end{equation}
Thus $\mathsf A(\mathfrak P,\tau,\mathbf r,\mathbf y)$ is precisely the
event that all equations in
$\mathcal F(\mathfrak P,\tau,\mathbf r,\mathbf y)$ hold.

{
Let $R'_1,\ldots,R'_s$ be an ordered block of representative affine forms of one remote type, and suppose that their numerical remote cut points satisfy $r_1<\cdots<r_s$. The \emph{triangular remote chain} associated with this block is the ordered family
\begin{equation}
\label{eq:triangular-remote-chain}
G_1:=R'_1,\qquad G_i:=R'_i-R'_{i-1}\quad(2\leq i\leq s).
\end{equation}
The numbers $r_1,\ldots,r_s$ are called its chain cut points. This definition also covers $r_1=0$, which can occur for an equation of type $\mathsf{LRC}$ whose interval begins at position~$1$.
}

\begin{example}[{The remote chain associated with Example~\ref{ex:two-triangularizations}}]
\label{ex:remote-chain-two-triangularizations}

{For Example~\ref{ex:two-triangularizations}, the ordered representative block is}
$R_\Gamma,R_\Delta$, with $r_\Gamma<r_\Delta$, and its triangular remote
chain is simply
\[
 G_1=R_\Gamma,\qquad G_2=R_\Delta-R_\Gamma.
\]
After exposing $J$, the remote parts are the consecutive increments
$(10,r_\Gamma]$ and $(r_\Gamma,r_\Delta]$.
\end{example}

{
The next lemma packages the two row operations already visible in the
examples.  First, subtracting the representative of a reduction class
from every non-representative row extracts all local information into a
right-to-left pivot family.  Second, taking successive differences of
the ordered representatives converts nested remote prefixes into
consecutive remote increments.  More importantly, the profile alone
determines a single unimodular transformation that works for every
admissible choice of $\mathbf r$ and $\mathbf y$.  Thus the row reduction
is uniform throughout the subsequent sums over these parameters.
}

\begin{lemma}[Triangularization of bounded obstruction systems]
\label{lem:clock-triangularization}
Fix $D\geq1$, a complete $J$-local profile $\mathfrak P$ of size $h$
arising from a right-to-left $D$-bounded obstruction system, and an
$\mathsf{RRT}$-layer assignment $\tau$ for its equations of type
$\mathsf{RRT}$.  Let
\[
 \nu_{\mathsf{AT}}:=
 \begin{cases}
  1,&\text{if $\mathfrak P$ contains an equation of type $\mathsf{AT}$},\\
  0,&\text{otherwise}.
 \end{cases}
\]
{
Put
\[
 k:=|\mathcal L|,
 \qquad
 h_{\mathrm{rem}}:=|\operatorname{RC}_{\rm rem}(\mathfrak P)|.
\]
}
There is an invertible triangular row transformation, depending only on
$\mathfrak P$, such that, for every
$\mathbf r\in\mathcal R(\mathfrak P,\tau)$ and every
$\mathbf y\in\mathcal Y(\mathfrak P,\tau)$, applying it to the rows of
$\mathcal F(\mathfrak P,\tau,\mathbf r,\mathbf y)$ in profile row order
produces
\begin{enumerate}[label=(\alph*),leftmargin=2.5em]
\item the representative of the $\mathsf{AT}$ class, when it is present;
\item for each nonempty remote type, the triangular remote chain
      associated with its ordered block of representative rows;
\item a right-to-left local-pivot family of size $k$.
\end{enumerate}
{
Altogether the chains in~\textup{(b)} contain $h_{\mathrm{rem}}$ rows, and
\[
        \nu_{\mathsf{AT}}+k+h_{\mathrm{rem}}=h.
\]
}
There are at most three such chains, one for each of the types
$\mathsf{LRC}$, $\mathsf{RRC}$, and $\mathsf{RRT}$.

More precisely, fix one nonempty remote type, let
$E_1^*,\ldots,E_s^*$ be its representative equation indices in the
strict order specified by $\mathfrak P$, and let $G_1,\ldots,G_s$ be
the associated triangular remote chain.  For the chosen $\mathbf r$,
put
\[
 r_i:=r_{E_i^*}(\mathbf r)\qquad(1\leq i\leq s).
\]
Then $r_1<\cdots<r_s$ are the chain cut points.  Recall that, for a
representative equation of type $\mathsf{RRC}$, $u_{E_i^*}$ is its local
left cut point, whereas, for one of type $\mathsf{LRC}$, $v_{E_i^*}$ is
its local right cut point.  Moreover, whenever a partial choice
$r_1,\ldots,r_{i-1},r$ extends to an element of
$\mathcal R(\mathfrak P,\tau)$, one has
\[
 r\equiv
 \begin{cases}
 u_{E_i^*}\pmod d,
   &\text{if }\nu_J(E_i^*)=\mathsf{RRC},\\
 v_{E_i^*}\pmod d,
   &\text{if }\nu_J(E_i^*)=\mathsf{LRC},\\
 \tau(E_i^*)\pmod d,
   &\text{if }\nu_J(E_i^*)=\mathsf{RRT}.
 \end{cases}
\]
\end{lemma}

\begin{proof}
Fix $\mathbf r\in\mathcal R(\mathfrak P,\tau)$ and
$\mathbf y\in\mathcal Y(\mathfrak P,\tau)$, and identify each equation
with its affine left-hand side.  Let $R$ be the
column vector of the chosen representatives, in profile row order, and
let $U$ be the column vector of the remaining rows, in pivot order.  For
each row of $U$ subtract the representative of its reduction class, and
subtract nothing from a row of type $\mathsf{LC}$.  Thus, for a
$0$--$1$ matrix $B$ with at most one $1$ in each row,
\[
 Z:=U-BR
\]
is precisely the right-to-left local-pivot family determined by the
profile.  The rows of $U$ are indexed by $\mathcal L$, so $U$ has $k$
rows. {
The vector $R$ has one row for the $\mathsf{AT}$ class, when present, and one for each remote reduction class, hence $\nu_{\mathsf{AT}}+h_{\mathrm{rem}}$ rows. Therefore
\[
 \nu_{\mathsf{AT}}+k+h_{\mathrm{rem}}=h.
\]
}

Fix one nonempty remote type, and let
$E_1^*,\ldots,E_s^*$ be its representative equation indices in profile
order.  For the fixed $\mathbf r$ and $\mathbf y$, write
\[
 R'_i:=F_{E_i^*}^{\mathbf r,\mathbf y},\qquad
 r_i:=r_{E_i^*}(\mathbf r),\qquad
 H_i:=H_{E_i^*}^{\mathfrak P},\qquad
 \kappa_i:=\kappa_{E_i^*}^{\mathbf r,\mathbf y}.
\]
Then $r_1<\cdots<r_s$.  Replace this block of representatives by
\begin{equation}
\label{eq:successive-representative-differences}
 G_1:=R'_1,\qquad G_i:=R'_i-R'_{i-1}\quad(2\leq i\leq s).
\end{equation}
All representatives in the block have the form
\[
 R'_i=\epsilon P_{r_i}+L_i+\kappa_i,
 \qquad \epsilon\in\{\pm1\},
\]
with the same sign $\epsilon$ throughout the block.  More explicitly,
\[
 L_i=
 \begin{cases}
  -P_{u_{E_i^*}}+H_i,
    &\text{if }\nu_J(E_i^*)=\mathsf{RRC},\\
  P_{v_{E_i^*}}+H_i,
    &\text{if }\nu_J(E_i^*)=\mathsf{LRC},\\
  H_i,
    &\text{if }\nu_J(E_i^*)=\mathsf{RRT}.
 \end{cases}
\]
For $i\geq2$, the difference $L_i-L_{i-1}$ is $J$-local: the
difference of the local cut-point terms is supported on $J$, and so is
$H_i-H_{i-1}$.  Consequently,
\[
 G_i=\epsilon(P_{r_i}-P_{r_{i-1}})
      +(L_i-L_{i-1})+(\kappa_i-\kappa_{i-1})
 \quad(i\geq2).
\]
Together with $G_1=R'_1$, these are precisely the rows in
\eqref{eq:triangular-remote-chain} for the ordered representative block
$R'_1,\ldots,R'_s$.  There is at most one such block for each of
$\mathsf{LRC},\mathsf{RRC},\mathsf{RRT}$.

Let $\widetilde R$ be the column vector consisting, in profile row
order, of the unchanged representative of the $\mathsf{AT}$ class,
when present, and of all the difference rows just constructed.  Let $S$
be the direct sum
of the lower-bidiagonal matrices implementing
\eqref{eq:successive-representative-differences}, together with a
$1\times1$ identity block for the $\mathsf{AT}$ representative, when
present.  The complete row operation is
\begin{equation}
\label{eq:full-triangularization-matrix}
 \begin{pmatrix}\widetilde R\\ Z\end{pmatrix}
 =
 \begin{pmatrix}S&0\\-B&I_k\end{pmatrix}
 \begin{pmatrix}R\\U\end{pmatrix}.
\end{equation}
Its matrix is lower triangular with diagonal $1$, hence unimodular.  The
matrices $B$ and $S$ depend only on the reduction classes, their
representatives, and their profile orders; thus the transformation
depends only on $\mathfrak P$. {Its rows are the representative of the $\mathsf{AT}$ class, when present, the $h_{\mathrm{rem}}$ rows belonging to the triangular remote chains, and the $k$ local pivots.}

Finally, any partial choice of the chain cut points that extends to
$\mathbf r\in\mathcal R(\mathfrak P,\tau)$ satisfies the congruence
conditions in the definition of $\mathcal R(\mathfrak P,\tau)$.  These
are exactly the three congruences stated in the lemma.  This proves the
final assertion.
\end{proof}

At any stage of the argument, to \emph{expose} a set of positions means
to condition on the values of the corresponding coordinates.  A
coordinate is exposed at that stage if its value is measurable with
respect to the sigma-algebra generated by the initial conditioning and
the coordinate values revealed in the preceding steps.

A complete $J$-local profile is called a \emph{single-chain profile}
if it has no equation of type $\mathsf{AT}$ and at most one of the
types $\mathsf{LRC}$, $\mathsf{RRC}$, and $\mathsf{RRT}$ occurs.  By
Lemma~\ref{lem:clock-triangularization}, the representatives of its remote
reduction classes are transformed into either no remote chain or a single
triangular remote chain.

For $S_0\subseteq[m]$ and an injective map $\xi:S_0\to X$, write
\[
 {\Pr}_\xi(\,\cdot\,)
 :=
 \Pr\bigl[\,\cdot\mid x_j=\xi(j)
                 \text{ for every }j\in S_0\bigr].
\]
The conditioning event has positive probability, and under ${\Pr}_\xi$
the coordinates outside $S_0$ form a uniform bijection from
$[m]\setminus S_0$ to $X\setminus\xi(S_0)$.

{
The next lemma gives the quantitative consequence of this procedure for
single-chain profiles.  If the transformed system has $k$ local-pivot
rows and $h_{\rm rem}$ rows in its triangular remote chain, then
$h=k+h_{\rm rem}$, and its proof yields
\[
        C_Dm^{-k}\vartheta_{d,a}(m,N)^{h_{\rm rem}}.
\]
Since $\vartheta_{d,a}(m,N)\geq1/m$, this is at most
$C_D\vartheta_{d,a}(m,N)^h$.  This sharper formula explains why both
triangularizations are needed: equations sharing one remote prefix
cannot be charged as independent remote equations, while their
differences still contribute through distinct local pivots.
}

\begin{lemma}[Clock-restricted interval chains]
\label{lem:clock-chain}
Fix $D\geq1$ and $0<a<1$.  There exist constants
\[
        C_D=C_D(a,L)>0
        \qquad\text{and}\qquad
        m_0=m_0(D,a,L)
\]
such that the following holds.  Let $m\geq m_0$, and assume
\eqref{eq:clock-chain-nonperiodicity}.
Fix a single-chain complete $J$-local profile $\mathfrak P$ of size
$h\leq D$, an $\mathsf{RRT}$-layer assignment $\tau$, and
$\mathbf y\in\mathcal Y(\mathfrak P,\tau)$. {
Let $E^{\rm loc}_1,\ldots,E^{\rm loc}_k$ be the equations in $\mathcal L$, listed in the pivot order recorded by $\mathfrak P$, and put
\[
        p_i:=\ell(E^{\rm loc}_i)\quad(1\leq i\leq k),
        \qquad
        p_1>\cdots>p_k,
        \qquad
        h_{\mathrm{rem}}:=h-k.
\]
If $h_{\mathrm{rem}}>0$, write $J=[j_-,j_+]$ and assume that
}
\[
        1<j_-\leq j_+<m.
\]
Let
\[
        S_0:=J\setminus\{p_1,\ldots,p_k\},
\]
and let $\xi:S_0\to X$ be any injective map.  Then
\begin{equation}
\label{eq:clock-chain-bound}
 \sum_{\substack{\mathbf r\in\mathcal R(\mathfrak P,\tau)\\
                   r_E(\mathbf r)<m\ \text{for every equation $E$ of type $\mathsf{RRT}$}}}
 {\Pr}_{\xi}\!\left[
   \mathsf A(\mathfrak P,\tau,\mathbf r,\mathbf y)
 \right]
 \leq C_D\vartheta_{d,a}(m,N)^h.
\end{equation}
The additional restriction in the sum is vacuous unless the profile has
type $\mathsf{RRT}$.  The constants are independent of
$m,N,d,X,g$, of the profile and its window, and of the tuple $\mathbf y$.
\end{lemma}

{
Before proving Lemma~\ref{lem:clock-chain}, we isolate the
residue-restricted estimate for one remote increment.  Lemma~\ref{lem:clock-chain}
will apply this estimate conditionally and iterate it along the remote
chain.

\begin{lemma}[Residue-restricted one-step summation]
\label{lem:residue-restricted-one-step}
Fix $D\geq1$ and $0<a<1$.  There exist constants
$C_D=C_D(a,L)>0$, $c_D=c_D(a,L)>0$, and
$m_{\rm step}=m_{\rm step}(D,a,L)$ such that the following holds for
$m\geq m_{\rm step}$.  Let $X_{\rm res}\subseteq K$ have cardinality
$M$ satisfying
\[
        c_Dm\leq M\leq m,
\]
and suppose that, for some $\rho_*\geq c_Dm^{a-1}$,
\[
 |X_{\rm res}\cap(z+H)|\leq(1-\rho_*)M
\]
for every non-trivial proper subgroup $H<K$ and every coset $z+H$.
Let $\sigma_{\rm res}$ be a uniformly random ordering of $X_{\rm res}$,
and, for $1\leq u\leq M$, put
\[
 S_u^{\rm res}:=\sum_{i=1}^u\sigma_{\rm res}(i).
\]
Let $\mathcal U\subseteq[1,M/2]$ be contained in one residue class
modulo $d$, and choose an arbitrary target $t_u\in K$ for every
$u\in\mathcal U$.  Then
\begin{equation}
\label{eq:residue-restricted-one-step}
 \sum_{u\in\mathcal U}
 \Pr\bigl[S_u^{\rm res}=t_u\bigr]
 \leq C_D\vartheta_{d,a}(m,N).
\end{equation}
The constants are independent of $m,N,d,X_{\rm res},\mathcal U$, and
of the target family $(t_u)_{u\in\mathcal U}$.
\end{lemma}

\begin{proof}
Put
\[
\widetilde\rho:=\min\left\{\rho_*,\frac14\right\}.
\]
After increasing $m_{\rm step}$, we may assume that
\[
c_Dm^{a-1}\leq\frac14.
\]
Since $\rho_*\geq c_Dm^{a-1}$, it follows that
\[
c_Dm^{a-1}\leq\widetilde\rho\leq\frac14.
\]
Moreover,
\[
|X_{\rm res}\cap(z+H)|
\leq
(1-\rho_*)M
\leq
(1-\widetilde\rho)M
\]
for every non-trivial proper subgroup $H<K$ and every coset $z+H$, and
\[
\widetilde\rho M
\geq
c_Dm^{a-1}M
\geq
c'_Dm^a
\geq
C_D\log^2(2+m),
\]
where we used $M\geq c_Dm$ and changed the constants if necessary.
Lemma~\ref{lem:parameter-dependent-nonperiodic-slice}, applied to
$X_{\rm res}$ with parameter $\widetilde\rho$, gives, uniformly in
$t_u$, for every $u\in\mathcal U$,
\[
\Pr\bigl[S_u^{\rm res}=t_u\bigr]
\leq
\frac1N+
C_D\widetilde\rho^{-1/2}
\frac{\sqrt{\log(2+M)}}{M\sqrt u}.
\]
Since
\[
M\asymp_Dm
\qquad\text{and}\qquad
\widetilde\rho^{-1/2}\leq C_Dm^{(1-a)/2},
\]
this is at most
\[
\frac1N+
C_D\frac{m^{(1-a)/2}\sqrt{\log(2+m)}}{m\sqrt u}.
\]

Because $\mathcal U$ is contained in one residue class modulo $d$, it
has at most $M/d+1$ elements.  Hence the uniform terms contribute at
most
\[
\left(\frac Md+1\right)\frac1N
\leq
\frac{m}{dN}+\frac1N
\leq
\frac{m}{dN}+\frac1m,
\]
where the last inequality follows from $m\leq N$.

There is at most one value $u\in\mathcal U$ with $u<d$.  For this
possible value, expose all but one entry of the prefix of length $u$ in
$\sigma_{\rm res}$.  The final entry is uniform on a set of cardinality
at least
\[
M-u+1\geq\frac M2,
\]
so this contribution is $O_D(1/m)$.

For the remaining values, every interval
\[
[jd,(j+1)d),\qquad j\geq1,
\]
contains at most one element of $\mathcal U$.  Therefore
\[
\sum_{\substack{u\in\mathcal U\\u\geq d}}
\frac1{\sqrt u}
\leq
\sum_{1\leq j\leq m/d}\frac1{\sqrt{jd}}
\leq
\frac{2\sqrt m}{d}.
\]
It follows that the non-uniform terms contribute at most
\[
C_D
\frac{m^{(1-a)/2}\sqrt{\log(2+m)}}{m}
\frac{2\sqrt m}{d}
\leq
C_D\frac{\sqrt{\log(2+m)}}{dm^{a/2}}.
\]
Combining the three contributions proves
\eqref{eq:residue-restricted-one-step}.
\end{proof}
}

\begin{proof}[Proof of Lemma~\ref{lem:clock-chain}]
{Increase $m_0$, if necessary, so that
$m_0\geq m_{\rm step}$.}
Lemma~\ref{lem:clock-triangularization} gives a lower-triangular
unimodular row transformation, depending only on $\mathfrak P$, under
which the event $\mathsf A(\mathfrak P,\tau,\mathbf r,\mathbf y)$ becomes the
simultaneous vanishing of a right-to-left local-pivot family
\[
        Z_1=0,\ldots,Z_k=0,
        \qquad p_1>\cdots>p_k,
\]
{and, when $h_{\mathrm{rem}}>0$, one triangular remote chain. The local-pivot rows are independent of the numerical remote cut-point tuple $\mathbf r$. Indeed, rows of type $\mathsf{LC}$ contain no remote parameter, while for every non-representative equation $E$, the row
\[
F_E^{\mathbf r,\mathbf y}-F_{R_{[E]_J}}^{\mathbf r,\mathbf y}
\]
is formed from two equations in the same reduction class. Hence their common remote prefix cancels, and the $\mathbf r$-dependent part of their affine constants cancels as well. Consequently, the local-pivot family and the factor obtained from it are uniform in $\mathbf r$. Fix $\mathbf r$ in the sum; all estimates below are uniform in this choice.} Under ${\Pr}_\xi$,
first estimate the local-pivot equations in the order
$Z_k,Z_{k-1},\ldots,Z_1$. 
{
For $1\leq i\leq k$, let $\mathcal H_i$ be the sigma-algebra generated by the coordinates in $S_0$ and by the pivots $x_{p_k},\ldots,x_{p_{i+1}}$, and put
\[
\mathcal E_i:=\bigcap_{j=i}^k\{Z_j=0\},
\qquad
\mathcal E_{k+1}:=\operatorname{Ord}(X).
\]
}
On $\mathcal E_{i+1}$, every term of $Z_i$ other than $x_{p_i}$ is
$\mathcal H_i$-measurable.  Conditional on $\mathcal H_i$, the coordinate $x_{p_i}$ is uniform on
the set of elements of $X$ not yet assigned to an exposed coordinate.
This set has cardinality at least $m-\Lambda_D\gg_Dm$, and the
coefficient of $x_{p_i}$ in $Z_i$ is $\pm1$.
Thus
\[
 {\Pr}_\xi[Z_i=0\mid\mathcal H_i]\leq \frac{C_D}{m}
 \qquad\text{on }\mathcal E_{i+1}.
\]
Since $\mathcal E_{i+1}\in\mathcal H_i$, the tower property gives
\[
 {\Pr}_\xi[\mathcal E_i]
 =\mathbb E_\xi\!\left[
    \mathbf 1_{\mathcal E_{i+1}}
    {\Pr}_\xi[Z_i=0\mid\mathcal H_i]
  \right]
 \leq \frac{C_D}{m}{\Pr}_\xi[\mathcal E_{i+1}].
\]
Iterating over $i=k,k-1,\ldots,1$ yields the factor
\begin{equation}
\label{eq:clock-local-pivot-factor}
        (C_D/m)^k,
\end{equation}
uniformly in $\mathbf r$ and $\mathbf y$.

Whenever $Z_i=0$, its coefficient $\pm1$ determines $x_{p_i}$ from the
coordinates already exposed.  Hence, on $\mathcal E_1$, all pivot coordinates
are exposed; since $S_0=J\setminus\{p_1,\ldots,p_k\}$, all coordinates
in $J$ are then exposed.  Moreover, conditioning a uniform bijection on
the value of one coordinate leaves a uniform bijection on the remaining
coordinates.  Applying this observation successively to the pivots
shows that, for every realization of the exposed values on which $\mathcal E_1$
occurs, the coordinates in
\[
        F:=[m]\setminus J
\]
form a uniform ordering of the elements of $X$ not assigned to $J$.

{
The proof is complete when $h_{\mathrm{rem}}=0$, because \eqref{eq:clock-local-pivot-factor} and $\vartheta_{d,a}(m,N)\geq1/m$ give the required bound. We henceforth assume that $h_{\mathrm{rem}}>0$. Let $E_1^*,\ldots,E_{h_{\mathrm{rem}}}^*$ be the representative equation indices of the remote chain, in the order recorded by $\mathfrak P$. For the fixed $\mathbf r$, put
\[
        r_i:=r_{E_i^*}(\mathbf r),
        \qquad
        R'_i:=F_{E_i^*}^{\mathbf r,\mathbf y}
        \qquad(1\leq i\leq h_{\mathrm{rem}}).
\]
Thus $r_1<\cdots<r_{h_{\mathrm{rem}}}$ are the chain cut points, and $R'_1,\ldots,R'_{h_{\mathrm{rem}}}$ are the corresponding representative affine forms before taking successive differences.}
The successive-difference operation is invertible, so the remote chain
vanishes if and only if all the representative equations vanish.  After
the entries of $J$ have been exposed, these equations are equivalent to
prescribed-sum conditions on one of the following families of
subsets of $F$:
{
\begin{equation}
\label{eq:clock-three-families}
 \begin{array}{c|c}
 \text{remote type}&\text{subsets of $F$}\ \\ \hline
 \mathsf{RRC}&(j_+,r_i]\quad(1\leq i\leq h_{\mathrm{rem}}),\\[1mm]
 \mathsf{LRC}&(r_{h_{\mathrm{rem}}+1-i},j_--1]\quad(1\leq i\leq h_{\mathrm{rem}}),\\[1mm]
 \mathsf{RRT}&F\cap[1,r_i]\quad(1\leq i\leq h_{\mathrm{rem}}).
 \end{array}
\end{equation}
}
Indeed, in the $\mathsf{RRC}$ case
\[
 P_{r_i}-P_{u_{E_i^*}}
 =(P_{r_i}-P_{j_+})+(P_{j_+}-P_{u_{E_i^*}}),
\]
and the second summand is supported on $J$.  In the $\mathsf{LRC}$ case
\[
 P_{v_{E_i^*}}-P_{r_i}
 =(P_{j_--1}-P_{r_i})+(P_{v_{E_i^*}}-P_{j_--1}),
\]
and again the second summand is supported on $J$.  These local terms,
together with the local corrections, are measurable under the current
conditioning and can be absorbed into the targets.  The
$\mathsf{LRC}$ representatives are listed in reverse order in
\eqref{eq:clock-three-families} because the corresponding
intervals outside $J$ shrink as the remote cut point increases.  The $\mathsf{RRT}$ statement
follows directly from its prefix form.

{Let $A_1\subset\cdots\subset A_{h_{\mathrm{rem}}}\subset F$ denote the family}
in \eqref{eq:clock-three-families} corresponding to the remote
type of $\mathfrak P$.  Since
$1<j_-\leq j_+<m$, all these inclusions are strict and $A_1$ is
nonempty in the
$\mathsf{LRC}$ and $\mathsf{RRC}$ cases. In the $\mathsf{RRT}$ case
$A_1$ contains $[1,j_--1]$, and {
the restriction $r_{h_{\mathrm{rem}}}<m$ ensures that $A_{h_{\mathrm{rem}}}\neq F$. Thus
\[
        1\leq |A_1|<\cdots<|A_{h_{\mathrm{rem}}}|<|F|.
\]}
Only the $|J|\leq\Lambda_D$ coordinates in $J$ have been prescribed,
and the coordinates indexed by $F$ form a uniform ordering of the
remaining set, whose cardinality is $|F|=m-O_D(1)$.  Hence the
conditioned form of
Proposition~\ref{prop:polynomial-nonperiodic-chain}, with $T=1$ and
$\kappa=1-a$, {applies to the chain $A_1\subset\cdots\subset A_{h_{\mathrm{rem}}}\subset F$.} Its estimates are uniform in
the targets, which may depend on the exposed entries and on the chosen
cut points.  Replacing $|F|$ by $m$ in those estimates changes only the
constant $C_D$.

It remains to sum the cut points subject to the congruence conditions
in the definition of $\mathcal R(\mathfrak P,\tau)$.
{For a fixed cut-point tuple, let
\[
 C_0:=A_1,\qquad C_i:=A_{i+1}\setminus A_i\ (1\leq i<h_{\mathrm{rem}}),
 \qquad C_{h_{\mathrm{rem}}}:=F\setminus A_{h_{\mathrm{rem}}}.
\]}
These sets form the ordered partition used in the proof of
Proposition~\ref{prop:polynomial-nonperiodic-chain}.  To incorporate the
congruence restrictions on the cut points, repeat the exposure argument
from that proof and split the sum
according to an index $j_*$ for which $|C_{j_*}|$ is maximal, breaking
ties by the least index; {there are at most $h_{\mathrm{rem}}+1\leq D+1$ possibilities.}
For each fixed $j_*$, keep $C_{j_*}$ unexposed and expose the remaining
increments in the order used in that proof, from the two ends towards
$C_{j_*}$.

At each step the cut point being summed lies in one prescribed residue
class modulo $d$; the order and extendability conditions in
$\mathcal R(\mathfrak P,\tau)$ can only reduce the set of candidates.
With all adjacent cut points fixed, the cardinality
$u$ of the increment currently being estimated therefore varies in
steps of $d$.  Let ${X_{\rm res}}$ be the set of elements occupying the
positions not yet exposed at that step, and put $M:=|{X_{\rm res}}|$.
Since ${X_{\rm res}}$ contains both the increment currently being estimated
and $C_{j_*}$, maximality of $C_{j_*}$ gives
\[
        1\leq u\leq M/2,
        \qquad
        {M\geq |C_{j_*}|\geq\frac{|F|}{h_{\mathrm{rem}}+1}\gg_Dm,}
        \qquad M\leq m.
\]
{Let $\mathcal U\subseteq[1,M/2]$ be the set of admissible values of $u$ at this stage; it is contained in one residue class modulo $d$.}

{
For each admissible cut-point tuple $\mathbf r$, let
$\PoolBad(\mathbf r)$ be the union of the stopping-time bad-pool events
arising in the remote exposure order determined by $\mathbf r$ and its
chosen maximal increment.  The preceding local-pivot exposure
conditions on only the coordinates in $J$, and
\[
|J|\leq\Lambda_D=O_D(1).
\]
The conditioned stopping-time argument in the proof of
Proposition~\ref{prop:polynomial-nonperiodic-chain} is uniform in the
realized values on these coordinates, in the measurable targets, and
in the prescribed cut points.  It may therefore be applied separately
to every fixed admissible tuple $\mathbf r$, before summing over the at
most $m^{h_{\mathrm{rem}}}\leq m^D$ possible tuples.  Hence}
\[
{\Pr}_\xi[\PoolBad(\mathbf r)]\leq C_D\exp(-c_Dm^a)
\]
for every admissible $\mathbf r$.

We split
\[
\sum_{\substack{\mathbf r\in\mathcal R(\mathfrak P,\tau)\\r_E(\mathbf r)<m\ \text{for every equation $E$ of type $\mathsf{RRT}$}}}{\Pr}_\xi[\mathsf A(\mathfrak P,\tau,\mathbf r,\mathbf y)]
\]
into the contribution on $\PoolBad(\mathbf r)^c$ and the contribution on $\PoolBad(\mathbf r)$.  On $\PoolBad(\mathbf r)^c$, every residual pool encountered in the remote exposure has size comparable to $m$ and satisfies the non-periodicity hypothesis of Lemma~\ref{lem:residue-restricted-one-step} with some $\rho_*\geq c_Dm^{a-1}$.  {Fix a value of $j_*$ and perform the summation over the cut points in the reverse of the corresponding exposure order.  At a given stage, condition on the preceding remote exposures and fix all cut points except the one determining the current increment.  Let $\mathcal H$ be the resulting sigma-algebra.  The admissible choices of the current cut point determine a set $\mathcal U\subseteq[1,M/2]$ of increment sizes contained in one residue class modulo $d$; the order, extendability, and maximal-increment conditions can only restrict this set.  For every admissible realization of $\mathcal H$ for which the current residual pool is good, the residual set $X_{\rm res}$ and the targets $(t_u)_{u\in\mathcal U}$ are fixed, while the still unexposed coordinates, in their induced order, form a uniformly random ordering $\sigma_{\rm res}$ of $X_{\rm res}$.  With
\[
 S_u^{\rm res}:=\sum_{i=1}^u\sigma_{\rm res}(i),
\]
Lemma~\ref{lem:residue-restricted-one-step} therefore gives, for every such realization,
\[
 \sum_{u\in\mathcal U}
 {\Pr}_\xi\!\left[S_u^{\rm res}=t_u\,\middle|\,\mathcal H\right]
 \leq C_D\vartheta_{d,a}(m,N).
\]
When estimating the contribution on which all residual pools are good, any indicator imposing goodness at later stages may be discarded inside this conditional probability.  Apply the displayed estimate at the last remote exposure and then proceed backwards through the preceding exposures.} Since the estimate is uniform in every admissible history, repeated use of the tower property gives
\[
\sum_{\mathbf r}{\Pr}_\xi[\mathsf A(\mathfrak P,\tau,\mathbf r,\mathbf y)\cap\PoolBad(\mathbf r)^c]\leq C_Dm^{-k}\vartheta_{d,a}(m,N)^{h_{\mathrm{rem}}},
\]
where the sum is over the admissible tuples appearing above.

Therefore
\[
\sum_{\mathbf r}{\Pr}_\xi[\PoolBad(\mathbf r)]\leq C_Dm^D\exp(-c_Dm^a),
\]
and consequently
\[
\sum_{\substack{\mathbf r\in\mathcal R(\mathfrak P,\tau)\\r_E(\mathbf r)<m\ \text{for every equation $E$ of type $\mathsf{RRT}$}}}{\Pr}_\xi[\mathsf A(\mathfrak P,\tau,\mathbf r,\mathbf y)]\leq C_Dm^{-k}\vartheta_{d,a}(m,N)^{h_{\mathrm{rem}}}+C_Dm^D\exp(-c_Dm^a).
\]

Since \(\vartheta_{d,a}(m,N)\geq1/m\) and \(h=k+h_{\mathrm{rem}}\leq D\),
\[
m^{-k}\vartheta_{d,a}(m,N)^{h_{\mathrm{rem}}}\geq m^{-h}\geq m^{-D}.
\]
Increase \(m_0=m_0(D,a,L)\) so that
\[
C_Dm^{2D}\exp(-c_Dm^a)\leq1
\]
for every \(m\geq m_0\). Then
\[
C_Dm^D\exp(-c_Dm^a)\leq m^{-D}\leq m^{-k}\vartheta_{d,a}(m,N)^{h_{\mathrm{rem}}},
\]
so the exceptional contribution is absorbed into the principal term after enlarging \(C_D\).

Finally, since \(\vartheta_{d,a}(m,N)\geq1/m\),
\[
m^{-k}\vartheta_{d,a}(m,N)^{h_{\mathrm{rem}}}\leq\vartheta_{d,a}(m,N)^{k+h_{\mathrm{rem}}}=\vartheta_{d,a}(m,N)^h.
\]
This proves \eqref{eq:clock-chain-bound}.

\end{proof}

In the remainder of {this section}, $C_D$ may denote different
positive constants at different occurrences, all depending only on
$D,a$, and $L$ and independent of the remaining parameters.

A complete $J$-local profile is called a \emph{local-repair profile}
if it belongs to one of the following four disjoint classes:
\begin{enumerate}[label=(\alph*),leftmargin=2.5em]
\item a \emph{left-clock profile}, in which every equation has type
      $\mathsf{LC}$ or $\mathsf{LRC}$;
\item an \emph{anchored-target profile}, in which every equation has type
      $\mathsf{AT}$;
\item a \emph{right-clock profile}, in which every equation has type
      $\mathsf{RRC}$;
\item a \emph{remote-target singleton profile}, of size one, whose
      unique equation has type $\mathsf{RRT}$.
\end{enumerate}
The first, third, and fourth classes are single-chain profiles.
{The anchored-target class is not a single-chain class and is treated
separately in part~\textup{(ii)} of Lemma~\ref{lem:layered-targets}.}
These four classes are precisely the profile classes used in the
local-repair estimates below.

For a $D$-local window $J$ and $1\leq h\leq D$, let
$\operatorname{Prof}^{\rm rep}_h(J)$ be the finite set of local-repair
profiles associated with right-to-left $D$-bounded obstruction systems
of size $h$; the superscript $\mathrm{rep}$ is mnemonic for
\emph{repair}.  For any complete $J$-local profile $\mathfrak P$, set
\begin{equation}
\label{eq:profile-total-contribution}
 \operatorname{Cont}(\mathfrak P):=
 \sum_{\tau}
 \sum_{\mathbf y\in\mathcal Y(\mathfrak P,\tau)}
 \sum_{\mathbf r\in\mathcal R(\mathfrak P,\tau)}
 \Pr\!\left[\mathsf A(\mathfrak P,\tau,\mathbf r,\mathbf y)\right],
\end{equation}
where the first sum is over all $\mathsf{RRT}$-layer assignments for
$\mathfrak P$.  Empty index sets contribute one empty assignment.  Thus
$\operatorname{Cont}(\mathfrak P)$ sums all admissible numerical remote
cut points, all $\mathsf{RRT}$-layer assignments, and all tuples
$\mathbf y\in\mathcal Y(\mathfrak P,\tau)$, while keeping the window and
its complete profile fixed.

{
The lemma below organizes the target bookkeeping at three levels.  Part~\textup{(i)}
estimates a single target equation after summing its endpoint along the
appropriate clock layer.  Part~\textup{(ii)} treats several local target
equations by keeping one representative and paying $\Wnorm/m$ for
each local difference pivot.  Part~\textup{(iii)} combines these fixed
profile estimates and sums over the $O(m)$ possible local windows.  No
new triangularization is introduced here; the lemma applies the same
local-pivot/remote-chain decomposition with the target values included
in the final summation.
}

\begin{lemma}[Layered target systems]
\label{lem:layered-targets}
Fix $D\geq1$ and $0<a<1$.  There exist constants
$C_D=C_D(a,L)>0$ and $m_{\rm tar}=m_{\rm tar}(D,a,L)$ such that the
following holds.  Assume that $m\geq m_{\rm tar}$, where
$m_{\rm tar}$ is chosen at least as large as the threshold in
Lemma~\ref{lem:clock-chain}, that
\eqref{eq:clock-chain-nonperiodicity} holds, and that
$\mathbf W=(W_0,\ldots,W_{d-1})$ satisfies
\eqref{eq:terminal-compatibility}.  Then:
\begin{enumerate}[label=(\roman*)]
\item
\begin{equation}
\label{eq:single-layered-target-bound}
        \sum_{b=1}^m
        \Pr[Q_b\in W_{\operatorname{res}_d(b)}]
        \leq C_D\Theta_{d,a}(m,N,\mathbf W).
\end{equation}
\item Let $J$ be a $D$-local window, and let $\mathfrak P$ be a
      complete $J$-local profile of size $h+1$, where
      $0\leq h\leq D-1$, such that all equations have type $\mathsf{AT}$
      and hence form the unique $J$-reduction class of that type.  Let
      $E_0$ be the representative
      specified by $\mathfrak P$, and let $E_1,\ldots,E_h$ be the
      remaining equations.  Since there are no remote cut points, let
      $\tau_\emptyset$ be the unique empty $\mathsf{RRT}$-layer assignment
      and let $\mathbf r_\emptyset$ be the unique empty tuple.  Order
      $E_1,\ldots,E_h$ so that the rows
      \[
        F_{E_i}^{\mathbf r_\emptyset,\mathbf y}
        -F_{E_0}^{\mathbf r_\emptyset,\mathbf y}=0
        \qquad(1\leq i\leq h)
      \]
      have the right-to-left pivot order specified by $\mathfrak P$.
      Put $p_i:=\ell(E_i)$ for $1\leq i\leq h$, and fix
      $y_0\in W_{\operatorname{res}_d(v_{E_0})}$.  Let
      \[
       S_0:=J\setminus\{p_1,\ldots,p_h\},
      \]
      and let $\xi:S_0\to X$ be an injective map.  Then
      \begin{equation}
      \label{eq:local-target-pivots}
       \sum_{\substack{\mathbf y\in
                       \mathcal Y(\mathfrak P,\tau_\emptyset)\\
                       y_{E_0}=y_0}}
       {\Pr}_\xi\!\left[
          F_{E_i}^{\mathbf r_\emptyset,\mathbf y}
          -F_{E_0}^{\mathbf r_\emptyset,\mathbf y}=0
          \text{ for every }1\leq i\leq h
       \right]
       \leq C_D\left(\frac{\Wnorm}{m}\right)^h.
      \end{equation}
      Since $\Wnorm/m\leq\Theta_{d,a}(m,N,\mathbf W)$, the right-hand
      side is also at most
      $C_D\Theta_{d,a}(m,N,\mathbf W)^h$.
\item For every $1\leq h\leq D$,
      \begin{equation}
      \label{eq:mixed-obstruction-bound}
       \sum_{\substack{J=[j_-,j_+]\subseteq[m]\ \text{$D$-local window}\\
                       1<j_-,\ j_+<m}}
       \ \sum_{\mathfrak P\in\operatorname{Prof}^{\rm rep}_h(J)}
       \operatorname{Cont}(\mathfrak P)
       \leq C_Dm\Theta_{d,a}(m,N,\mathbf W)^h.
      \end{equation}
\end{enumerate}
\end{lemma}

\begin{proof}
{
For~(i), fix $t\in\{0,\ldots,d-1\}$ and $y\in W_t$. For $2\leq b<m$, put $u_b:=\min\{b,m-b\}$. If $b\leq m/2$, the equation $Q_b=y$ is a prescribed-sum equation for the uniformly random $b$-subset occupying the prefix $[1,b]$. If $b>m/2$, rewrite the same equation as a prescribed-sum equation for the complementary uniformly random $(m-b)$-subset. As $b$ ranges over one residue class modulo $d$, the relevant slice sizes range over at most two residue classes modulo $d$, one for $b$ and one for $m-b$.

The case $D=1$ of Proposition~\ref{prop:polynomial-nonperiodic-chain}, together with its conditioned form, gives uniformly in the target
\[
\Pr[Q_b=y]\leq \frac{C_D}{N}+C_D\frac{m^{(1-a)/2}\sqrt{\log(2+m)}}{m\sqrt{u_b}}+C_D\exp(-c_Dm^a)
\]
whenever $u_b\geq d$. Summing the uniform terms over one residue class contributes at most $C_D(m/(dN)+1/m)$, while grouping the remaining principal terms according to $jd\leq u_b<(j+1)d$ gives at most
\[
C_D\frac{m^{(1-a)/2}\sqrt{\log(2+m)}}{m}\sum_{1\leq j\leq m/d}\frac1{\sqrt{jd}}\leq C_D\frac{\sqrt{\log(2+m)}}{dm^{a/2}}.
\]
There are only $O(1)$ relevant values with $u_b<d$ in each residue class. Exposing all but one entry of the smaller of the prefix and its complement bounds their total contribution by $O_D(1/m)$. Finally, the sum of the exceptional terms is at most $C_Dm\exp(-c_Dm^a)\leq C_D/m$ after increasing $m_{\rm tar}$. Hence
\[
\sum_{\substack{2\leq b<m\\b\equiv t\ (\mathrm{mod}\ d)}}\Pr[Q_b=y]\leq C_D\vartheta_{d,a}(m,N).
\]
The endpoint $b=1$ contributes at most $1/m$, while the endpoint $b=m$ makes no contribution by \eqref{eq:terminal-compatibility}. Summing over $y\in W_t$ and then over $t$ proves \eqref{eq:single-layered-target-bound}.

We shall also use the following conditioned form. For each $2\leq b<m$, let $S_b\subseteq[m]$ have cardinality at most $\Lambda_D$, let $\xi_b:S_b\to X$ be injective, and assume that both $[1,b]\setminus S_b$ and $[b+1,m]\setminus S_b$ are nonempty. Let $\delta_b$ be determined by the values prescribed by $\xi_b$. Then
\begin{equation}
\label{eq:conditioned-single-layered-target-bound}
\sum_{t=0}^{d-1}\ \sum_{y\in W_t}\ \sum_{\substack{2\leq b<m\\b\equiv t\ (\mathrm{mod}\ d)}}{\Pr}_{\xi_b}\!\left[Q_b=y+\delta_b\right]\leq C_D\Theta_{d,a}(m,N,\mathbf W).
\end{equation}

For fixed $b$, conditional on $\xi_b$, the unexposed coordinates form a uniformly random ordering of
\[
X_b:=X\setminus\xi_b(S_b),
\]
where $|X_b|=m-O_D(1)$. Put
\[
u_b:=b-|S_b\cap[1,b]|,\qquad v_b:=|X_b|-u_b.
\]
The non-emptiness assumptions give $u_b,v_b\geq1$. Use the unexposed part of the prefix when $u_b\leq v_b$, and use its complement when $v_b<u_b$. The conditioned estimate from Proposition~\ref{prop:polynomial-nonperiodic-chain} is uniform in $S_b$, $\xi_b$, and the resulting translated target.

There are only $O_D(1)$ possible pairs $\bigl(|S_b\cap[1,b]|,|S_b|\bigr)$. After fixing such a pair and a residue class for $b$, both $u_b$ and $v_b$ range in fixed residue classes modulo $d$. Repeating the preceding residue-restricted summation, and absorbing the total exceptional contribution $C_Dm\exp(-c_Dm^a)$ into the $1/m$ term, gives \eqref{eq:conditioned-single-layered-target-bound}.
}

For~(ii), work under ${\Pr}_\xi$ and estimate the difference rows
successively in reverse pivot order.  When estimating the row indexed by
$E_i$, all its pivot variables other than $x_{p_i}$ have already been
exposed, and the coefficient of $x_{p_i}$ is $\pm1$.  Sum at this stage
over
$y_{E_i}\in W_{\operatorname{res}_d(v_{E_i})}$.  For each realization of
the preceding exposure, the row can hold for at most
$|W_{\operatorname{res}_d(v_{E_i})}|\leq \Wnorm$ values of
$x_{p_i}$ among the elements of $X$ not yet assigned to exposed
coordinates.  There are $\gg_Dm$ such elements, so the combined
conditional contribution of this sum and this row is at most
$C_D\Wnorm/m$.  Iterating over the distinct pivots and
using the tower property gives the first inequality in
\eqref{eq:local-target-pivots}.  The second follows from
$\vartheta_{d,a}(m,N)\geq1/m$ and
$\Theta_{d,a}(m,N,\mathbf W)=(1+\Wnorm)\vartheta_{d,a}(m,N)$.

For~(iii), split the sum according to the four classes of local-repair
profiles.

If $\mathfrak P$ is a left-clock or right-clock profile, then it is a
single-chain profile and has no target equations.  Every window in
\eqref{eq:mixed-obstruction-bound} satisfies $1<j_-$ and $j_+<m$, so we
may apply Lemma~\ref{lem:clock-chain} conditionally on the non-pivot
entries of $J$ and average over their values.  This gives
\[
        \operatorname{Cont}(\mathfrak P)
        \leq C_D\vartheta_{d,a}(m,N)^h
        \leq C_D\Theta_{d,a}(m,N,\mathbf W)^h.
\]

If $\mathfrak P$ is a remote-target singleton profile, condition first on
the entries in $J$.  For each $t\in\{0,\ldots,d-1\}$ and $y\in W_t$,
the event with terminal endpoint $m$ is empty by
\eqref{eq:terminal-compatibility}.  Lemma~\ref{lem:clock-chain}, applied
under this conditioning with the $\mathsf{RRT}$-layer assignment
$\tau(E)=t$, gives one factor $C_D\vartheta_{d,a}(m,N)$ after summing the
remaining remote endpoints congruent to $t$ modulo $d$.  Averaging over
the entries in $J$ and then summing $t$ and $y$ gives
\[
        \operatorname{Cont}(\mathfrak P)
        \leq C_D\Wnorm\vartheta_{d,a}(m,N)
        \leq C_D\Theta_{d,a}(m,N,\mathbf W).
\]

Suppose finally that $\mathfrak P$ is an anchored-target profile, and let
$E_0$ be the representative of its unique $\mathsf{AT}$ reduction class.
For a fixed value of $y_{E_0}$, apply part~\textup{(ii)} to the $h-1$
difference rows.  For every realization of the exposed coordinates for which these rows vanish,
all entries in $J$ are exposed, and their combined conditional contribution,
after summing the other coordinates of $\mathbf y$, is at most
\[
        C_D\left(\frac{\Wnorm}{m}\right)^{h-1}.
\]
The remaining representative equation has the form
\[
        Q_{v_{E_0}}=y_{E_0}+\delta_{v_{E_0}},
\]
where $\delta_{v_{E_0}}$ is determined by the exposed entries in $J$.
Since $1<j_-$ and $j_+<m$, its endpoint satisfies
$2\leq v_{E_0}<m$.  The sum over $y_{E_0}$ for this fixed endpoint is
bounded by the larger sum over all endpoints in
\eqref{eq:conditioned-single-layered-target-bound}.  The tower property
therefore gives
\[
        \operatorname{Cont}(\mathfrak P)
        \leq
        C_D\Theta_{d,a}(m,N,\mathbf W)
        \left(\frac{\Wnorm}{m}\right)^{h-1}
        \leq C_D\Theta_{d,a}(m,N,\mathbf W)^h.
\]

For each fixed window, Remark~\ref{rem:number-complete-local-profiles}
gives only $O_D(1)$ complete profiles of size $h$, and there are
$O_D(m)$ $D$-local windows.  Summing the preceding fixed-profile bounds
over these choices proves
\eqref{eq:mixed-obstruction-bound}.
\end{proof}

For every $1\leq h\leq D$, the proof also gives the fixed-profile
estimate
\begin{equation}
\label{eq:fixed-profile-contribution}
        \operatorname{Cont}(\mathfrak P)
        \leq C_D\Theta_{d,a}(m,N,\mathbf W)^h
        \qquad
        (\mathfrak P\in\operatorname{Prof}^{\rm rep}_h(J),\ 
         J=[j_-,j_+]\text{ with }1<j_-,\ j_+<m),
\end{equation}
with the same parameter dependence.  In the anchored-target case this is
the combination of part~\textup{(ii)} with
\eqref{eq:conditioned-single-layered-target-bound}; the other three cases
are the corresponding fixed-profile estimates in the proof of
part~\textup{(iii)}.

\begin{unnumberedremark}[Local-pivot refinements]
{We record two sharper forms of the preceding estimates which are already contained in their proofs. First, under the hypotheses of Lemma~\ref{lem:clock-chain}, if a single-chain profile has $k$ local pivot rows and $h_{\mathrm{rem}}$ rows in its triangular remote chain, then $h=k+h_{\mathrm{rem}}$ and
\begin{equation}
\label{eq:pivot-refined-clock-chain}
 \sum_{\substack{\mathbf r\in\mathcal R(\mathfrak P,\tau)\\
                   r_E(\mathbf r)<m\ \text{for every equation $E$ of type $\mathsf{RRT}$}}}
 {\Pr}_\xi[\mathsf A(\mathfrak P,\tau,\mathbf r,\mathbf y)]
 \leq C_Dm^{-k}\vartheta_{d,a}(m,N)^{h_{\mathrm{rem}}}.
\end{equation}
The exponentially small additive error term is absorbed into the right-hand side after increasing the same threshold, since $m^{-k}\vartheta_{d,a}^{h_{\mathrm{rem}}}\geq m^{-h}\geq m^{-D}$.}

Second, let $1\leq k\leq D-1$ and consider the complete
$J$-local profiles of size $k+1$ having exactly one equation of type
$\mathsf{AT}$ and $k$ equations of type $\mathsf{LC}$.  The latter are
the $k$ local-pivot rows, while the unique target equation is the
representative of the $\mathsf{AT}$ class.  Then
\begin{equation}
\label{eq:anchored-target-clock-pivots}
 \sum_{\substack{J=[j_-,j_+]\subseteq[m]\ \text{$D$-local window}\\
                 1<j_-,\ j_+<m}}
 \ \sum_{\substack{\mathfrak P\ \text{complete $J$-local profile of size $k+1$}\\
                    \mathfrak P\ \text{has one $\mathsf{AT}$ equation and $k$ $\mathsf{LC}$ equations}}}
       \operatorname{Cont}(\mathfrak P)
 \leq C_Dm^{-k}\Theta_{d,a}(m,N,\mathbf W).
\end{equation}
Condition on the non-pivot entries of $J$ and estimate the $k$
local-pivot equations first.  Their successive conditional probabilities
give a factor $C_Dm^{-k}$.  On every realization for which all $k$
local-pivot equations hold, all entries in $J$ are exposed.  The remaining
$\mathsf{AT}$ equation is then
of the form
\[
        Q_b=y_E+\delta_b,
\]
where $b\in J$, $2\leq b<m$, and $\delta_b$ is determined by the exposed
entries.  Equation~\eqref{eq:conditioned-single-layered-target-bound}
sums $b$ and $y_E$ with cost
$C_D\Theta_{d,a}(m,N,\mathbf W)$.  For each fixed endpoint there are only
$O_D(1)$ windows satisfying the displayed restrictions, and
Remark~\ref{rem:number-complete-local-profiles} gives only $O_D(1)$
complete profiles for each window.  Averaging over the initial
conditioning and using the tower property proves
\eqref{eq:anchored-target-clock-pivots}.
\end{unnumberedremark}

{

The preceding lemmas estimate every local-repair profile.  It remains to
show that a bounded family of failed candidates canonically produces one
of them.
}

We next record the permutation data used in the repair argument.  For
$u<v$, let $\pi_{u,v}$ denote the transposition of the positions $u$ and
$v$.  Fix $\Lambda\in\mathbb N$ with
$1\leq\Lambda\leq\Lambda_D$.  For $b\in[m]$, an
\emph{admissible permutation at $b$} is a permutation
\[
        \pi=\prod_{i=1}^t\pi_{b_i,z_i},
        \qquad t\geq0,
\]
whose transpositions are pairwise disjoint and satisfy
\[
        b<b_i<z_i\leq\min\{m,b_i+\Lambda\}
        \qquad(1\leq i\leq t).
\]
The empty product is the identity.  The support of $\pi$ is
$\supp(\pi)=\{b_i,z_i:1\leq i\leq t\}$, and $\pi$ fixes every position
in $[b]$.  A position
\[
        z\in(b,b+\Lambda]\setminus\supp(\pi)
\]
is a \emph{candidate for $(b,\pi)$}.

\begin{unnumberedremark}[Restriction to fixed supports]
Let $\mathcal I$ be a family of at most $D+1$ intervals or prefixes of
$[m]$, and let $\pi$ be admissible at $b$.  Define
\begin{equation}
\label{eq:restricted-repair-permutation}
        \pi_{\mathcal I}:=
        \prod_{\substack{(u,v)\text{ a transposition in }\pi\\
                    |\{u,v\}\cap I|=1\text{ for some }I\in\mathcal I}}
        (u,v).
\end{equation}
Then, for every $I\in\mathcal I$,
\begin{equation}
\label{eq:support-sum-restriction}
        \sum_{j\in I}x_{\pi(j)}
        =\sum_{j\in I}x_{\pi_{\mathcal I}(j)}.
\end{equation}
Once the supports in $\mathcal I$ are fixed, $\pi_{\mathcal I}$ has
$O_D(1)$ transposition factors, and there are only $O_D(1)$ possible
choices for $\pi_{\mathcal I}$.  Consequently, any event determined by obstruction equations
whose supports belong to $\mathcal I$ depends on $\pi$ only through
$\pi_{\mathcal I}$; the union over all admissible $\pi$ is contained in
a union over at most $C_D$ possible choices of $\pi_{\mathcal I}$.

To verify these assertions, observe first that the transpositions in $\pi$ are disjoint and therefore commute.  Consider
a transposition $(u,v)$ not included in
\eqref{eq:restricted-repair-permutation}.  For each $I\in\mathcal I$,
either both $u,v$ belong to $I$ or neither belongs to $I$.  In the first
case $(u,v)$ merely interchanges two summands of the sum over $I$; in the
second case it changes no coordinate indexed by $I$.  Deleting all such
transpositions therefore gives \eqref{eq:support-sum-restriction}.

Let $(u,v)$, with $u<v$, be a transposition included in
$\pi_{\mathcal I}$.  For some $I\in\mathcal I$, exactly one of $u,v$
belongs to $I$.  If ${I=[s,t]}$ is an interval, then either
\[
        {u<s\leq v
        \qquad\text{or}\qquad
        u\leq t<v;}
\]
for a prefix ${I=[1,t]}$, only the second alternative can occur.  Since
$v-u\leq\Lambda\leq\Lambda_D$, both endpoints of $(u,v)$ lie in
\[
        [r-\Lambda_D,r+\Lambda_D]\cap[m]
\]
for one endpoint $r$ of one of the supports in $\mathcal I$.  Let $\operatorname{End}(\mathcal I)$ consist of both endpoints of each
interval in $\mathcal I$ and the right endpoint of each prefix in
$\mathcal I$.  Hence all positions moved by $\pi_{\mathcal I}$ belong
to the set
\[
        S(\mathcal I):=
        \bigcup_{r\in\operatorname{End}(\mathcal I)}
        [r-\Lambda_D,r+\Lambda_D]\cap[m].
\]
Since $\mathcal I$ has at most $D+1$ members,
\[
        |S(\mathcal I)|
        \leq2(D+1)(2\Lambda_D+1)=O_D(1).
\]
The transpositions in $\pi_{\mathcal I}$ are pairwise disjoint, so they
form a matching on $S(\mathcal I)$.  They are therefore $O_D(1)$ in
number, and the number of possible such matchings is also $O_D(1)$.
Finally, the constant terms of the obstruction equations do not depend
on $\pi$, so \eqref{eq:support-sum-restriction} gives the asserted
consequence for events determined by the fixed supports.
\end{unnumberedremark}

Let $z$ be a candidate for $(b,\pi)$, and let $E'$ be an obstruction
equation in the ordering obtained after applying
$\pi\circ\pi_{b,z}$.  Suppose that $\supp(E')$ contains exactly one of
$b,z$.  We call $E'$, together with its affine form pulled back to the
original variables, a \emph{blocker equation for $z$ relative to
$(b,\pi)$}; the same symbol will be used for the pulled-back affine form
when no confusion can arise.  The blocker is \emph{$b$-oriented} when
$b\in\supp(E')$ and $z\notin\supp(E')$, and \emph{$z$-oriented} when
$z\in\supp(E')$ and $b\notin\supp(E')$.

These two cases are precisely the ones in which the transposition
$\pi_{b,z}$ can change the sum over the support: if the support contains
both $b,z$, the transposition merely interchanges two summands, whereas
if it contains neither position, it changes no summand.  In particular,
an obstruction equation that holds for exactly one of the orderings
before and after applying $\pi_{b,z}$ must have one of the two
orientations above.  A
$z$-oriented blocker cannot be a target equation.  Indeed, the support of
a target equation is a prefix $[1,t]$, and a prefix containing $z>b$
also contains $b$.  Thus every $z$-oriented blocker is a clock-collision
equation.

\begin{unnumbereddefinition}[Blocker data and canonical profiles]
Fix $b\in[m]$, an admissible permutation $\pi$ at $b$, distinct
candidates $\mathbf z=(z_1,\ldots,z_h)$, and a $D$-local window
$J$ containing $b,z_1,\ldots,z_h$.  For each $i\in[h]$, let $E_i$ be a
blocker equation for $z_i$ relative to $(b,\pi)$, and put
$\mathbf E=(E_1,\ldots,E_h)$.  
{Set $\widetilde\sigma:=\sigma^\pi$ and write
$\widetilde x_i:=\widetilde\sigma(i)$.  Since
\[
        \sigma^{\pi\circ\pi_{b,z_i}}
        =\widetilde\sigma^{\pi_{b,z_i}}
        \qquad(1\leq i\leq h),
\]
pull the obstruction equation indexed by $E_i$ back from the ordering
$\widetilde\sigma^{\pi_{b,z_i}}$ to the coordinate variables
$\widetilde x_1,\ldots,\widetilde x_m$.  Assume that every equation
obtained in this way is $J$-admissible.  We then rename
$\widetilde\sigma$ as $\sigma$ and its coordinates $\widetilde x_i$ as
$x_i$; all types, reduction classes, and coefficient restrictions
$\varepsilon_{E_i}|_J$ below refer to these pulled-back equations.}

We call
\[
        \BlockDatum=(J;b,\pi,\mathbf z,\mathbf E)
\]
a \emph{blocker datum} when exactly one of the following conditions holds.
Its \emph{size} is $h$.
\begin{enumerate}[label=(B\arabic*),leftmargin=2.8em]
\item Every $E_i$ is $b$-oriented, and either all $E_i$ are clock
      equations or all $E_i$ are target equations.
\item Every $E_i$ is $z_i$-oriented, and, writing the interval
      support of $E_i$ before pullback through $\pi_{b,z_i}$ as
      $[s_i,t_i]$, the left cut point $s_i-1$ is $J$-local and the right
      cut point $t_i$ is $J$-remote.
\end{enumerate}
For a blocker datum $\BlockDatum$, choose the representatives, marking, and
orders as follows.
\begin{enumerate}[label=(C\arabic*),leftmargin=2.8em]
\item Under \textup{(B1)}, for each $J$-reduction class other than the
      $\mathsf{LC}$ class, choose as representative the unique label
      \[
            R_\Gamma=E_{i_\Gamma},
            \qquad
            z_{i_\Gamma}
            =\min\{z_i:E_i\in\Gamma\}.
      \]
      Mark every label of type $\mathsf{LC}$ and every
      non-representative label $E_i$ by $\ell(E_i):=z_i$.
\item Under \textup{(B2)}, for each $J$-reduction class $\Gamma$, choose
      \[
            R_\Gamma=E_{i_\Gamma},
            \qquad
            (s_{i_\Gamma},i_\Gamma)
            =\min_{\rm lex}\{(s_i,i):E_i\in\Gamma\}.
      \]
      Mark every non-representative label $E_i$ by
      $\ell(E_i):=z_i$.
\item In either case, order the marked labels by strictly decreasing
      marked position $z_i$, and order the remote reduction classes by
      strictly increasing numerical remote cut point.
\end{enumerate}
\end{unnumbereddefinition}

The choices in \textup{(C1)} and \textup{(C2)} are unambiguous.  In
\textup{(C1)}, the candidates $z_i$ are pairwise distinct, so each nonempty
reduction class has a unique label with minimum candidate.  In
\textup{(C2)}, the representative has minimum left endpoint $s_i$ within
its class, while the second coordinate in the lexicographic order resolves
ties; hence the pair $(s_i,i)$ has a unique minimum.  Consequently,
\textup{(C1)}--\textup{(C3)} determine at most one complete $J$-local
profile.

If the choices in \textup{(C1)}--\textup{(C3)} satisfy the right-to-left
conditions, denote the resulting complete $J$-local profile by
\[
        \operatorname{Can}(\BlockDatum)
\]
and call it the \emph{canonical profile} of $\BlockDatum$.  Thus
$\operatorname{Can}$ is a partial map on blocker data; denote its domain by
$\operatorname{Dom}(\operatorname{Can})$.

\begin{example}[A canonical blocker profile]
\label{ex:canonical-blocker-profile}
Assume that $\Lambda\geq3$.  Let $d=2$, $J=[7,10]$, $b=7$,
$\pi=\mathrm{id}$, and take the distinct candidates $z_1=8$ and
$z_2=10$.  Consider two clock-collision equations whose
interval supports, before pullback through $\pi_{7,8}$ and $\pi_{7,10}$,
are respectively $[4,7]$ and $[4,9]$.  Their pulled-back affine forms are
\[
 \begin{aligned}
  F_{E_1}&=x_4+x_5+x_6+x_8+2c,\\
  F_{E_2}&=x_4+x_5+x_6+x_8+x_9+x_{10}+3c.
 \end{aligned}
\]
Both equations are $b$-oriented, have type $\mathsf{LRC}$, and have the
same remote left cut point $3$.  Hence they form one $J$-reduction class.
Since $z_1<z_2$, rule \textup{(C1)} chooses $E_1$ as its representative
and marks $E_2$ by $\ell(E_2)=10$.  The corresponding local row is
\[
        F_{E_2}-F_{E_1}=x_9+x_{10}+c,
\]
which has coefficient $1$ at its pivot $x_{10}$.  There is only one remote
class, so its order is automatic.  Thus \textup{(C1)}--\textup{(C3)}
determine a complete $J$-local profile consisting of one remote
representative and one local-pivot row.
\end{example}

For fixed $J,b,\pi$, and $\mathbf z$, define
\begin{equation}
\label{eq:blocker-profile-set}
 \operatorname{Prof}^{\rm blk}_h(J;b,\pi,\mathbf z)
 :=\{\operatorname{Can}(\BlockDatum):
       \BlockDatum=(J;b,\pi,\mathbf z,\mathbf E)
       \in\operatorname{Dom}(\operatorname{Can})\}.
\end{equation}
Whenever $\operatorname{Can}(\BlockDatum)$ is defined, condition
\textup{(B1)} makes it a left-clock profile or an anchored-target profile,
according as the equations are clock-collision equations or target equations, while
condition \textup{(B2)} makes it a right-clock profile.  Since it arises
from a right-to-left $D$-bounded obstruction system of size $h$, it follows
from the definition of $\operatorname{Prof}^{\rm rep}_h(J)$ that
\[
 \operatorname{Prof}^{\rm blk}_h(J;b,\pi,\mathbf z)
 \subseteq \operatorname{Prof}^{\rm rep}_h(J).
\]

\begin{lemma}[Blocker systems]
\label{lem:blocker-systems}
Fix $D\geq1$, $0<a<1$, and $\Lambda\in\mathbb N$ with
$1\leq\Lambda\leq\Lambda_D$.  There is a constant
$C_D=C_D(a,L)>0$ such that the following holds under the hypotheses of
Lemma~\ref{lem:layered-targets}.  Let
$\BlockDatum=(J;b,\pi,\mathbf z,\mathbf E)$ be a blocker datum of size $1\leq h\leq D$, and assume that
$J=[j_-,j_+]$ satisfies $1<j_-$ and $j_+<m$.  Then $\operatorname{Can}(\BlockDatum)$ is defined and is a right-to-left
complete $J$-local profile in
$\operatorname{Prof}^{\rm rep}_h(J)$.  Moreover, with $\operatorname{Cont}$ as in
\eqref{eq:profile-total-contribution},
\begin{equation}
\label{eq:blocker-system-bound}
 \sum_{\mathfrak P\in
       \operatorname{Prof}^{\rm blk}_h(J;b,\pi,\mathbf z)}
       \operatorname{Cont}(\mathfrak P)
 \leq C_D\Theta_{d,a}(m,N,\mathbf W)^h.
\end{equation}
\end{lemma}

\begin{proof}

{Put $\widetilde\sigma:=\sigma^\pi$ and write
$\widetilde x_i:=\widetilde\sigma(i)$.  By
\eqref{eq:permutation-absorption}, $\widetilde\sigma$ is a uniformly
random ordering of $X$.  Moreover,
\[
        \sigma^{\pi\circ\pi_{b,z_i}}
        =\widetilde\sigma^{\pi_{b,z_i}}
        \qquad(1\leq i\leq h).
\]
Hence, after replacing $\sigma$ by $\widetilde\sigma$, the equation
indexed by $E_i$ is obtained by pulling its obstruction equation back
through $\pi_{b,z_i}$.  Since $b,z_i\in J$, this transposition is
supported on $J$.  We rename $\widetilde\sigma$ as $\sigma$, rename
its coordinates $\widetilde x_i$ as $x_i$, and write $F_{E_i}$ for the
resulting affine left-hand side.}  We now verify that the choices in
\textup{(C1)}--\textup{(C3)} satisfy the right-to-left conditions.

Assume first condition~\textup{(B1)}.  Before pullback through
$\pi_{b,z_i}$, the support of $E_i$ contains $b$ and does not contain
$z_i$.  Since $b<z_i$ and the support is an interval or a prefix, its
right endpoint is smaller than $z_i$.  Pulling back through
$\pi_{b,z_i}$ therefore gives coefficient $\pm1$ at $x_{z_i}$, and the
resulting row contains no candidate larger than $z_i$.

In every $J$-reduction class other than the $\mathsf{LC}$ class, the
representative selected in \textup{(C1)} has the smallest attached
candidate.  Mark every equation of type $\mathsf{LC}$ and every
non-representative equation $E_i$ by $\ell(E_i):=z_i$.  If $E_i$ is
non-representative and $R$ is its class representative, then the support
of $R$ before pullback has right endpoint smaller than the candidate
attached to $R$, which is itself smaller than $z_i$.  Hence $F_R$ has
coefficient zero at $x_{z_i}$, whereas $F_{E_i}$ has coefficient $\pm1$;
thus
\[
        [x_{z_i}](F_{E_i}-F_R)=\pm1.
\]
Moreover, neither $F_{E_i}$ nor $F_R$ contains a candidate larger than
$z_i$.  Ordering the marked rows by decreasing candidate position
therefore gives \eqref{eq:local-pivot-family}.

Assume next condition~\textup{(B2)}.  Every equation is of type
$\mathsf{RRC}$.  Fix one reduction class.  Its equations have a common
remote right cut point $t$; write the interval support of $E_i$ before
pullback through $\pi_{b,z_i}$ as $[s_i,t]$.  Then
\[
        b<s_i\leq z_i\leq t,
\]
and the pullback through $\pi_{b,z_i}$ replaces the term $x_{z_i}$ by
$x_b$.

Let $R=E_r$ be the representative selected in \textup{(C2)}.  Thus
$s_r\leq s_i$ for every equation $E_i$ in the class; the second coordinate
in the lexicographic order only resolves ties among equal values of
$s_i$.  Mark every non-representative equation $E_i$ by
$\ell(E_i):=z_i$.  If $E_i\neq R$, then
\[
        s_r\leq s_i\leq z_i\leq t,
\]
so the interval $[s_r,t]$ contains $z_i$.  The pullback defining
$F_{E_i}$ removes $x_{z_i}$, whereas the pullback defining $F_R$ removes
$x_{z_r}$; since the candidates are distinct, $z_r\neq z_i$, and $F_R$
still contains $x_{z_i}$.  Consequently,
\[
        [x_{z_i}](F_{E_i}-F_R)=\pm1.
\]

Now let $z_j>z_i$ be a larger marked candidate.  If $z_j>t$, neither
interval $[s_i,t]$ nor $[s_r,t]$ contains $z_j$.  If $z_j\leq t$, then
both intervals contain $z_j$ because $s_r\leq s_i\leq z_i<z_j$.  Since
$z_j$ is marked, its equation is not the representative $R$, and hence
neither of the two pullbacks removes $x_{z_j}$.  Thus
$[x_{z_j}](F_{E_i}-F_R)=0$ in either case.  The decreasing order of the
marked candidates is therefore a right-to-left local-pivot order.
Distinct reduction classes retain distinct right remote cut points, and
their representatives form the $\mathsf{RRC}$ triangular remote chain of
Lemma~\ref{lem:clock-triangularization}.

Thus $\operatorname{Can}(\BlockDatum)$ is defined and is a right-to-left complete
$J$-local profile.  Under~\textup{(B1)}, blocker equations that are clock-collision equations
give a left-clock profile, while blocker equations that are target equations
give an anchored-target profile;
under~\textup{(B2)} the profile is right-clock.  Hence
$\operatorname{Can}(\BlockDatum)\in
\operatorname{Prof}^{\rm rep}_h(J)$.

By Remark~\ref{rem:number-complete-local-profiles}, a fixed window $J$
admits only $O_D(1)$ complete $J$-local profiles of size $h$.  Since
$\operatorname{Prof}^{\rm blk}_h(J;b,\pi,\mathbf z)$ is a subfamily of
these profiles,
$|\operatorname{Prof}^{\rm blk}_h(J;b,\pi,\mathbf z)|=O_D(1)$.
For each fixed profile, the fixed-profile estimate
\eqref{eq:fixed-profile-contribution} gives at most
$C_D\Theta_{d,a}(m,N,\mathbf W)^h$ after carrying out the sums in
\eqref{eq:profile-total-contribution}.  Summing over these $O_D(1)$ profiles proves
\eqref{eq:blocker-system-bound}.
\end{proof}

{

We now isolate the combinatorial part of the argument.  This is the point
at which the proof becomes algorithmic; no probability estimate is used in
Lemma~\ref{lem:deterministic-layered-repair}.

The tail label at position $i$ is $g+x_i$, with $x_i\in K$, and therefore
its image in $G/K$ is $g+K$.  Consequently, the image in $G/K$ of every
prefix of length $j$ is $j(g+K)$, independently of the ordering of the
coordinates $x_1,\ldots,x_m$.   {Interchanging two coordinates $x_b$ and $x_z$} therefore does not change $\operatorname{res}_d(j)$ or the set
$W_{\operatorname{res}_d(j)}$ attached to an endpoint $j$.

For a bijection $\sigma:[m]\to X$, write $x_i:=\sigma(i)$ and call $b$ a
\emph{bad endpoint} if at least one of the following occurs:
\begin{enumerate}[label=(\alph*),leftmargin=2.6em]
\item $\mathcal C([s,b])$ holds for some $1\leq s\leq b$ with
      $d\mid(b-s+1)$;
\item $\mathcal T(b,y)$ holds for some
      $y\in W_{\operatorname{res}_d(b)}$.
\end{enumerate}
Denote the set of bad endpoints of the initial ordering by $B(\sigma)$.
The first alternative records a collision between two tail vertices, and
the second is equivalent to
$Q_b\in W_{\operatorname{res}_d(b)}$.

Fix $D\geq1$, put $\Lambda:=10D$, and define
\[
 \partial:=
 \bigl([1,4\Lambda+1]\cup[m-4\Lambda,m]\bigr)\cap[m].
\]
Let $b\in B(\sigma)$ and let $\pi$ be admissible at $b$.  A position
\[
 \RepairPos\in(b,b+\Lambda]\setminus
 \bigl(\supp(\pi)\cup B(\sigma)\bigr)
\]
is called \emph{available}.  It is \emph{blocked} for
$(\sigma,b,\pi)$ if some obstruction equation with right endpoint at
least $b$ is false for $\sigma\circ\pi$ and true for
$\sigma\circ\pi\circ\pi_{b,\RepairPos}$.

An obstruction can change under $\pi_{b,\RepairPos}$ only if its support contains
exactly one of $b$ and $\RepairPos$.  Thus every blocker is one of the blocker
equations studied above.  Conversely, every obstruction ending exactly at
$b$ contains $b$ and not $\RepairPos$; since $x_b\neq x_{\RepairPos}$, the transposition
changes its left-hand side by the non-zero element $x_{\RepairPos}-x_b$.  Hence every
obstruction ending at $b$ that is true before the transposition is false
afterwards.  The notion of blocked candidate therefore records precisely
the only possible failure of a repair step: the creation of a new
obstruction at an endpoint already processed or currently being processed.

\begin{lemma}[Deterministic right-to-left repair]
\label{lem:deterministic-layered-repair}
Let $\sigma:[m]\to X$ be an ordering.  Assume that
\begin{enumerate}[label=(R\arabic*),leftmargin=2.8em]
\item $B(\sigma)\cap\partial=\varnothing$;
\item every interval of $2\Lambda+1$ consecutive positions contains at
      most $2D$ elements of $B(\sigma)$;
\item for every $b\in B(\sigma)$ and every admissible permutation $\pi$
      at $b$, at most $4D$ available positions in $(b,b+\Lambda]$ are
      blocked for $(\sigma,b,\pi)$.
\end{enumerate}
Then there is a product $\pi_*$ of pairwise disjoint transpositions
$\pi_{b_i,\RepairPos_i}$, with
\[
        b_i<\RepairPos_i\leq b_i+\Lambda,
\]
such that no obstruction equation holds for $\sigma\circ\pi_*$.
\end{lemma}

\begin{proof}
Examine the positions $b=m,m-1,\ldots,1$ in decreasing order.  Starting
with $\sigma_m:=\sigma$, construct
$\sigma_m,\sigma_{m-1},\ldots,\sigma_0$, where $\sigma_{b-1}$ is the
ordering after the stage indexed by $b$.  Before that stage we maintain:
\begin{enumerate}[label=(I\arabic*),leftmargin=2.8em]
\item no obstruction equation with right endpoint greater than $b$ holds
      for $\sigma_b$;
\item $\sigma_b=\sigma\circ\pi$, where $\pi$ is a product of pairwise
      disjoint transpositions $\pi_{b_i,\RepairPos_i}$ chosen at earlier stages,
      and every such pair satisfies
      $b<b_i<\RepairPos_i\leq b_i+\Lambda$.
\end{enumerate}
For $b=m$, condition~\textup{(I1)} is vacuous and~\textup{(I2)} holds
with $\pi=\mathrm{id}$.

Suppose first that $b\notin B(\sigma)$.  By~\textup{(I2)}, every
transposition in $\pi$ uses positions greater than $b$, so $\pi$ fixes
$[b]$.  Every obstruction equation with right endpoint $b$ is therefore
unchanged from the initial ordering and none holds.  Set
$\sigma_{b-1}:=\sigma_b$.

Now suppose that $b\in B(\sigma)$.  The permutation $\pi$ in
\textup{(I2)} is admissible at $b$.  By~\textup{(R1)},
$(b,b+\Lambda]\subseteq[m]$.  By~\textup{(R2)}, this interval contains
at most $2D$ positions of $B(\sigma)$.  If a previously chosen partner
$\RepairPos_i$ lies in $(b,b+\Lambda]$, then
\[
        b<b_i<\RepairPos_i\leq b+\Lambda,
\]
so its bad endpoint $b_i$ lies in the same interval.  There are therefore
at most $2D$ previously chosen partners in $(b,b+\Lambda]$.  The left
endpoints $b_i$ of the earlier transpositions already belong to
$B(\sigma)$, and hence
$B(\sigma)\cup\supp(\pi)$ excludes at most $4D$ positions from this
interval.  At least
\[
        \Lambda-4D=6D
\]
positions are available.  By~\textup{(R3)}, at most $4D$ of them are
blocked, so choose an available unblocked position $\RepairPos$ and set
\[
        \sigma_{b-1}:=\sigma_b\circ\pi_{b,\RepairPos}.
\]
Availability makes this transposition disjoint from all previous ones,
so~\textup{(I2)} remains valid at the next stage.

It remains to check~\textup{(I1)}.  An obstruction with right endpoint
smaller than $b$ is supported in $[b-1]$ and is unchanged by the exchange
of positions $b$ and $\RepairPos$.  Every obstruction ending at $b$ that holds
before the exchange is destroyed, as observed above.  Finally, because
$\RepairPos$ is not blocked, no obstruction with right endpoint at least $b$ that
was false before the exchange becomes true afterwards.  Together with the
induction hypothesis, this proves that no obstruction with right endpoint
at least $b$ holds for $\sigma_{b-1}$.  Thus~\textup{(I1)} also propagates.

After the stage $b=1$, no obstruction equation holds for $\sigma_0$.
Taking $\pi_*$ to be the product of all chosen transpositions completes
the proof.
\end{proof}

It remains only to show that a uniformly random ordering avoids the three
failure modes in Lemma~\ref{lem:deterministic-layered-repair}.  The
auxiliary event~\textup{(E0)} below is not a fourth deterministic
hypothesis: it is introduced solely to ensure that a family of
$\RepairPos$-oriented blockers has one genuinely remote cut point, so that
Lemma~\ref{lem:blocker-systems} applies.
}

{
For reference, the final first-moment argument has the following
structure:
\[
\begin{array}{c|c|c}
\text{event}&\text{deterministic condition}&\text{probabilistic input}\\ \hline
\textup{(E1)}&\textup{(R1): boundary}&\text{one-equation clock/target sums}\\
\textup{(E2)}&\textup{(R2): clustering}&\text{local-repair profiles of size }D\\
\textup{(E3)}&\textup{(R3): blocking}&\text{canonical blocker profiles of size }D\\
\textup{(E0)}&\text{auxiliary only}&\text{localization of a right-remote cut point}
\end{array}
\]
The first and auxiliary events cost $O_D(\Theta)$, while the two
high-multiplicity failures cost $O_D(\Theta+m\Theta^D)$.  Choosing
$D>3/\rho$ and using $\Theta\leq m^{-\rho}$ makes their union have
probability less than one.
}

\begin{proof}[Proof of Lemma~\ref{lem:layered-local-repair}]
{
Write
\[
        \Theta:=\Theta_{d,a}(m,N,\mathbf W)
\]
and choose an integer $D>3/\rho$. All constants below are now fixed. Set $\Lambda:=10D$. By the definition of $\Lambda_D$, one has $4\Lambda+1=\Lambda_D$. Keep the boundary set $\partial$ from Lemma~\ref{lem:deterministic-layered-repair}. Thus $|\partial|\leq8\Lambda+2=O_D(1)$, and if $b\notin\partial$, then
$1<b$ and $b+4\Lambda<m$.

The events below have transparent roles.  Event~\textup{(E1)} is the
boundary failure in~\textup{(R1)}, event~\textup{(E2)} is the clustering
failure in~\textup{(R2)}, and event~\textup{(E3)} is the blocking failure
in~\textup{(R3)}.  Event~\textup{(E0)} is only an auxiliary localization
event used in the estimate of~\textup{(E3)}.
}

Let \emph{(E0)} be the event that there exist
$b\in B(\sigma)\setminus\partial$, an admissible permutation $\pi$ at
$b$, and an available $\RepairPos$ for which a $\RepairPos$-oriented blocker equation with
right cut point at most $b+3\Lambda$ holds.
We also exclude the following events:
\begin{enumerate}[label=(E\arabic*)]
\item $B(\sigma)\cap\partial\neq\emptyset$;
\item some interval of $2\Lambda+1$ consecutive positions contains more
      than $2D$ elements of $B(\sigma)$;
\item for some $b\in B(\sigma)$ and some admissible permutation at $b$,
      more than $4D$ available positions in $(b,b+\Lambda]$ are blocked.
\end{enumerate}

We first estimate~(E0).  Put $J_b=[b,b+3\Lambda]$.  Since
$|J_b|=3\Lambda+1\leq\Lambda_D$, this is a $D$-local window.  By the
definition of $\partial$, the condition $b\notin\partial$ gives
$1<b$ and $b+3\Lambda<m$.  A $\RepairPos$-oriented blocker equation counted by
(E0) is a clock-collision equation.  Before pullback through
$\pi_{b,\RepairPos}$, write its interval support as $[s,t]$.  Since this interval
contains $\RepairPos$ but not $b$, one has
\[
        b<s\leq \RepairPos\leq t\leq b+3\Lambda.
\]
Consequently both cut points are $J_b$-local.  The interval cannot be a
singleton: otherwise its clock-collision relation would assert that the
corresponding tail label is zero, contrary to $0\notin g+X$.  After
pullback through $\pi_{b,\RepairPos}$, its affine form therefore has coefficient
$\pm1$ at some position $j\in J_b\setminus\{\RepairPos\}$; mark the row by the
largest such position.  For fixed $b$ and $\RepairPos$, there are only $O_D(1)$
possible interval supports, markings, and choices of $\pi_{\mathcal I}$.

Choose one obstruction equation with right endpoint $b$ that holds for the
initial ordering $\sigma$; such an equation exists because
$b\in B(\sigma)$.  Since every admissible permutation at $b$ fixes
$[b]$, the same equation holds for $\sigma\circ\pi$.  Pair its affine
row with the marked row above.  If the equation belongs to a
$J_b$-reduction class other than the $\mathsf{LC}$ class, choose it as the
representative of that class.  If it has type $\mathsf{LC}$, mark it by
$b$.  Its support is contained in $[b]$, so its affine form has coefficient
$\pm1$ at $x_b$ and coefficient zero at $x_j$.  Since $j>b$, the marked
rows have the right-to-left pivot order $j>b$ in the latter case; in the
former case the blocker equation supplies the only local-pivot row.  Thus
each such pair determines a right-to-left two-row system with a complete
$J_b$-local profile.

We estimate these systems using the local-pivot refinements
\eqref{eq:pivot-refined-clock-chain} and
\eqref{eq:anchored-target-clock-pivots}.  If the obstruction equation with
right endpoint $b$ is a clock-collision equation of type $\mathsf{LRC}$,
the system has one local-pivot row and one remote representative row;
after summing over the $O(m)$ possible values of $b$, its contribution is
at most $C_D\vartheta_{d,a}(m,N)$.  If it is a clock-collision equation of
type $\mathsf{LC}$, the system has two local-pivot rows and no remote row;
summing over $b$ gives $C_D/m$.  If it is a target equation, it has type
$\mathsf{AT}$ and the system has one $\mathsf{AT}$ representative and one
$\mathsf{LC}$ local-pivot row.  In this case
\eqref{eq:anchored-target-clock-pivots} gives $C_D\Theta/m$ after summing
its endpoint and target.  Since
$1/m\leq\vartheta_{d,a}(m,N)\leq\Theta$, all three contributions are at
most $C_D\Theta$.  The bounded number of candidates, complete local
profiles, and choices of $\pi_{\mathcal I}$ is absorbed into $C_D$.
Therefore
\begin{equation}
\label{eq:E0-clock-bound}
        \Pr[(\mathrm{E0})]\leq C_D\Theta.
\end{equation}

{
For~(E1), fix $b\in\partial$ and consider a clock-collision equation $\mathcal C([s,b])$. Put $\ell:=b-s+1$, so that the admissibility condition is $d\mid\ell$. The possible full interval $\ell=m$ is excluded by \eqref{eq:terminal-compatibility}. For $\ell<m$, put $r_\ell:=\min\{\ell,m-\ell\}$.

If $\ell\leq m/2$, the entries occupying $[s,b]$ form a uniformly random $\ell$-subset of $X$. If $\ell>m/2$, use the complementary set of positions instead. Since $\mathcal C([s,b])$ is the equation
\[
\sum_{i=s}^{b}x_i+\frac{\ell}{d}c=0,
\]
it is equivalently the prescribed-sum equation
\[
\sum_{i\notin[s,b]}x_i=P_m+\frac{\ell}{d}c
\]
for the uniformly random complementary subset of size $m-\ell=r_\ell\leq m/2$.

After increasing $m_1$, Lemma~\ref{lem:parameter-dependent-nonperiodic-slice}, applied with $\rho=m^{a-1}$, gives uniformly in the resulting target
\[
\Pr[\mathcal C([s,b])]\leq\frac1N+C_D\frac{m^{(1-a)/2}\sqrt{\log(2+m)}}{m\sqrt{r_\ell}}.
\]
As $\ell$ ranges over multiples of $d$, the values with $\ell\leq m/2$ lie in the residue class $0$ modulo $d$, while the complementary sizes $m-\ell$ arising from $\ell>m/2$ lie in the residue class $m$ modulo $d$. Thus the values of $r_\ell$ lie in at most two residue classes modulo $d$. The residue-restricted summation used in the proof of Lemma~\ref{lem:layered-targets}(i) therefore gives
\[
\sum_{\substack{1\leq s\leq b\\ d\mid b-s+1}}\Pr[\mathcal C([s,b])]\leq C_D\vartheta_{d,a}(m,N).
\]
Here the full interval contributes zero, and the $O(1)$ values with $r_\ell<d$ are bounded by exposing all but one entry of the smaller set, giving a total contribution $O_D(1/m)$.

The target obstruction equations with $b\in\partial$ are bounded by the larger sum in Lemma~\ref{lem:layered-targets}(i). Since $|\partial|\leq8\Lambda+2=O_D(1)$, summing the clock-collision and target estimates over $b\in\partial$ changes only the constant. Hence
\begin{equation}
\label{eq:E1-clock-bound}
        \Pr[(\mathrm{E1})]\leq C_D\Theta.
\end{equation}
}

Suppose that~(E2) occurs outside~(E1).  Choose one obstruction equation
ending at each bad endpoint in the interval appearing in the definition
of~(E2).  Each selected obstruction equation is either a clock-collision equation
or a target equation, and the interval contains more than $2D$ bad
endpoints.  Hence one may retain $D$ distinct bad endpoints whose
selected obstruction equations are all clock-collision equations or all
target equations.  Since~(E1) does not occur, the selected endpoints
are at distance more than $4\Lambda$ from both ends of $[m]$.  Their
right endpoints lie in an interval of length $2\Lambda+1\leq\Lambda_D$;
hence they are contained in a $D$-local window $J=[j_-,j_+]$ satisfying
$1<j_-$ and $j_+<m$.
Each selected obstruction equation is then $J$-admissible: its right cut
point is local, while its left cut point, when present, is either local or
left-remote.  For every $J$-reduction class other than the $\mathsf{LC}$ class, choose
the equation with the smallest right endpoint as representative; mark each equation of type $\mathsf{LC}$ and
each non-representative equation by its own right endpoint.  The marked
row has coefficient $\pm1$ at that endpoint and coefficient zero at every
larger marked endpoint.  Thus decreasing right endpoints give a
right-to-left marking, and the selected equations form a right-to-left
$D$-bounded obstruction system of size $D$.  If the selected equations
are clock-collision equations, their profile is left-clock; if they are
target equations, their profile is anchored-target.  In either case it
is a local-repair profile.  After summing over the $O_D(1)$ possible
complete $J$-local profiles,
Lemma~\ref{lem:layered-targets}(iii) gives
\[
        \Pr[(\mathrm{E2})\setminus(\mathrm{E1})]
        \leq C_Dm\Theta^D.
\]
Together with \eqref{eq:E1-clock-bound},
\begin{equation}
\label{eq:E2-clock-bound}
        \Pr[(\mathrm{E2})]
        \leq C_D\bigl(\Theta+m\Theta^D\bigr).
\end{equation}

Suppose that~(E3) occurs while neither~(E0) nor~(E1) occurs.  Then
$b\notin\partial$.  For each blocked available candidate, choose one
blocker equation.  Separate the chosen equations into the following three
classes: $b$-oriented clock-collision equations, $b$-oriented target
equations, and $\RepairPos$-oriented clock-collision equations.  These classes are
exhaustive because every $\RepairPos$-oriented blocker is a clock-collision
equation.  Since there are more than $4D$ blocked available candidates,
one class contains at least $D$ of them.  Retain $D$ candidates from that
class.

Put $J=[b,b+3\Lambda]$.  Since
$|J|=3\Lambda+1\leq\Lambda_D$, this is a $D$-local window.  If the
retained equations are
$\RepairPos$-oriented, then their interval supports contain the corresponding
candidates in $(b,b+\Lambda]$ but do not contain $b$.  Their left cut
points therefore lie in
$[b,b+\Lambda-1]\subseteq J^-$.  Since (E0) does not occur,
their right cut points are greater than $b+3\Lambda$ and hence are
$J$-remote.  Thus they satisfy~\textup{(B2)} and are $J$-admissible.

If the retained equations are $b$-oriented, their supports contain $b$
but not their corresponding candidates.  Hence each right endpoint is
smaller than its candidate and therefore belongs to
$[b,b+\Lambda-1]\subseteq J$.  Any left cut point is either $J$-local
or lies to the left of $J$.  Thus these equations are $J$-admissible,
and condition~\textup{(B1)} holds by their selection.  Finally, since
(E1) does not occur, the definition of $\partial$ gives
$1<b$ and $b+3\Lambda<m$.  The retained equations therefore form blocker
data to which Lemma~\ref{lem:blocker-systems} applies.

{
Fix $b$, an ordered candidate tuple $\mathbf z$, an orientation class, and the supports of the retained obstruction equations. By the restriction argument following \eqref{eq:support-sum-restriction}, the corresponding event depends on the admissible permutation $\pi$ only through its restricted permutation $\pi_{\mathcal I}$, and there are at most $C_D$ possible such restrictions. For each fixed restriction, the map
\[
{\sigma\longmapsto\sigma^{\pi_{\mathcal I}}}
\]
is a bijection of $\operatorname{Ord}(X)$. Hence \eqref{eq:permutation-absorption} shows that replacing ${\sigma}$ by ${\sigma^{\pi_{\mathcal I}}}$ does not change the probability of the associated profile event. Therefore the union over all admissible restrictions costs only a factor $C_D$. Lemma~\ref{lem:blocker-systems}, which already sums over all admissible numerical remote cut points $\mathbf r$ and target tuples $\mathbf y$, consequently bounds the total contribution for the fixed $b$, candidate tuple, and orientation by
\[
C_D\Theta^D.
\]
There are only $O_D(1)$ ordered candidate tuples and orientation classes for each fixed $b$. Summing over $b\in[m]$ and adding}
\eqref{eq:E0-clock-bound}--\eqref{eq:E1-clock-bound} gives
\begin{equation}
\label{eq:E3-clock-bound}
        \Pr[(\mathrm{E3})]
        \leq C_D\bigl(\Theta+m\Theta^D\bigr).
\end{equation}

By \eqref{eq:clock-repair-smallness} and $D>3/\rho$,
\[
        \Theta+m\Theta^D
        \leq m^{-\rho}+m^{1-\rho D}=o(1).
\]
{
After increasing $m_1(a,\rho,L)$ if necessary, the union of
(E0)--(E3) has probability less than one.  Fix an ordering $\sigma$
for which none of these events occurs.  In particular,
\textup{(E1)}--\textup{(E3)} do not occur, so the hypotheses
\textup{(R1)}--\textup{(R3)} of
Lemma~\ref{lem:deterministic-layered-repair} hold.  The lemma therefore
produces an ordering $\sigma_0$ for which no obstruction equation holds.

At this point all probabilistic and repair bookkeeping is finished.  It
remains only to translate the absence of the two obstruction types back
into the two conclusions of Lemma~\ref{lem:layered-local-repair}.
}

Finally, if two tail vertices with indices $0\leq i<j\leq m$ coincide,
then $i\equiv j\pmod d$ and the obstruction equation
$\mathcal C([i+1,j])$ holds; the case $i=0$ is included.  Conversely,
if $\mathcal C(I)$ holds for some interval ${I=[s,t]}$, then the two tail
vertices with endpoints ${s-1}$ and ${t}$ coincide.  Hence no
clock-collision obstruction equation holds if and only if the
coordinates $(Q_j,\operatorname{res}_d(j))$, including the initial coordinate at
$j=0$, are pairwise distinct.  Likewise, no target obstruction equation
$\mathcal T(b,y)$ holds if and only if
$Q_b\notin W_{\operatorname{res}_d(b)}$ for every $b$.  This proves the lemma.
\end{proof}

\section{\texorpdfstring{{Concluding remarks}}{Concluding remarks}}
\label{sec:concluding-remarks}

{
The layered local-repair and reverse-absorption framework developed here
appears flexible enough to extend beyond cyclic groups.  A natural next
case is the dihedral group $D_{2p}$ of order $2p$.  The small-set range needed for
such a program is already available from the semidirect-product methods
of Costa, Della Fiore, and Engel
\cite{CostaDellaFioreEngelSemidirect}, while the large-set theorem
applies to arbitrary finite groups \cite{BBKMM}.  This leaves an
intermediate-range problem in which one would need to adapt the
Kneserized anticoncentration and layered local-repair mechanism to the
two cosets and to non-commutative partial products.  The deterministic
clock and local/remote bookkeeping of Section~7 suggest a plausible
route, although carrying it out will require a separate argument.
}

\subsection*{Statement on the use of generative AI}

During the preparation of this manuscript, the authors used ChatGPT
(OpenAI) for language editing and organization.  It also played a
substantial role as an interactive tool in developing and formalizing
the proof of the layered local-repair lemma, in particular in organizing
a version compatible with the bounded conditioning and local
permutations arising during the recursive descent.  The authors
independently checked, revised, and verified every mathematical statement
and proof, and retain full responsibility for the content.

\end{document}